\documentclass[a4paper,reqno,12pt]{amsart}
\usepackage{mathrsfs}
\usepackage{stmaryrd}
\usepackage{cases}
\usepackage{epic}
\usepackage{amsfonts}
\usepackage{graphicx}
\usepackage{amsmath}
\usepackage{amssymb, upgreek}
\usepackage{mathtools}
\usepackage{bm}
\usepackage{array}
\usepackage{chngcntr} 
\counterwithin{figure}{section}
\counterwithin{table}{section}
\usepackage{latexsym,todonotes}
\usepackage{pdflscape}
\usepackage[all]{xypic}
\usepackage[all]{xy}
\usepackage{color}
\usepackage{colordvi}
\usepackage{multicol}
\usepackage{tikz}
\usepackage{tikz-cd}
\usepackage[linktocpage=true]{hyperref}
\hypersetup{colorlinks,linkcolor=blue,urlcolor=cyan,citecolor=blue}
\newcommand{\bvs}{\mathbf{\varsigma}}
\newcommand{\vs}{\varsigma}
\allowdisplaybreaks

\DeclareMathAlphabet{\mathbbb}{U}{bbold}{m}{n}

\begin{document}
	\input xy
	\xyoption{all}

	\newtheorem{innercustomthm}{{\bf Theorem}}
	\newenvironment{customthm}[1]
	{\renewcommand\theinnercustomthm{#1}\innercustomthm}
	{\endinnercustomthm}
	
	\newtheorem{innercustomcor}{{\bf Corollary}}
	\newenvironment{customcor}[1]
	{\renewcommand\theinnercustomcor{#1}\innercustomcor}
	{\endinnercustomthm}
	
	\newtheorem{innercustomprop}{{\bf Proposition}}
	\newenvironment{customprop}[1]
	{\renewcommand\theinnercustomprop{#1}\innercustomprop}
	{\endinnercustomthm}
	\def \wt{\mathrm{wt}}
	\def \wItau{\I_{\circ,\tau}}
	\def \brW{\mathrm{Br}(W_\btau)}
    \def \br{\mathbf{r}}
	\newcommand{\iadd}{\operatorname{iadd}\nolimits}
	\newcommand{\Gr}{\operatorname{Gr}\nolimits}
	\newcommand{\FGS}{\operatorname{FGS}\nolimits}
	\def \tbU{\tU_{\bullet}}
\def \bU{\U_{\bullet}}
	\renewcommand{\mod}{\operatorname{mod^{\rm nil}}\nolimits}
	\newcommand{\proj}{\operatorname{proj}\nolimits}
	\newcommand{\inj}{\operatorname{inj.}\nolimits}
	\newcommand{\rad}{\operatorname{rad}\nolimits}
	\newcommand{\Span}{\operatorname{Span}\nolimits}
	\newcommand{\soc}{\operatorname{soc}\nolimits}
	\newcommand{\ind}{\operatorname{inj.dim}\nolimits}
	\newcommand{\Ginj}{\operatorname{Ginj}\nolimits}
	\newcommand{\res}{\operatorname{res}\nolimits}
	\newcommand{\np}{\operatorname{np}\nolimits}
	\newcommand{\Fac}{\operatorname{Fac}\nolimits}
	\newcommand{\Aut}{\operatorname{Aut}\nolimits}
	\newcommand{\DTr}{\operatorname{DTr}\nolimits}
	\newcommand{\TrD}{\operatorname{TrD}\nolimits}
	\newcommand{\Dr}{\operatorname{Dr}\nolimits}
	\newcommand{\Mod}{\operatorname{Mod}\nolimits}
	\newcommand{\R}{\operatorname{R}\nolimits}
	\newcommand{\End}{\operatorname{End}\nolimits}
	\newcommand{\lf}{\operatorname{l.f.}\nolimits}
	\newcommand{\Iso}{\operatorname{Iso}\nolimits}
	\newcommand{\aut}{\operatorname{Aut}\nolimits}
	\newcommand{\Ui}{{\mathbf U}^\imath}	\newcommand{\UU}{{\mathbf U}\otimes {\mathbf U}}
	\newcommand{\UUi}{(\UU)^\imath}
	\newcommand{\tUU}{{\tU}\otimes {\tU}}
	\newcommand{\tUUi}{(\tUU)^\imath}
	\newcommand{\tUi}{\widetilde{{\mathbf U}}^\imath}
	\newcommand{\sqq}{{\bf v}}
	\newcommand{\sqvs}{\sqrt{\vs}}
	\newcommand{\dbl}{\operatorname{dbl}\nolimits}
	\newcommand{\swa}{\operatorname{swap}\nolimits}
	\newcommand{\Gp}{\operatorname{Gp}\nolimits}
\def \tTD{\widetilde{T}}
    \def \fX{\Upsilon}
	\newcommand{\U}{{\mathbf U}}
	\newcommand{\tU}{\widetilde{\mathbf U}}
	\newcommand{\fgm}{{\rm mod}^{{\rm fg}}}
	\newcommand{\fgmz}{\mathrm{mod}^{{\rm fg},\Z}}
	\newcommand{\fdmz}{\mathrm{mod}^{{\rm nil},\Z}}
	\def \Sym{\mathrm{Sym}}
	\newcommand{\ov}{\overline}
	\newcommand{\und}{\underline}
	\newcommand{\tk}{\widetilde{k}}
	\newcommand{\tK}{K}
	\newcommand{\tTT}{\operatorname{\widetilde{\mathbf{T}}}\nolimits}
	\def \de{\delta}
\def \hU{\widehat{\U}}
\def \hUi{\widehat{\U}^\imath}
\def \bfk{\mathbf{k}}
\def \wI{\I_\circ}
    \def \cR{\mathcal{R}}
	 \newcommand{\tUD}{\operatorname{{}^{\textup{\texttt{Dr}}}\widetilde{\bU}}\nolimits}
     \newcommand{\tUDT}{\operatorname{{}^{\textup{\texttt{Dr}}}\widetilde{\bU}^\texttt{T}}\nolimits}
     	 \newcommand{\tUiD}{{}^{\textup{\texttt{Dr}}}\widetilde{\bU}^\imath}

      \newcommand{\BD}{\operatorname{{}^{\textup{\texttt{Dr}}}\widetilde{\mathbf{B}}}\nolimits}
	\newcommand{\tH}{\operatorname{{\ch}_{\rm{tw}}}\nolimits}
	
	\newcommand{\utM}{\operatorname{\cm\ch}\nolimits}
	\newcommand{\tM}{\operatorname{\cs\cd\widetilde{\ch}}\nolimits}
	\newcommand{\rM}{\operatorname{\cm\ch_{\rm{red}}}\nolimits}
	\newcommand{\utMH}{\cs\cd\ch(\Lambda^\imath)}
	\newcommand{\tMH}{\cs\cd\widetilde{\ch}(\Lambda^\imath)}
	\newcommand{\tCMH}{{\cc\widetilde{\ch}(\K Q,\btau)}}
	
	\newcommand{\rMH}{\operatorname{\cs\cd\ch_{\rm{red}}(\Lambda^\imath)}\nolimits}
	\newcommand{\utMHg}{\operatorname{\ch(Q,\btau)}\nolimits}
	\newcommand{\tMHg}{\operatorname{\widetilde{\ch}(Q,\btau)}\nolimits}
	\newcommand{\tMHk}{{\widetilde{\ch}(\K Q,\btau)}}
	\newcommand{\rMHg}{\operatorname{\ch_{\rm{red}}(Q,\btau)}\nolimits}
	
	\newcommand{\rMHd}{\operatorname{\cm\ch_{\rm{red}}(\Lambda^\imath)_{\bvsd}}\nolimits}
	\newcommand{\tMHd}{\operatorname{\cs\cd\widetilde{\ch}(\Lambda^\imath)_{\bvsd}}\nolimits}

   \def \h{h} 
	\newcommand{\tMHl}{\cs\cd\widetilde{\ch}({\bs}_\ell\Lambda^\imath)}
	\newcommand{\rMHl}{\cm\ch_{\rm{red}}({\bs}_\ell\Lambda^\imath)_{\bvsd}}
	\newcommand{\tMHi}{\cs\cd\widetilde{\ch}({\bs}_i\Lambda^\imath)}
	\newcommand{\rMHi}{\cm\ch_{\rm{red}}({\bs}_i\Lambda^\imath)_{\bvsd}}
	\newcommand{\tMHgi}{\widetilde{\ch}({\bs}_i Q,\btau)}
	\def \bI{\I_\bullet}
	\newcommand{\utGpg}{\operatorname{\ch^{\rm Gp}(Q,\btau)}\nolimits}
	\newcommand{\tGpg}{\operatorname{\widetilde{\ch}^{\rm Gp}(Q,\btau)}\nolimits}
	\newcommand{\rGpg}{\operatorname{\ch_{red}^{\rm Gp}(Q,\btau)}\nolimits}
	\def \tB{\widetilde{\mathbf{B}}}

	\newcommand{\colim}{\operatorname{colim}\nolimits}
	\newcommand{\gldim}{\operatorname{gl.dim}\nolimits}
	\newcommand{\cone}{\operatorname{cone}\nolimits}
	\newcommand{\rep}{\operatorname{rep}\nolimits}
	\newcommand{\Ext}{\operatorname{Ext}\nolimits}
	\newcommand{\Tor}{\operatorname{Tor}\nolimits}
	\newcommand{\Hom}{\operatorname{Hom}\nolimits}
	\newcommand{\Top}{\operatorname{top}\nolimits}
	\newcommand{\Coker}{\operatorname{Coker}\nolimits}
	\newcommand{\thick}{\operatorname{thick}\nolimits}
	\newcommand{\rank}{\operatorname{rank}\nolimits}
	\newcommand{\Gproj}{\operatorname{Gproj}\nolimits}
	\newcommand{\Len}{\operatorname{Length}\nolimits}
	\newcommand{\RHom}{\operatorname{RHom}\nolimits}
	\renewcommand{\deg}{\operatorname{deg}\nolimits}
	\renewcommand{\Im}{\operatorname{Im}\nolimits}
	\newcommand{\Ker}{\operatorname{Ker}\nolimits}
	\newcommand{\Coh}{\operatorname{Coh}\nolimits}
	\newcommand{\Id}{\operatorname{Id}\nolimits}
	\newcommand{\Qcoh}{\operatorname{Qch}\nolimits}
	\newcommand{\CM}{\operatorname{CM}\nolimits}
	\newcommand{\sgn}{\operatorname{sgn}\nolimits}
	\newcommand{\Gdim}{\operatorname{G.dim}\nolimits}
	\newcommand{\fpr}{\operatorname{\mathcal{P}^{\leq1}}\nolimits}
		\def \bTH { \boldsymbol{\Theta}}
	\def \bB{{\mathbf B }}

	\newcommand{\For}{\operatorname{{\bf F}or}\nolimits}
	\newcommand{\coker}{\operatorname{Coker}\nolimits}
	\renewcommand{\dim}{\operatorname{dim}\nolimits}
	\newcommand{\rankv}{\operatorname{\underline{rank}}\nolimits}
	\newcommand{\dimv}{{\operatorname{\underline{dim}}\nolimits}}
	\newcommand{\diag}{{\operatorname{diag}\nolimits}}
	\newcommand{\qbinom}[2]{\begin{bmatrix} #1\\#2 \end{bmatrix} }
	
	\renewcommand{\Vec}{{\operatorname{Vec}\nolimits}}
	\newcommand{\pd}{\operatorname{proj.dim}\nolimits}
	\newcommand{\gr}{\operatorname{gr}\nolimits}
	\newcommand{\id}{\operatorname{Id}\nolimits}
	\newcommand{\Res}{\operatorname{Res}\nolimits}
	\def \tT{\widetilde{\mathscr{T} }}

    \def \bwi{w_{\bullet,i}}
	\def \tTL{\tT(\Lambda^\imath)}
    \def \hB{\widehat{\mathbf{B}}}
	
	\newcommand{\mbf}{\mathbf}
	\newcommand{\mbb}{\mathbb}
	\newcommand{\mrm}{\mathrm}
	\newcommand{\cbinom}[2]{\left\{ \begin{matrix} #1\\#2 \end{matrix} \right\}}
	\newcommand{\dvev}[1]{{B_1|}_{\ev}^{{(#1)}}}
	\newcommand{\dv}[1]{{B_1|}_{\odd}^{{(#1)}}}
	\newcommand{\dvd}[1]{t_{\odd}^{{(#1)}}}
	\newcommand{\dvp}[1]{t_{\ev}^{{(#1)}}}
	\newcommand{\ev}{\bar{0}}
	\newcommand{\odd}{\bar{1}}
	\newcommand{\Iblack}{\I_{\bullet}}
	\newcommand{\wb}{w_\bullet}
	\newcommand{\Uidot}{\dot{\bold{U}}^{\imath}}
	
	\newcommand{\kk}{h}
	\newcommand{\la}{\lambda}
	\newcommand{\LR}[2]{\left\llbracket \begin{matrix} #1\\#2 \end{matrix} \right\rrbracket}
	\newcommand{\ff}{\mathbf{f}}
	\newcommand{\pdim}{\operatorname{proj.dim}\nolimits}
	\newcommand{\idim}{\operatorname{inj.dim}\nolimits}
	\newcommand{\Gd}{\operatorname{G.dim}\nolimits}
	\newcommand{\Ind}{\operatorname{Ind}\nolimits}
	\newcommand{\add}{\operatorname{add}\nolimits}
        \newcommand{\ad}{\operatorname{ad}\nolimits}
	\newcommand{\pr}{\operatorname{pr}\nolimits}
	\newcommand{\oR}{\operatorname{R}\nolimits}
	\newcommand{\oL}{\operatorname{L}\nolimits}
	\newcommand{\ext}{{ \mathfrak{Ext}}}
	\newcommand{\Perf}{{\mathfrak Perf}}
	\def\scrP{\mathscr{P}}
	\newcommand{\bk}{{\mathbb K}}
	\newcommand{\cc}{{\mathcal C}}
	\newcommand{\gc}{{\mathcal GC}}
	\newcommand{\dg}{{\rm dg}}
	\newcommand{\ce}{{\mathcal E}}
	\newcommand{\cs}{{\mathcal S}}
	\newcommand{\cl}{{\mathcal L}}
	\newcommand{\cf}{{\mathcal F}}
	\newcommand{\cx}{{\mathcal X}}
	\newcommand{\cy}{{\mathcal Y}}
	\newcommand{\ct}{{\mathcal T}}
	\newcommand{\cu}{{\mathcal U}}
	\newcommand{\cv}{{\mathcal V}}
	\newcommand{\cn}{{\mathcal N}}
	\newcommand{\mcr}{{\mathcal R}}
	\newcommand{\ch}{{\mathcal H}}
	\newcommand{\ca}{{\mathcal A}}
	\newcommand{\cb}{{\mathcal B}}
	\newcommand{\ci}{{\I}_{\btau}}
	\newcommand{\cj}{{\mathcal J}}
	\newcommand{\cm}{{\mathcal M}}
	\newcommand{\cp}{{\mathcal P}}
	\newcommand{\cg}{{\mathcal G}}
	\newcommand{\cw}{{\mathcal W}}
	\newcommand{\co}{{\mathcal O}}
	\newcommand{\cq}{{Q^{\rm dbl}}}
	\newcommand{\cd}{{\mathcal D}}
	\newcommand{\ck}{\widetilde{\mathcal K}}
	\newcommand{\calr}{{\mathcal R}}
	\newcommand{\iLa}{\Lambda^{\imath}}
	\newcommand{\La}{\Lambda}
	\newcommand{\ol}{\overline}
	\newcommand{\ul}{\underline}
	\newcommand{\st}{[1]}
	\newcommand{\ow}{\widetilde}
	\renewcommand{\P}{\mathbf{P}}
	\newcommand{\pic}{\operatorname{Pic}\nolimits}
	\newcommand{\Spec}{\operatorname{Spec}\nolimits}
	
	\newtheorem{theorem}{Theorem}[section]
	\newtheorem{acknowledgement}[theorem]{Acknowledgement}
	\newtheorem{algorithm}[theorem]{Algorithm}
	\newtheorem{assumption}[theorem]{Assumption}
	\newtheorem{axiom}[theorem]{Axiom}
	\newtheorem{case}[theorem]{Case}
	\newtheorem{claim}[theorem]{Claim}
	\newtheorem{conclusion}[theorem]{Conclusion}
	\newtheorem{condition}[theorem]{Condition}
	\newtheorem{conjecture}[theorem]{Conjecture}
	\newtheorem{construction}[theorem]{Construction}
	\newtheorem{corollary}[theorem]{Corollary}
	\newtheorem{criterion}[theorem]{Criterion}
	\newtheorem{definition}[theorem]{Definition}
	\newtheorem{example}[theorem]{Example}
	\newtheorem{exercise}[theorem]{Exercise}
	\newtheorem{lemma}[theorem]{Lemma}
	\newtheorem{notation}[theorem]{Notation}
	\newtheorem{problem}[theorem]{Problem}
	\newtheorem{proposition}[theorem]{Proposition}
	\newtheorem{solution}[theorem]{Solution}
	\newtheorem{summary}[theorem]{Summary}
	\numberwithin{equation}{section}
	
	\theoremstyle{remark}
	\newtheorem{remark}[theorem]{Remark}
	\newcommand{\Pd}{\pi_*}
	\def \bvs{{\boldsymbol{\varsigma}}}
	\def \bvsd{{\boldsymbol{\varsigma}_{\diamond}}}
	\def \btau{{{\tau}}}
	
	\def \Br{\mathrm{Br}}
	\def \bp{{\mathbf p}}
	\def \bq{{\bm q}}
	\def \bv{{v}}
	\def \bs{{r}}
    \def \ttC{\mathcal{C}}
	
	\def \bfK{{\mathbf K}}

    \def \bw{w_\bullet}
	
	\newcommand{\tCMHg}{\cc\widetilde{\ch}(Q,\btau)}
	\newcommand{\bfv}{\mathbf{v}}
	\def \bA{{\mathbf A}}
	\def \ba{{\mathbf a}}
        \def \bi{{\mathbf i}}
	\def \bL{{\mathbf L}}
	\def \bF{{\mathbf F}}
	\def \bS{{\mathbf S}}
	\def \bC{{\mathbf C}}
	\def \bU{{\mathbf U}}
	\def \bc{{\mathbf c}}
	\def \fpi{\mathfrak{P}^\imath}
	\def \Ni{N^\imath}
	\def \fp{\mathfrak{P}}
	\def \fg{\mathfrak{g}}
	\def \fk{\fg^\theta}  

\def\Inv{\mathrm{Inv}}
\def \bbW{W^{\circ}}
\def \bbw{{\boldsymbol{w}}}
\def \bb{\mathbf{B}}
   \def \reW{W^\circ} 
	\def \fn{\mathfrak{n}}
	\def \fh{\mathfrak{h}}
	\def \fu{\mathfrak{u}}
	\def \fv{\mathfrak{v}}
	\def \fa{\mathfrak{a}}
	\def \fq{\mathfrak{q}}
	\def \Z{{\Bbb Z}}
	\def \F{{\Bbb F}}
	\def \D{{\Bbb D}}
	\def \C{{\Bbb C}}
	\def \N{{\Bbb N}}
	\def \Q{{\Bbb Q}}
	\def \G{{\Bbb G}}
	\def \P{{\Bbb P}}
	\def \K{{\mathbb K}}
	\def \bK{{\Bbb K}}
	\def \J{{\Bbb J}}
	\def \E{{\Bbb E}}
	\def \A{{\Bbb A}}
	\def \L{{\Bbb L}}
	\def \I{{\Bbb I}}
	\def \H{{\Bbb H}}
	\def \T{{\Bbb T}}
    \def \upLa {\raisebox{8pt}{\rotatebox{180}{$\Lambda$}}}
	\newcommand{\TT}
{\operatorname{\mathbf{T}}\nolimits}
	\newcommand {\lu}[1]{\textcolor{red}{$\clubsuit$: #1}}
\def\SS{\mathbb{S}}
    \def \II{\I_0}
    \def \sym{\mathrm{Sym}}
    \def \Lg{L\fg}
	\newcommand{\nc}{\newcommand}
	\newcommand{\browntext}[1]{\textcolor{brown}{#1}}
	\newcommand{\greentext}[1]{\textcolor{green}{#1}}
	\newcommand{\redtext}[1]{\textcolor{red}{#1}}
	\newcommand{\bluetext}[1]{\textcolor{blue}{#1}}
	\newcommand{\brown}[1]{\browntext{ #1}}
	\newcommand{\green}[1]{\greentext{ #1}}
	\newcommand{\red}[1]{\redtext{ #1}}
	\newcommand{\blue}[1]{\bluetext{ #1}}
	
	\newcommand{\wtodo}{\todo[inline,color=orange!20, caption={}]}
	\newcommand{\lutodo}{\todo[inline,color=green!20, caption={}]}

	\title[Drinfeld presentation of affine $\mathrm i$quantum groups]{Drinfeld presentation of affine $\mathrm i$quantum groups via $\mathrm i$Hopf algebras}

	\author[Ming Lu]{Ming Lu}
	\address{Department of Mathematics, Sichuan University, Chengdu 610064, P.R.China}
	\email{luming@scu.edu.cn}

\author[Xiaolong Pan]{Xiaolong Pan}
\address{Department of Mathematics, Sichuan University, Chengdu 610064, P.R.China}
\email{xiaolong\_pan@stu.scu.edu.cn}

	
\begin{abstract}
Affine quantum groups $\tU$ and affine iquantum groups $\tUi$ admit two complementary descriptions: a Serre-type presentation and a Drinfeld (current) presentation. While iquantum groups $\tUi$ are typically defined via the Serre presentation, their Drinfeld presentations are substantially more subtle and less well understood. In this paper, we give a conceptual derivation of the Drinfeld presentation of $\tUi$ starting directly from the Drinfeld presentation of $\tU$, working within the framework of iHopf algebras. This approach leads naturally to the definition of iquantum affinization associated with arbitrary Kac-Moody algebras $\fg$. We further show that, in this general setting, these iquantum affinizations can be realized as coideal subalgebras of quantum affinizations with respect to the Drinfeld coproduct. Finally, in the affine case, we describe the imaginary root vectors $\Theta_{i,r}$ of $\tUi$ explicitly in terms of the root vectors $\psi^\pm_{i,r}$ of $\tU$, thereby extending the results of Li and Prze\'zdziecki for certain affine types of $\tUi$. 
\end{abstract}
	
	\maketitle
	\setcounter{tocdepth}{1}
	\tableofcontents
	
	\section{Introduction}

    \subsection{Background}

Quantum groups $\U$, introduced independently by Drinfeld \cite{Dr87} and Jimbo \cite{Jim85}, are Hopf algebras associated to arbitrary Kac-Moody algebras $\fg$. As in the setting of Kac-Moody algebras, quantum groups are classified by Dynkin diagrams. Associated with Satake diagrams $(\I,\tau)$ (with only white nodes), Letzter \cite{Let99} introduced the (quasi-split) iquantum groups $\Ui$ as coideal subalgebras of $\U$, also known as the quantum symmetric pair coideal subalgebras. 
In this paper, we work with the Drinfeld double quantum group $\tU$, in which the Cartan subalgebra is doubled, together with the universal iquantum group $\tUi$ introduced in \cite{LW22a}. This framework provides the most natural setting for relative braid group symmetries; see \cite{KP11,LW22b,WZ23,Z23}. We know $\tUi$ is a right coideal subalgebra of $\tU$, and the pair $(\tU,\tUi)$ is called the (universal) quantum symmetric pair; see \cite{Let99,Ko14,LW22a}.

The affine quantum group $\tU_v(\fg)$ admits two realizations: the Drinfeld-Jimbo realization (also called the Serre presentation) and the Drinfeld new realization (also known as the Drinfeld or current presentation); see \cite{Dr88}. An explicit isomorphism between these two presentations was first stated by Drinfeld and later proved by Beck \cite{Be94} and Damiani \cite{Da12,Da15}, using Lusztig’s braid group actions \cite{Lus93}.

The Drinfeld presentation of affine quantum groups has played a central role in many developments in algebra, geometry, and category theory. This presentation has proven to be particularly crucial in the representation theory of affine quantum groups; see the survey \cite{CH10} for further references. Frenkel and Reshetikhin \cite{FR99} introduced the $q$-character morphism, which has since become a powerful tool for the study of representations.

The Drinfeld presentation extends to quantum groups of Kac-Moody type $\tU_v(\fg)$, giving rise to the so-called quantum affinization (with central charges in this paper), usually denoted by $\tUD_v(\widehat{\fg})$; see \cite{Jin98,GKV95}. In general, quantum affinizations are not quantum groups, except in the simplest case of affine quantum groups; rather, they are topological Hopf algebras (cf. \cite{G07,H07,Da24}). These algebras are of particular interest, especially the quantum toroidal algebras arising when $\fg$ is of affine type; see \cite{GKV95,VV96,N01,H05,H07}.

The Drinfeld presentations of the (universal) iquantum groups of split ADE type were initiated by the first author and Wang in \cite{LW21b}, and were subsequently completed through a series of works \cite{Z22,LWZ23,LWZ24,LPWZ25}, reflecting the inherent complexity of Satake diagrams. Here, “split” refers to the case $\tau=\mathrm{Id}$ in quasi-split Satake diagrams. The overall strategy is largely inspired by \cite{Be94,Da12,Da15}, with the relative braid group actions \cite{LW22b,WZ23} of $\tUi$ playing a crucial role in the construction. However, in these works, iquantum groups are treated as independent algebraic objects, and their embeddings into quantum groups are not taken into account. 

It is therefore natural to investigate the relationship between the Drinfeld presentation of affine quantum groups and iquantum groups. This question was first studied by Prze\'zdziecki \cite{P25} for split type A, and subsequently by Li and Prze\'zdziecki \cite{LiP25} for split type BCD, as well as quasi-split type AIII \cite{LiP26}. As an application, they introduced boundary $q$-characters for $\tUi$ and used these $q$-characters to study the representation theory of $\tUi$. Their approach is to expand the Drinfeld generators of $\tUi$ in terms of the Drinfeld generators of $\tU$ via the natural embedding. This method relies on extensive calculations involving explicit reduced expressions for fundamental weights in the extended affine Weyl groups of classical types.

The study of affine iquantum groups  is very relevant to quantum
integrable models with boundary conditions (cf. \cite{EKP07}) such as open spin chains of
XXZ type \cite{Sk88}, reflection equations \cite{MRS03} and boundary affine Toda field theories (cf. \cite{BB10}).

\subsection{Goal}

The goal of this paper is to derive the Drinfeld presentation of the affine iquantum groups $\tUi$ from that of the affine quantum groups $\tU$, and to extend the framework of iquantum groups to the setting of iquantum affinization.

Our approach is based on the theory of iHopf algebras initiated in \cite{CLPRWa}. An iHopf algebra is a new associative algebraic structure defined on a Hopf algebra equipped with a Hopf pairing. When applied to a Borel subalgebra, this structure yields a realization of the iquantum groups $\tUi$. Owing to the existence of two Borel subalgebras arising from the two presentations of affine quantum groups, we develop two corresponding constructions of iHopf algebras. These constructions allow us to derive the Drinfeld presentation of $\tUi$ and to express its Drinfeld generators explicitly in terms of those of $\tU$. Moreover, the second construction applies to arbitrary quantum groups of Kac-Moody type, leading to the definition of iquantum affinization.

\subsection{Main results}

\subsubsection{Drinfeld presentation of affine iquantum groups via iHopf algebras}

Let $C=(c_{ij})_{i,j\in\I}$ be an affine Cartan matrix. Let $\II=\I-\{0\}$, where $0$ is the affine node. Let $\tU=\tU_v(\widehat{\fg})$ be an affine quantum group. 
Its Drinfeld presentation $\tUD=\tUD_v(\widehat{\fg})$ is the $\Q(v^{\frac12})$-algebra generated by $K_i^{\pm1}$, $K_i'^{\pm1}$, $C^{\pm1}$, $C'^{\pm1}$, and root vectors $X_{i,r}^\pm$, $\psi_{i,m}^\pm$, for $r\in\Z$, $m\geq1$, $i\in\II$. We identify $\tU$ and $\tUD$ via the Drinfeld-Beck isomorphism $\widetilde{\Omega}:\tUD\rightarrow\tU$.

Given an involution $\tau\in\Aut(C)$, we always assume that $\tau$ fixes the affine node $0$. The iquantum group $\tUi$ is the subalgebra of $\tU$, depending on parameters 
$\vs_i\in\Q(v^{\frac12})^{\times}$ ($i\in\I$). 
Its Drinfeld presentation $\tUiD=\tUiD_v(\widehat{\fg})$ is the $\Q(v^{\frac12})$-algebra generated by $\K_i^{\pm1}$, $\ttC^{\pm1}$, and the root vectors $B_{i,l}$, $\Theta_{i,m}$, for $i\in\II$, $l\in\Z$, $m\geq1$. We identify $\tUi$ and $\tUiD$ via the  isomorphism $\widetilde{\Omega}^\imath:\tUiD\rightarrow\tUi$ given in \cite{LW21b,Z22,LWZ24,LPWZ25}.

Let $\widetilde{\mathbf{B}}$ be the Borel subalgebra of $\tU$ generated by $\theta_i:=E_i$, $h_i^{\pm1}:=K_i^{\pm1}$, $i\in\I$. Set $\H_i:=v_i^{\frac{1}{2}c_{i,\tau i}}h_i$.
Define in $\tB$ the root vectors
\begin{align*}
    \theta_{i,r}:=\begin{cases}
    X_{i,r}^+ & \text{ if }r\geq0
    \\
    -o(i)^ro(\tau i)^rv_iX^-_{\tau i,-r}& \text{ if }r<0;
    \end{cases}
\end{align*}
and $\psi_{i,m}:=\psi^+_{i,m}$; cf. \eqref{eq:real-root}. 
Let $\widetilde{\mathbf{B}}^\imath_\tau$ be the iHopf algebra of $\widetilde{\mathbf{B}}$ associated with $\tau$ and the Hopf pairing $\varphi(-,-)$ defined in \eqref{eq:hopf-pairing}. 
By the main result of \cite{CLPRWa}, there exists an algebra isomorphism 
$\widetilde{\Phi}^\imath:\tB^\imath_\tau\rightarrow\tUi$. 

\begin{customthm}{{\bf A}} 
[Theorem~\ref{thm:iDrinfeld-iHopf}]\label{thm:A}
    The isomorphism $\widetilde{\Phi}^\imath:\tB^\imath_\tau\to \tUi$ satisfies
\[\widetilde{\Phi}^\imath(\H_i)=\K_{\tau i},\quad \widetilde{\Phi}^\imath(\psi_{i,m})=\Theta_{i,m},\quad \widetilde{\Phi}^\imath(\theta_{i,r})=B_{i,r},\quad \widetilde{\Phi}^\imath(\mathcal{C})=\mathcal{C}.\]
\end{customthm}

By Theorem \ref{thm:A}, the Drinfeld generators of $\tUi$ coincide with those of $\tB$ (and hence of $\tU$). Moreover, the Drinfeld relations of $\tUi$ can be derived directly from those of the Borel subalgebra $\tB$ through a modification of its multiplication, which is precisely the iHopf algebra construction. This unified approach significantly simplifies the derivation of the Drinfeld presentation of $\tUi$, in contrast to the case-by-case treatments employed in earlier works \cite{LW21b,Z22,LWZ23,LWZ24,LPWZ25}. Although a rigorous formulation of the root vectors $B_{i,l}$ and $\Theta_{i,m}$ in $\tUi$ is essential, the explicit construction of $\Theta_{i,m}$ presents substantial difficulties in certain cases, as discussed in \cite{LW21b,LPWZ25}.

As an application of Theorem \ref{thm:A}, we express the imaginary root vectors $\Theta_{i,m}$ of $\tUi$ in terms of the elements $\psi^\pm_{i,\pm m}$ of $\tU$. The resulting formulas are most conveniently stated using generating functions; see \eqref{eq:GF-QG1}, \eqref{eq:GF-QG2}, \eqref{eq:Genfun}. Let ${^{\textup{\texttt{Dr}}}\tU^{\geq 0}_+}$ denote the subspace of ${^{\textup{\texttt{Dr}}}\tU^{\geq 0}}$ consisting of elements with nonzero Drinfeld grading.

\begin{customthm}{{\bf B}} 
[Theorem~\ref{thm:cmpatible-imagine}]\label{thm:B}
We have
\begin{align*}
\Theta_{i}(z)&-\frac{1-\chi(\mathcal{C})z^2C'^2}{1-v_i^{-c_{i,\tau i}}\chi(\mathcal{C})z^2C'^2}\psi_{\tau i}^+(\chi(\mathcal{C})zC')\psi_i^-(z^{-1})\in{^{\textup{\texttt{Dr}}}\tU^{\geq 0}_+}\llbracket z\rrbracket,
\\
\Delta(\Theta_{i}(z))&-\Theta_i(zC'_2)\otimes\psi_{\tau i}^+(\chi(\mathcal{C})z\mathcal{C}_1C'_2)\psi_i^-(z^{-1})\in\big(\tUD\otimes{^{\textup{\texttt{Dr}}}\tU^{\geq 0}_+}\big)\llbracket z\rrbracket,
\end{align*}
where $\chi(\ttC)$ is given in \eqref{eq:chi-central-split}--\eqref{eq:chi-central-qsplit}. 
\end{customthm}

The identities in Theorem~\ref{thm:B} first appear in \cite{P25} for split type A, and were subsequently established in \cite{LiP25,LiP26} for split type BCD, as well as quasi-split type AIII. The formulas presented here differ slightly from those in \cite{P25} and \cite{LiP25,LiP26}, since these works are formulated for the algebra $\tU_v(L\fg)$ (that is, without central charge), whereas we work in the setting of $\tU_v(\widehat{\fg})$. Using these identities, Li and Prze\'zdziecki constructed boundary $q$-characters for $\tUi$, and applied them to the study of representations of $\tUi$. It is remarkable that their boundary $q$-characters work for all (untwisted) affine iquantum groups based on Theorem \ref{thm:B}.

\subsubsection{iQuantum affinization}

Let $\fg$ be an arbitrary Kac-Moody algebra. The (untwisted) quantum loop algebra $\tUD=\tUD_v(\widehat{\fg})$ is a quantum affinization of $\tU(\fg)$, generalizing quantum affine algebras in their current presentation. By definition, $\tUD$ is generated by elements $K_i^{\pm1}$, $K_i'^{\pm1}$, $C^{\pm1}$, $C'^{\pm1}$, and root vectors $X_{i,r}^\pm$, $\psi_{i,m}^\pm$, where $r\in\Z$, $m\geq1$, $i\in\II$, subject to the same defining relations as those of quantum affine algebras. Throughout, we adopt the reduced form of the defining relations as in \cite{Da12,Da15}. It is known that $\tUD$ carries the structure of a topological Hopf algebra.

Given an involution $\tau \in \Aut(C)$, one can analogously introduce the iquantum loop algebra $\tUiD :=\tUiD_v(\widehat{\fg})$. This algebra is generated by $\K_i^{\pm1}$, $\ttC^{\pm1}$, and the root vectors $B_{i,l}$ and $\Theta_{i,m}$, where $i \in \II$, $l \in \Z$, and $m \ge 1$. It was first introduced in \cite{LR24} for the simply-laced case, arising from Hall algebras of weighted projective lines. In what follows, we use the framework of iHopf algebras to show that $\tUiD$ is a natural generalization of iquantum groups, and that the pair $(\tUD, \tUiD)$ provides a natural generalization of quantum symmetric pairs.

Let $\BD$ denote the Drinfeld Borel subalgebra of $\tUD$, generated by $h_i^{\pm1}:=K_i^{\pm1}$, $C^{\pm1}$, $\theta_{i,l}:=X_{i,l}^+$, and $\psi_{i,m}:=\psi^+_{i,m}$, for $i\in\II$, $l\in\Z$, and $m \ge 1$. Let $\BD^\imath_\tau$ be the iHopf algebra associated with the pair $(\BD,\tau)$ and the weak Hopf pairing $\langle-,-\rangle$ introduced in \cite{G07}.

\begin{customthm}{\bf C}[Theorem \ref{thm:iH=iQG-Dr}]

    \label{thm:C}
    We have an algebra isomorphism 
    \begin{align*}
        &\qquad\widetilde{\Xi}: {^{\textup{\texttt{Dr}}}\tB^\imath_\tau}\longrightarrow \tUiD,\\
        \H_i &\mapsto \K_{\tau i},\qquad C\mapsto \mathcal{C},
        \quad
        \theta_{i,r}\mapsto B_{i,r},\quad
        \psi_{i,m}\mapsto \Theta_{i,m},\quad \forall i\in\II, r\in\Z,m>0.
    \end{align*}
\end{customthm}

Since the coproduct and antipode act explicitly on the root vectors $\theta_{i,r}$ and $\psi_{i,m}$ in the topological Hopf algebra $\BD$, we are able to derive the full set of general defining relations for $\tUiD$. Nevertheless, the iSerre relations in the case $\tau i = i$ with $i \neq j$ are highly intricate; previous works \cite{LW21b,Z22} established these relations only for $c_{ij}=0,-1,-2$. By virtue of Theorem~\ref{thm:C}, we present here the complete general defining relations in the case $c_{ij} c_{ji} \le 3$.

The following theorem illustrates the relationship between $\tUiD$ and $\tUD$.

\begin{customthm}{\bf D}[Theorem \ref{thm:subalgebra-Dr}]
\label{thm:D}
    There exists an injective algebra homomorphism
    \begin{align}
        &\mathcal{L}^{\textup{\texttt{Dr}}}:{}^\textup{\texttt{Dr}}\tUi\longrightarrow {}^{\textup{\texttt{Dr}}}\tU^{\texttt{T}},
        \\
        &\K_i\mapsto v^{\frac{1}{2}(\alpha_i,\alpha_{\tau i})}K_iK_{\tau i}', \quad \ttC\mapsto CC'
        \\
        \label{eq:em-Thetai}
        &\Theta_i(z)\mapsto \frac{1-z^2C'^2}{1-v_i^{-c_{i,\tau i}}z^2C'^2}\psi_{\tau i}^+(zC')\psi_i^-(z^{-1}),\\
        &B_{i}(z)\mapsto X_i^-(z^{-1})+\frac{1-z^2C'^2}{1-v_i^{-c_{i,\tau i}}z^2C'^2}X_{\tau i}^+(zC')K_i'\psi_i^-(z^{-1}).
    \end{align}
\end{customthm}

By Theorem~\ref{thm:D}, we may view $\tUiD$ as a subalgebra of $\tUD^{\textup{\texttt{T}}}$ (the topological completion of $\tUD$). We further show that it is in fact a coideal subalgebra; see Proposition~\ref{prop:coideal}.

\subsection{Organization}
The paper is organized as follows.
In Section~\ref{sec:QG-iQG}, we review the basics of quantum groups and iquantum groups, including their Drinfeld presentations and braid group actions.
In Section~\ref{sec:iDrinfeld-affine}, we use the framework of iHopf algebras to derive the Drinfeld presentation of affine iquantum groups $\tUi$, and express the imaginary root vectors $\Theta_{i,m}$ of $\tUi$ in terms of the elements $\psi^\pm_{i,\pm m}$ of $\tU$.

In Section~\ref{sec:iQLA}, we introduce iquantum affinization associated with arbitrary Kac-Moody algebras, as a generalization of affine iquantum groups in the Drinfeld presentation. We then prove that these algebras are coideal subalgebras of the corresponding quantum affinization within the framework of iHopf algebras.

\subsection{Perspectives}
The method developed in this paper, using iHopf algebras to establish Drinfeld presentations, also applies to twisted affine types. In particular, the Drinfeld presentation of twisted Yangians can be obtained by degeneration from affine iquantum groups. This has previously been established case-by-case in \cite{KLWZ25a,KLWZ25b,LZ26} via extensive Gauss decomposition computations. We expect that the Drinfeld presentation of twisted Yangians can alternatively be derived directly from that of Yangians using the framework of iHopf algebras.

As a distinguished class of quantum affinizations, quantum toroidal algebras (also known as double affine quantum algebras) have attracted significant attention. For their Schur-Weyl duality with elliptic Cherednik algebras (elliptic double affine Hecke algebras), see \cite{VV96}; for geometric realizations, see \cite{GKV95,N01}; and for their representation theory, see \cite{H09}. For the iquantum toroidal algebras introduced in this paper, we conjecture that they likewise admit a Schur-Weyl duality with elliptic double affine Hecke algebras (cf. \cite{BW18a,FLLLWW20,AP23} for the Schur-Weyl duality between iquantum groups and Hecke algebras). Moreover, we expect that these algebras also admit geometric realizations, analogous to those for affine iquantum groups (cf. \cite{FLLLW20,SW24}). 
	
	\vspace{2mm}
\noindent{\bf Acknowledgments} 
We thank Jiayi Chen, Shiquan Ruan, Weiqiang Wang, and Weinan Zhang for their collaboration on related projects. We also thank Jian-Rong Li, Tomasz Prze\'zdziecki, Weiqiang Wang, and Weinan Zhang for their interest in our work and for their valuable comments on the manuscript, which was shared with them in February 2006.
ML was partially supported by the National Natural Science Foundation of China (No.~12671049, 12631002).

%
%

\section{Quantum groups and iquantum groups}
\label{sec:QG-iQG}

In this section, we review (affine) quantum groups and iquantum groups, their braid group actions and Drinfeld presentations.

\subsection{Affine Weyl and braid groups}

Let $\II=\{1,2,\dots,n\}$, and $\I=\{0,1,\dots,n\}$, where $n\geq1$. Let $C=(c_{ij})_{i,j\in \II}$ be the Cartan matrix of a simple Lie algebra $\fg$.
Let $D=\diag(d_i\mid i\in \I_0)$ with $d_i\in\Z_{>0}$ be the symmetrizer of $C$, i.e., $DC$ is symmetric. 
Let $\cR_0$ be the set of roots for $\fg$, and fix a set $\cR^+_0$ of positive roots  with simple roots $\alpha_i$ $(i\in \II)$. Denote by $\theta=\sum_{i\in\II}a_i\alpha_i$ the highest root of $\fg$. Let $Q=\bigoplus_{i\in\II}\Z\alpha_i$ be the root lattice of $\fg$ and $Q^\vee=\bigoplus_{i\in\II}\Z\alpha_i^\vee$ be the coroot lattice. Let $P=\bigoplus_{i\in\II}\Z\omega_i$ be the weight lattice, where $\omega_i$ are fundamental weights of $\fg$.  Note $\alpha_i =\sum_{j\in \II} c_{ji}\omega_j$. We define a bilinear pairing $\langle \cdot, \cdot \rangle : P\times Q^\vee \rightarrow \Z$ such that $\langle \omega_i, \alpha_j^\vee \rangle =\delta_{i,j}$, for $i,j \in \II$, and thus $\langle \alpha_i, \alpha_j^\vee \rangle = c_{ji}$.
	
Let $\widehat{\fg}$ be the (untwisted) affine Lie algebra with affine Cartan matrix denoted by $(c_{ij})_{i,j\in\I}$. Let $\alpha_i$ $(i\in \I)$ be the simple roots of $\widehat{\fg}$, and $\alpha_0=\de -\theta$, where $\de$ denotes the basic imaginary root. 
 Extend $Q$ to the affine root lattice $\widetilde{Q} := Q \oplus \Z\alpha_0$,
 which has a symmetric bilinear form given by $( \alpha_i\mid\alpha_j)=d_ic_{ij}$, $i, j \in\I$. 
The root system $\cR$  for $\widehat{\fg}$ is given by
\begin{align}
	\cR &=\{\pm (\beta + k \delta) \mid \beta \in \cR_0^+, k  \in \Z\}  \cup \{m \delta \mid m \in \Z\backslash \{0\} \}.
		\label{eq:roots}
\end{align}
and we let $\mathcal{R}^+$ (resp. $\mathcal{R}^-$) be the subset of positive (resp. negative) roots. For $\gamma =\sum_{i\in \I} n_i \alpha_i \in \N^\I$, the height $\text{ht} (\gamma)$ is defined as $\text{ht} (\gamma) =\sum_{i\in \I} n_i$.

The Weyl group $W_0$ of $\fg$ is generated by the simple reflections $s_i$, for $i \in \II$. They act on $P$ by $s_i(x)=x-\langle x, \alpha_i^\vee \rangle\alpha_i$ for $x\in P$. The extended affine Weyl group $\widetilde{W}:=W_0 \ltimes P$ contains the affine Weyl group $W:=W_0 \ltimes Q =\langle s_i \mid i \in \I \rangle$ as a subgroup; we denote
\[
t_\omega =(1, \omega) \in \widetilde W, \quad \text{ for } \omega \in P,
\]
so that $t_{\omega}t_{\omega'} =t_{\omega+\omega'}$. In particular, for $\omega\in P$, $j\in\II$, $t_\omega(\alpha_j)=\alpha_j-\langle \omega,\alpha_j^\vee\rangle \de$. We identify $P/Q$ with a finite group $\Omega$ of Dynkin diagram automorphisms, and so $\widetilde{W} =\Omega\cdot W$. There is a length function $\ell(\cdot)$ on $\widetilde{W}$ such that $\ell(s_i)=1$, for $i\in \I$, and each element in $\Omega$ has length 0.
	

\subsection{Quantum groups}

Let $v$ be an indeterminate. Let 
$$v_i=v^{d_i},\qquad \forall i\in\I.$$
For $A,B$ in a $\Q(v^{\frac12})$-algebra, we write $[A, B]=AB-BA$, and $[A,B]_q=AB-qBA$ for any $q\in\Q(v^{\frac12})^\times$. Denote, for $r\in\N,m \in \Z$,
\[
 [r]_i=\frac{v_i^r-v_i^{-r}}{v_i-v_i^{-1}},
 \quad
 [r]_i^!=\prod_{i=1}^r [i]_i, \quad \qbinom{m}{r}_i =\frac{[m]_i[m-1]_i\ldots [m-r+1]_i}{[r]_i^!}.
\]
The (Drinfeld double) quantum group $\tU := \tU_v(\fg)$ is defined to be the $\Q(v^{\frac12})$-algebra generated by $E_i,F_i$ ($i\in\I$), $\tK_\alpha ,\tK_\alpha'$ ($\alpha\in\widetilde{Q}$),  subject to the following relations: 
\begin{align}
[E_i,F_j]= \delta_{ij} \frac{\tK_i-\tK_i'}{v_i^{-1}-v_i},  &\qquad [\tK_\alpha,\tK_\beta]=[\tK_\alpha,\tK_\beta']  =[\tK_\alpha',\tK_\beta']=0,
\label{eq:KK}
\\
\tK_\alpha E_j=v^{(\alpha\mid \alpha_j)} E_j \tK_\alpha, & \qquad \tK_\alpha F_j=v^{-(\alpha\mid \alpha_j)}  F_j \tK_\alpha,
\label{eq:EK}
\\
\tK_\alpha' E_j=v^{-(\alpha\mid \alpha_j)} E_j \tK_\alpha', & \qquad \tK_\alpha' F_j=v^{(\alpha\mid \alpha_j)} F_j \tK_\alpha',
 \label{eq:K2}
\end{align}
and the quantum Serre relations (which we skip). The quantum group $\tU$ is defined for any Kac-Moody algebra $\fg$, in particular, for the affine Lie algebra $\widehat{\fg}$. 
Here $K_i=K_{\alpha_i}$, $K_i'=K_{\alpha_i}'$.

The Drinfeld-Jimbo quantum group $\bU$ can be viewed as the quotient algebra of $\tU$ modulo the ideal generated by $K_iK_i'-1$ ($i\in\I$); see \cite{Dr87}. 
Then $\tU$ and $\bU$ are $\Z^\I$-graded algebras by setting 
$$\deg E_i=\alpha_i, \quad\deg F_i=-\alpha_i,\quad \deg K_i=0=\deg K_i'.$$ 
Let $\tU_\mu$ (resp. $\bU_\mu$) be the homogeneous subspace of degree $\mu$.  Then $\tU=\oplus_{\mu\in\Z^\I} \tU_\mu$, $\bU=\oplus_{\mu\in\Z^\I} \bU_\mu$. Let $\U^+$ be the subalgebra of $\tU$ (and also $\U$) generated by $E_i$ $(i\in \I)$, and let $\U^-$ be the subalgebra generated by $F_i$ $(i\in\I)$. Let $\tU^0$ be the subalgebras of $\tU$ generated by $\tK_i^{\pm1}, \tK_i'^{\pm1}$ $(i\in \I)$, and $\U^0$ be the subalgebra of $\bU$ generated by $\tK_i^{\pm 1}$ $(i\in \I)$. Then the algebras $\tU$ and $\bU$ have triangular decompositions:
\begin{align*}
\widetilde{\bU} =\U^+\otimes \widetilde{\bU}^0\otimes\U^-,
\qquad
\bU &=\bU^+\otimes \bU^0\otimes\bU^-.
\end{align*}

The algebras $\tU$ (and $\U$) are Hopf algebras, with the coproduct $\Delta$, the counit $\varepsilon$ and the antipode $S$ defined by
\begin{align}
    \begin{split}
        &\Delta(E_i)=E_i\otimes 1+K_i\otimes E_i,\quad \Delta(F_i)=1\otimes F_i+F_i\otimes K_i',
        \\
        &\Delta(K_i)=K_i\otimes K_i,\quad \Delta(K_i')=K_i'\otimes K_i';
       \\
&\varepsilon(E_i)=0=\varepsilon(F_i),\quad \varepsilon(K_i)=1=\varepsilon(K_i'),
\\
&S(E_i)=-K_i^{-1}E_i,\quad S(F_i)=-F_iK_i'^{-1},\quad S(K_i)=K_i^{-1},\quad S(K_i')=K_i'^{-1}.
     \end{split}
\end{align}

\begin{lemma}[cf. \cite{Lus93,BG17a}]
\label{lem:anti-involut-QG}
\begin{enumerate}
\item There exists an anti-involution $\sigma$ on $\tU$ (and also $\U$) given by $\sigma(E_i)=E_i$, $\sigma(F_i)=F_i$, and $\sigma(K_i)=K_i'$, for $i\in\I$.
\item There exists an anti-involution (called the bar-involution) $u\mapsto \ov{u}$ on $\hU$ (and also $\tU$, $\U$) given by $\ov{v^{1/2}}=v^{-1/2}$, $\ov{E_i}=E_i$, $\ov{F_i}=F_i$, and $\ov{K_i}=K_i$, $\ov{K_i'}=K_i'$, for $i\in\I$. 
\end{enumerate}
\end{lemma}

\begin{lemma} \label{lem:twisting Psi}
Let $\F$ be the algebraic closure of $\Q(v^{\frac{1}{2}})$ and $\F^\times=\F\setminus\{0\}$. For scalars $\ba=(a_i)_{i\in\I}\in (\F^{\times})^\I$, we have an automorphism $\widetilde{\Psi}_{\ba}$ on the $\F$-algebra $\tU$ such that
\[
\widetilde{\Psi}_{\ba}:K_i\mapsto a_i^{\frac{1}{2}}K_i,\quad K'_i\mapsto a_i^{\frac{1}{2}}K'_i,\quad E_i\mapsto a_i^{\frac{1}{2}}E_i,\quad F_i\mapsto F_i.
\]
\end{lemma}

Let $\Br(W)$ be the braid group associated to the Weyl group $W$, generated by simple reflections $t_i$ ($i\in\I$). Lusztig introduced 4 variants of braid group symmetries on the quantum group $\U$; see \cite[\S37.1.3]{Lus93}. These braid group symmetries can be lifted to the Drinfeld double $\tU$; see, e.g., \cite[Propositions 6.20–6.21]{LW22b}, which are denoted by $\widetilde{T}_{i,e}',\widetilde{T}_{i,-e}''$, $e=\pm1$. Here we shall choose renormalization scalars such that the braid group actions of $\tU$ are bar invariant; cf. \cite[\S5]{BG17a}. Recall the divided powers 
\[
F_i^{(m)} =F_i^m/[m]^!_i, \quad E_i^{(m)} =E_i^m/[m]^!_i, \quad \text{ for } m\ge 1 \text{ and  } i\in \I.
\]

\begin{proposition} {\rm(cf. \cite[\S5]{BG17a})}
\label{prop:BG1U}
For $i\in \I$, there exist automorphisms $\widetilde{T}_i$ on $\tU$ such that
\begin{align*}
&\widetilde{T}_i(K_\mu)= K_{s_i(\mu)},
\qquad \widetilde{T}_i(K'_\mu)= K'_{s_i(\mu)},\;\;\forall \mu\in \widetilde{Q},\\
&\widetilde{T}_i(E_i)=v_i^{-1}F_i(K_i')^{-1},\qquad \widetilde{T}_i(F_i)=v_i K_i^{-1}E_i,\\
&\widetilde{T}_i(E_j)
=\sum_{r+s=-c_{ij}} (-1)^{r}v_i^{r+\frac{1}{2}c_{ij}}E_i^{(r)} E_j E_i^{(s)} \quad\forall j\neq i,
\\
&\widetilde{T}_i(F_j)
= \sum_{r+s=-c_{ij}} (-1)^{r}v_i^{r+\frac{1}{2}c_{ij}}F_{i}^{(r)} F_j F_i^{(s)} \quad\forall j\neq i.
\end{align*}
\end{proposition}

\begin{proposition}
\label{prop:BG2U}
For $i\in \I$, the automorphisms $\widetilde{T}_i^{-1}$ on $\tU$ given by $\widetilde{T}_i^{-1}=\sigma\circ\widetilde{T}_i\circ\sigma$ satisfy that
\begin{align*}
&\widetilde{T}_i^{-1}(K_\mu)= K_{s_i(\mu)},
\qquad \widetilde{T}_i^{-1}(K'_\mu)= K'_{s_i(\mu)},\;\;\forall \mu\in \widetilde{Q},\\
&\widetilde{T}_i^{-1}(E_i)=v_i^{-1} K_i^{-1}F_i,\qquad \widetilde{T}_i^{-1}(F_i)=v_i E_i(K_i')^{-1},\\
&\widetilde{T}_i^{-1}(E_j)
=\sum_{r+s=-c_{ij}} (-1)^{r}v_i^{r+\frac{1}{2}c_{ij}}E_i^{(s)}  E_j E_i^{(r)} \quad\forall j\neq i,
\\
&\widetilde{T}_i^{-1}(F_j)
= \sum_{r+s=-c_{ij}} (-1)^{r}v_i^{r+\frac{1}{2}c_{ij}}F_i^{(s)}  F_j F_i^{(r)} \quad\forall j\neq i.
\end{align*}
\end{proposition}

The $\widetilde{T}_i$'s satisfy the braid group relations and so 
$\widetilde{T}_w 
 := \widetilde{T}_{i_1}\cdots
\widetilde{T}_{i_r} \in \Aut(\tU)$ is well defined,
where $w = s_{i_1}\cdots s_{i_r}$ is any reduced expression of $w \in W$.

	\subsection{Drinfeld presentation of affine quantum groups}
	\label{subsec:Drsl2}
 Let us consider the affine quantum group $\tU_v(\widehat{\fg})$ in this subsection.  
The Drinfeld new presentation of  $\tU_v(\fg)$ is the $\Q(v^{\frac12})$-algebra $\tUD_v(\widehat{\fg})$ with generators $X_{i,r}^{\pm}$, $K_i^{\pm1}$, $K_i'^{\pm1}$, $C^{\pm 1}$, $C'^{\pm 1}$, $\psi_{i,\pm m}^{\pm}$, $i\in\I_0$, $r\in\Z$, $m\geq 1$,  subject to the following relations: 
\begin{align}
&C, C' \text{ are  central},\quad [K_i,K_j]  = [K_i,K_j']=[K_i',K_j']=0,
\\
     \label{Drge1a}&\psi_i^+(z)\psi_j^+(w)=\psi_j^+(w)\psi_i^+(z),\qquad \psi_i^-(z)\psi_j^-(w)=\psi_j^-(w)\psi_i^-(z),
    \\
    \label{Drge1b}
     &g_{ij}(zw^{-1}C)\psi_i^+(z)\psi^-_j(w)=g_{ij}(zw^{-1}C')\psi_j^-(w)\psi_i^+(z),
     \\
     \label{Drge2}
     &K_iX_j^\pm(z)=v_i^{\pm c_{ij}}X_j^\pm(z)K_i,\qquad K_i'X_j^\pm(z)=v_i^{\mp c_{ij}}X_j^\pm(z)K_i',
     \\
         \label{Drge2b}
     &K_i\psi_j^\pm(w)=\psi_j^\pm(w)K_i,\qquad K_i'\psi_j^\pm(w)=\psi_j^\pm(w)K_i',
     \\
     \label{Drge3a}
     & \psi_i^{\pm}(z)X_j^{\pm}(w)=g_{ij}((zw^{-1})^{\pm1})X_j^\pm(w)\psi_i^\pm(z),
     \\
     \label{Drge3b}
     & \psi_i^+(z)X_j^-(w)=g_{ij}(zw^{-1}C)^{-1}X_j^-(w)\psi^+_i(z),
     \\  \label{Drge3c}
     &\psi_i^-(z)X_j^+(w)=g_{ij}(wz^{-1}C')^{-1}X_j^+(w)\psi^-_i(z),
     \\
     \label{Drge4}
     &[X_{i}^+(z),X_j^-(w)]=\frac{\delta_{ij}}{v_i^{-1}-v_i}\Big(K_i\delta(C^{-1}wz^{-1})\psi_i^+(z)-K_i'\delta(C'w^{-1}z)\psi_i^-(w)\Big),
     \\
     \label{Drge5}
     &(w-zv_i^{ \pm c_{ij}})X_i^\pm(z)X_j^\pm(w)=(v_i^{\pm c_{ij}}w -z)X_j^\pm(w)X_i^\pm(z),
     \\
     \label{Drge6}&\sym_{z_1,\dots,z_r}\sum_{a=0}^{r} (-1)^a \qbinom{r}{a}_i
		X_{i}^{\pm}(z_1)\cdots
		X_{i}^{\pm}(z_a) X_{j}^{\pm}(w)  X_{i}^{\pm}(z_{a+1}) \cdots X_{i}^{\pm}(z_r) =0, \\\notag
        &\hspace{7cm}\text{ for } i\neq j, r=1-c_{ij}.
\end{align}

where $\sym_{k_1,\dots,k_{1-c_{ij}}}$ denotes the symmetrization with respect to the indices $k_1,\dots,k_{1-c_{ij}}$; 
\begin{align}
\label{eq:GF-QG1}
        &X^{\pm}_i(w)=\sum_{r\in\Z} X_{i,r}^\pm w^{r},
        \qquad
        \psi_i^+(u)=1+\sum_{m\geq 1}\psi_{i,m}^+u^m,
        \\
        \label{eq:GF-QG2}
        \psi_i^-(u)=&1+\sum_{m\geq 1}\psi_{i,-m}^-u^{-m},
        \qquad
        \delta(z)=\sum_{k\in\Z}z^k,\qquad g_{ij}(z)=\frac{1-v_i^{-c_{ij}}z}{1-v_i^{c_{ij}}z}.
\end{align}
For convenience, we set $\psi_{i,0}^\pm=1$. The relation \eqref{Drge6} is called Serre relations. It is equivalent to the following relations, for all $k_1,\dots,k_{1-c_{ij}}\in\Z$,
\begin{align}
    \label{DR7}
		&\sym_{k_1,\dots,k_{1-c_{ij}}}\sum_{a=0}^{1-c_{ij}} (-1)^a \qbinom{1-c_{ij}}{a}_i
		X_{i,k_1}^{\pm}\cdots
		X_{i,k_a}^{\pm} X_{j,t}^{\pm}  X_{i,k_{a+1}}^{\pm} \cdots X_{i,k_{1-c_{ij}}}^{\pm} =0, \text{ for } i\neq j.
\end{align}

\begin{lemma}[cf. {\cite[\S11]{Da12}}]
    The Serre relation \eqref{DR7} is equivalent to
    \begin{align}
    \label{eq:reduced-Serre}
    \sum_{a=0}^{1-c_{ij}} (-1)^a \qbinom{1-c_{ij}}{a}_i
		X_{i,k}^{\pm}\cdots
		X_{i,k}^{\pm} X_{j,t}^{\pm}  X_{i,k}^{\pm} \cdots X_{i,k}^{\pm} =0,
\end{align}
for (fixed) $k,t\in\Z$, under the assumption of \eqref{Drge1a}--\eqref{Drge5}.
\end{lemma}

\begin{proof}
    We need the generators $h_{i,l}$ for $i\in\II$, $l\in\Z^\times$, defined by
    	\begin{align}
		\label{exp h+}
		1+ \sum_{m\geq 1}\psi^+_{i,m}u^m &=  \exp\Big((v_i -v_i^{-1}) \sum_{m\ge 1}  h_{i,m}u^m\Big),
		\\
		\label{exp h-}
		1+ \sum_{m\geq1 }\psi^-_{i, -m}u^{-m} &= \exp \Big((v_i^{-1}-v_i) \sum_{m\ge 1} h_{i,-m}u^{-m}\Big).
	\end{align}
    Then \eqref{Drge3a}--\eqref{Drge3c} are equivalent to
    \begin{align}
        \label{DR4}
		&[h_{i,l},X_{j,r }^{\pm}] =\begin{cases}\displaystyle\pm\frac{[lc_{ij}]_i}{l} C^{\frac{l\mp l}{2}} X_{j,r+l}^{\pm},&\text{ if }l>0,
        \\
\displaystyle\pm\frac{[lc_{ij}]_i}{l} C'^{-\frac{l \pm l}{2}} X_{j,r+l}^{\pm},&\text{ if }l<0.
		\end{cases}
    \end{align}

It is enough to prove that \eqref{eq:reduced-Serre} implies \eqref{DR7}. Without loss of generality, we assume $k=0=t$ for \eqref{eq:reduced-Serre}, and only prove for the plus version. 

Denote  
\begin{align*}\SS^\pm (k_1,\dots,k_{1-c_{ij}}\mid t):= &\sym_{k_1,\dots,k_{1-c_{ij}}}\sum_{a=0}^{1-c_{ij}} (-1)^a \qbinom{1-c_{ij}}{a}_i
		\\
        &X_{i,k_1}^{\pm}\cdots
		X_{i,k_a}^{\pm} X_{j,t}^{\pm}  X_{i,k_{a+1}}^{\pm} \cdots X_{i,k_{1-c_{ij}}}^{\pm}.
        \end{align*}
        It follows from our assumption that 
        $\SS^+(0,\dots,0\mid 0)=0$. 
By applying $\big[h_{i,l_1},- \big]$ to 
$\SS^+(0,\dots,0\mid 0)=0$ and using \eqref{DR4}, we get 
\begin{align}
\label{eq:Serrel1-1}
    \frac{[2l_1]_i}{l_1} \SS^+(l_1,\dots,0\mid 0)+\frac{[l_1c_{ij}]_i}{l_1} \SS^+(0,\dots,0\mid l_1)=0.
\end{align}
Similarly, applying $\big[h_{j,l_1},- \big]$ to 
$\SS^+(0,\dots,0\mid 0)=0$, we get
\begin{align}
\label{eq:Serrel1-2}
    \frac{[l_1c_{ji}]_{v_j}}{l_1}\SS^+(l_1,\dots,0\mid 0)+\frac{[2l_1]_{v_j}}{l_1} \SS^+(0,\dots,0\mid l_1)=0.
\end{align}
Combining \eqref{eq:Serrel1-1}--\eqref{eq:Serrel1-2}, we can see
\begin{align}
     \SS^+(l_1,\dots,0\mid 0)=0,\qquad \SS^+(0,\dots,0\mid l_1)=0,\quad \forall l_1\in\Z.
\end{align}

Performing the same operation on $\SS^+(l_1,\dots,0\mid 0)=0$ recursively, we can get 

    \begin{align}
    \label{eq:Serre-reduced-1} \SS^+(k_1,\dots,k_{1-c_{ij}}\mid 0)=0,\quad\forall k_1,\dots,k_{1-c_{ij}}\in\Z.
\end{align}
Then \eqref{DR7} is obtained 
by applying $\big[h_{i,t},- \big]$ to \eqref{eq:Serre-reduced-1}.
\end{proof}

For convenience, we denote 
\begin{align}
    \SS^{\pm}(z_1,\dots,z_r\mid w):=&\sym_{z_1,\dots,z_r}\sum_{a=0}^{r} (-1)^a \qbinom{r}{a}_i
    \\\notag
    &
		X_{i}^{\pm}(z_1)\cdots
		X_{i}^{\pm}(z_a) X_{j}^{\pm}(w)  X_{i}^{\pm}(z_{a+1}) \cdots X_{i}^{\pm}(z_r),
\end{align}
for $i\neq j\in\II$, $r=1-c_{ij}$.

 The quantum loop algebra ${}^\textup{\texttt{Dr}}\U_v(\widehat{\fg})$ can be described as a quotient algebra of $\tUD_v(\widehat{\fg})$ modulo the ideal generated by
	$$CC'-1,\qquad K_iK_i'-1, \,\,\forall i\in\II.$$
There exist isomorphisms of $\Q(v^{\frac12})$-algebras \cite{Dr88, Be94, Da15}
	\begin{align}
		\label{eq:phi1}
		\widetilde{\Omega}: {^{\textup{\texttt{Dr}}}\tU} \longrightarrow \tU,\qquad \Omega: {^{\textup{\texttt{Dr}}}\U} \longrightarrow \U,
	\end{align}
    with the root vectors constructed from Lusztig's braid group actions of $\tU$.

  Let ${^{\textup{\texttt{Dr}}}\tU^0}$ denote the $\Q(v^{\frac12})$-subalgebra generated by $\{C^{\pm1},C'^{\pm1},K_i^{\pm1},K_i'^{\pm1},\psi_{i,\pm m}^\pm\mid i\in\I_0,m\geq 1\}$, and ${^{\textup{\texttt{Dr}}}\tU^\pm}$ denote the $\Q(v^{\frac12})$-subalgebras generated respectively by $\{X_{i,r}^\pm\mid i\in\I_0,r\in\Z\}$.
Then ${^{\textup{\texttt{Dr}}}\tU}$ admits the triangular decomposition 
$${^{\textup{\texttt{Dr}}}\tU}\cong {^{\textup{\texttt{Dr}}}\tU^+}\otimes {^{\textup{\texttt{Dr}}}\tU^0}\otimes {^{\textup{\texttt{Dr}}}\tU^-}.$$
    
Note that we can also define a $\Z^{\I_0}$-grading on ${^{\textup{\texttt{Dr}}}\tU}$ by
\[\deg^{\textup{\texttt{Dr}}}(X_{i,r}^\pm)=\pm\alpha_i,\quad \deg^{\textup{\texttt{Dr}}}(h_{i,m})=\deg^{\textup{\texttt{Dr}}}(K_i)=\deg^{\textup{\texttt{Dr}}}(K_i')=\deg^{\textup{\texttt{Dr}}}(C)=\deg^{\textup{\texttt{Dr}}}(C')=0.\]

\subsection{iQuantum groups}

	For an (affine) Cartan matrix $C=(c_{ij})_{\I\times\I}$, let $\Aut(C)$ be the group of all permutations $\tau$ of the set $\I$ such that $c_{ij}=c_{\tau i,\tau j}$. 	Let $\tau$ be an involution in $\Aut(C)$, i.e., $\tau^2={\rm Id}$. Following \cite{LW22a} we define $\widetilde{\U}^\imath$ to be the $\Q(v^{\frac12})$-subalgebra of $\tU$ generated by
	\begin{equation}
		\label{eq:Bi}
		B_i= F_i+\varsigma_iE_{\tau i} \tK_i',
		\qquad \tk_i= \tK_i \tK_{\tau i}', \quad \forall i \in \I.
	\end{equation}
    where $\varsigma_i\in\Q(v^{\frac{1}{2}})^\times$ is a system of parameters satisfying $\varsigma_i=\varsigma_{\tau i}$.
    These iquantum groups are of quasi-split type, and they are classified by Satake diagrams $(\I,\tau)$ with only white nodes. The iquantum groups are also defined for all Kac-Moody algebras.

The split Satake diagrams (i.e., those with $\tau ={\rm Id}$) are identical to the usual Dynkin diagrams. In this paper, for  affine iquantum groups, we only consider the quasi-split ones with the affine node $0$ fixed. A complete list of such quasi-split affine Satake diagrams is given in Table \ref{tab:Satakediag} below, where we have used a superscript $(\tau)$ to denote this class of quasi-split affine Satake diagrams.
 \begin{table}[h]  
\begin{center}
\centering
\begin{tabular}{|m{4.5cm}<{\centering}|m{8.5cm}<{\centering}|}
\hline
Types &  Quasi-split affine Satake diagrams \\
\hline

${\rm AIII}_{2r}^{(\tau)}$ $(r\geq 1)$ & \setlength{\unitlength}{0.6mm}
 	\begin{picture}(70,35)(0,5)
				\put(0,10){$\circ$}
				\put(0,30){$\circ$}
				\put(50,10){$\circ$}
				\put(50,30){$\circ$}
				\put(72,10){$\circ$}
				\put(72,30){$\circ$}
				\put(73.5,13){\line(0,1){17}}
				\put(0,5.5){$2r$}
				\put(-2,34){${1}$}
				\put(47,6){\small $r+2$}
				\put(47,34){\small $r-1$}
				\put(69,6){\small $r+1$}
				\put(69,34){\small $r$}
				
				\put(3,11.5){\line(1,0){16}}
				\put(3,31.5){\line(1,0){16}}
				\put(23,10){$\cdots$}
				\put(23,30){$\cdots$}
				\put(33.5,11.5){\line(1,0){16}}
				\put(33.5,31.5){\line(1,0){16}}
				\put(53,11.5){\line(1,0){18.5}}
				\put(53,31.5){\line(1,0){18.5}}
				\put(-17.5,22.5){\line(2,1){17}}
				\put(-17.5,20.5){\line(2,-1){17}}
				
				\put(-21,20){$\circ$}
				\put(-21,15){\small $0$}
				
				\color{red}

                \qbezier(0,13.5)(-4,21.5)(0,29.5)
				\put(-0.25,14){\vector(1,-2){0.5}}
				\put(-0.25,29){\vector(1,2){0.5}}
				
				\qbezier(50,13.5)(46,21.5)(50,29.5)
				\put(49.75,14){\vector(1,-2){0.5}}
				\put(49.75,29){\vector(1,2){0.5}}

				\qbezier(75,13.5)(79,21.5)(75,29.5)
				\put(75.25,14){\vector(-1,-2){0.5}}
				\put(75.25,29){\vector(-1,2){0.5}}
			\end{picture}
\\
\hline
${\rm AIII}_{2r+1}^{(\tau)}$ $(r\geq 1)$ & \setlength{\unitlength}{0.6mm}			\begin{picture}(70,35)(0,5)

   
				\put(0,10){$\circ$}
				\put(0,30){$\circ$}
				\put(50,10){$\circ$}
				\put(50,30){$\circ$}
				\put(72,10){$\circ$}
				\put(72,30){$\circ$}
				\put(92,20){$\circ$}
				\put(-5,5.5){$2r+1$}
				\put(-2,34){${1}$}
				\put(47,6){\small $r+3$}
				\put(47,34){\small $r-1$}
				\put(69,6){\small $r+2$}
				\put(71,34){\small $r$}
				\put(92,16){\small $r+1$}
				
				\put(3,11.5){\line(1,0){16}}
				\put(3,31.5){\line(1,0){16}}
				\put(23,10){$\cdots$}
				\put(23,30){$\cdots$}
				\put(33.5,11.5){\line(1,0){16}}
				\put(33.5,31.5){\line(1,0){16}}
				\put(53,11.5){\line(1,0){18.5}}
				\put(53,31.5){\line(1,0){18.5}}
				\put(-17.5,22.5){\line(2,1){17}}
				\put(-17.5,20.5){\line(2,-1){17}}
				\put(75,12){\line(2,1){17}}
				\put(75,31){\line(2,-1){17}}
				\put(-21,20){$\circ$}
				\put(-21,15){\small $0$}
				
				\color{red}
				\qbezier(0,13.5)(-4,21.5)(0,29.5)
				\put(-0.25,14){\vector(1,-2){0.5}}
				\put(-0.25,29){\vector(1,2){0.5}}
				
				\qbezier(50,13.5)(46,21.5)(50,29.5)
				\put(49.75,14){\vector(1,-2){0.5}}
				\put(49.75,29){\vector(1,2){0.5}}

				\qbezier(72,13.5)(68,21.5)(72,29.5)
				\put(71.75,14){\vector(1,-2){0.5}}
				\put(71.75,29){\vector(1,2){0.5}}
			\end{picture}\\
\hline
${\rm DI}_r^{(\tau)}$ $(r\geq4)$ & \setlength{\unitlength}{0.6mm}	\begin{picture}(40,35)(20,-15)
				\put(0,9.5){$\circ$}
				\put(0,13.5){\small$1$}
				
				\put(3,9.5){\line(2,-1){16.5}}
				\put(3,-8.5){\line(2,1){16.5}}
				\put(0,-11){$\circ$}
				\put(0,-16){\small$0$}
				\put(20,-1){$\circ$}
				\put(20,-6){\small$2$}
				\put(64,-1){$\circ$}
				\put(56,-6){\small$r-2$}
				\put(84,-10){$\circ$}
				\put(80,-13){\small${r-1}$}
				\put(84,9.5){$\circ$}
				\put(84,13.5){\small${r}$}

				\put(38,0){\line(-1,0){15.5}}
				\put(64,0){\line(-1,0){15}}
				
				\put(40,-1){$\cdots$}
				
				\put(83.5,9.5){\line(-2,-1){16.5}}
				\put(83.5,-8.5){\line(-2,1){16.5}}
				
				\put(12,-20.5){\begin{picture}(100,40)	\color{red}
						\qbezier(75,13.5)(79,21.5)(75,29.5)
						\put(75.25,14){\vector(-1,-2){0.5}}
						\put(75.25,29){\vector(-1,2){0.5}}
				\end{picture}}
			\end{picture}
\\
\hline
${\rm EI}_6^{(\tau)}$  &	\setlength{\unitlength}{0.6mm}		\begin{picture}(70,35)(20,7)
				\put(97,36){\small${6}$}
				\put(75,36){\small${5}$}
				\put(97,6.5){\small${1}$}
				\put(75,6.5){\small${2}$}
				\put(11,16){\small $0$}
				\put(33,16){\small $4$}	
				\put(55,16){\small $3$}	\put(10,35){\rotatebox[origin=c]{180}{\begin{picture}(100,10)
							
							\put(10,10){$\circ$}
							
							\put(32,10){$\circ$}
							
							\put(10,30){$\circ$}
							
							\put(32,30){$\circ$}
							\put(96,20){$\circ$}

							\put(76.7,21.2){\line(1,0){19}}
							
							\put(31.5,11){\line(-1,0){19}}
							\put(31.5,31){\line(-1,0){19}}
							
							\put(52,22){\line(-2,1){17.5}}
							\put(52,20){\line(-2,-1){17.5}}
							
							\put(54.7,21.2){\line(1,0){19}}

							\put(52,20){$\circ$}
							
							\put(74,20){$\circ$}
					\end{picture}}

					\put(-37,-32.5){\begin{picture}(100,40)	\color{red}
							\qbezier(5,13.5)(9,21.5)(5,29.5)
							\put(5.25,14){\vector(-1,-2){0.5}}
							\put(5.25,29){\vector(-1,2){0.5}}
					\end{picture}}
					\put(-15,-32.5){\begin{picture}(100,40)	\color{red}
							\qbezier(5,13.5)(9,21.5)(5,29.5)
							\put(5.25,14){\vector(-1,-2){0.5}}
							\put(5.25,29){\vector(-1,2){0.5}}
					\end{picture}}
				}
				
			\end{picture}
\\\hline
\end{tabular} 
\end{center}
\vspace{0.5cm}
\caption{Quasi-split affine Satake diagrams with $\tau \neq {\rm Id}, \tau(0)=0$}
\label{tab:Satakediag}
\end{table}

    We shall need the following symmetries of $\tUi$.
\begin{lemma}  
\label{lem:involution-iQG}
\begin{enumerate}
    \item 
There exists an anti-involution $\sigma^\imath$ on $\tUi$ given by $\sigma^\imath(B_i)=B_i$, $\sigma^\imath(\tk_i)=\tk_{\tau i}$, for $i\in\I$;
\item There exists an anti-involution (called bar-involution) $:u\mapsto \ov{u}$ on $\tUi$ given by $\ov{v^{1/2}}=v^{-1/2}$, $\ov{B_i}=B_i$, and $\ov{\tk_i}=v_i^{c_{i,\tau i}}\tk_{i}$,  for $i\in\I$. 
\end{enumerate}
\end{lemma}

 We fix
\begin{align}   \label{eq:ci}
\I_\tau &= \{ \text{the chosen representatives of $\tau$-orbits in $\I$} \}.
\end{align}
Denote by $\bs_{i}$ the following element of order $2$ in the Weyl group $W$, i.e.,
	\begin{align}
		\label{def:simple reflection}
		\bs_i= \left\{
		\begin{array}{ll}
			s_{i}, & \text{ if } c_{i,\tau i}= 2 \; (\text{i.e.}, \tau i=i);
			\\
			s_is_{\tau i}, & \text{ if }  c_{i,\tau i}=0;
			\\
			s_is_{\tau i}s_i, &\text{ if }  c_{i,\tau i}=-1.
		\end{array}
		\right.
	\end{align}
	Note that $\bs_i=\bs_{\tau i}$ for any $i\in\I$ and hence we can parametrize the $\bs_i$ by $\I_\tau$.  
 Consider the following subgroup $W_\tau$ of the affine Weyl group $W$ associated with $\cR$:
	\begin{align}
		\label{eq:Wtau}
		W_{\tau} =\{w\in W\mid \tau w =w \tau\}.
	\end{align}
 The group $W_{\tau}$ is a Weyl group with $\bs_i$ ($i\in \I_\tau$) as its generating set (cf. \cite[\S2.3]{DK19}).

The braid group associated to the relative Weyl group for $(\fg, \fg^\theta)$ is of the form
	\begin{equation}
		\label{eq:braidCox}
		\brW =\langle r_i \mid i\in \I_\tau \rangle
	\end{equation}
	where $r_i$ satisfy the same braid relations as for $\bs_i$ in $W_{\tau}$ (but no quadratic relations on $r_i$ are imposed). The relative braid (or ibraid) group symmetries $\tTT_{i,e}'$ and $\tTT''_{i,e}$ ($i\in\I$, $e\in\{+1,-1\}$) on $\tUi$ are established in \cite{LW22a, WZ23} (and \cite{Z23}); see \cite{KP11} for earlier conjectures on iquantum groups with specific parameters. In this paper, we shall also use the bar-equivariant versions of these ibraid group symmetries of $\tUi$. 
    
    Quasi K-matrices appeared earlier in different formulations; see \cite{BW18a, BK19,AV22}. We shall need the following.

    \begin{proposition}
[\text{\cite[Theorem~3.6]{WZ23}}]
\label{prop:Kmatrix}
There exists a unique element $\widetilde{\Upsilon}=\sum_{\mu\in\N^\I}\widetilde{\Upsilon}^\mu$ (called quasi K-matrix) such that $\widetilde{\Upsilon}^0=1$, $\widetilde{\Upsilon}^\mu\in\U^+_{\mu}$ and the following identities hold:
\begin{align}
B_i\widetilde{\Upsilon}&=\widetilde{\Upsilon}B_i^\sigma,\qquad (i\in\I),
\\
x\widetilde{\Upsilon}&=\widetilde{\Upsilon}x,\qquad (x\in\tU^{\imath0}),
\end{align}
where $B_i^\sigma:=\sigma(B_i)=F_i+\varsigma_iK_iE_{\tau i}$.  Moreover, $\widetilde{\Upsilon}^\mu=0$ unless $\tau(\mu)=\mu$.
\end{proposition}

Denote by $\tU_{i,\tau i}$ the quantum group associated to $\I_i=\{i,\tau i\}$. Let $\widetilde{\Upsilon}_i$ be the rank one quasi K-matrix associated to $\I_i=\{i,\tau i\}$, i.e., $\widetilde{\Upsilon}_i=\sum_{\mu\in\N^\I}\widetilde{\Upsilon}_{i}^{\mu}$ with $\widetilde{\Upsilon}_{i}^{\mu}\in\tU_{i,\tau i}^+$, and $\widetilde{\Upsilon}_{i}^{0}=1$. Let us fix another parameter $\bm{\varsigma}_\diamond=(\varsigma_{i,\diamond})_{i\in\I}$ such that
\begin{align}
\label{eq:disting-para}
\varsigma_{i,\diamond}=\varsigma_{\tau i,\diamond},\quad \varsigma_i\varsigma_{i,\diamond}=v^{-\frac{1}{2}(\alpha_i,\alpha_{\tau i})}.
\end{align}
Recall the automorphism $\widetilde{\Psi}_{\diamond}:=\widetilde{\Psi}_{\bm{\varsigma}_\diamond}$ of $\tU$ from Lemma~\ref{lem:twisting Psi}. We set
\begin{equation}\label{eq: T_i twisted def}
\widetilde{\mathscr{T}}_i:=\widetilde{\Psi}_{\diamond}^{-1}\circ\widetilde{T}_i\circ \widetilde{\Psi}_{\diamond},
\qquad \widetilde{\mathscr{T}}_i^{-1}:=\widetilde{\Psi}_{\diamond}^{-1}\circ\widetilde{T}_i^{-1}\circ \widetilde{\Psi}_{\diamond}.
\end{equation}
Clearly $\widetilde{\mathscr{T}}_i$ and $\widetilde{\mathscr{T}}_i^{-1}$, for $i\in\I$, are automorphisms of $\tU$ and satisfy the braid group relations. Hence, we can define
$\widetilde{\mathscr{T}}_w:=\widetilde{\mathscr{T}}_{i_1}\cdots \widetilde{\mathscr{T}}_{i_r},$
where $w=s_{i_1}\cdots s_{i_r}$ is any reduced expression. 

\begin{theorem}[{cf. \cite{WZ23,Z23}}]
\label{thm: relative T_i conjugate}
For $i\in\I$, there are mutually inverse automorphisms $\widetilde{\TT}_i$ and $\widetilde{\TT}_i^{-1}$ on $\tUi$ such that
\begin{align}
\label{eq:relbraid1}
\widetilde{\TT}_i^{-1}(x)\widetilde{\Upsilon}_i &=\widetilde{\Upsilon}_i\widetilde{\mathscr{T}}_{r_i}^{-1}(x),
\\
\label{eq:relbraid2}
\widetilde{\TT}_i(x)\widetilde{\mathscr{T}}_i^{-1}(\widetilde{\Upsilon}_i)^{-1} &=\widetilde{\mathscr{T}}_i^{-1}(\widetilde{\Upsilon}_i)^{-1}\widetilde{\mathscr{T}}_{r_i}(x).
\end{align}
Moreover, we have $\widetilde{\TT}_i^{-1}=\sigma^\imath\circ \widetilde{\TT}_i\circ \sigma^\imath$, and there exists a group homomorphism $\Br(W_\tau)\rightarrow \Aut(\tUi)$, $\br_i\mapsto \tTT_i$ for $i\in\I$.
\end{theorem}

Following \cite[Section 4.5]{WZ23}, let us define $\K_i:=\widetilde{\Psi}_{\diamond}^{-1}(\tk_i)=\varsigma_{i,\diamond}^{-1}\tk_i$ and
\[\K_\lambda:=\prod_{i\in\I}\K_i^{m_i}\quad\text{if $\lambda=\sum_im_i\alpha_i$}.\]
Then we have $\widetilde{\TT}_i(\K_i)=\K_{r_i(\alpha_i)}$ for each $i\in\I$. Also, denote by $\tau_i$ the diagram involution of $\I_i:=\{i,\tau i\}$ defined by 
\begin{align} \label{taui}
r_i(\alpha_i)=-\alpha_{\tau_i(i)}, 
\qquad
r_i(\alpha_{\tau i})=-\alpha_{\tau_i(\tau i)}.
\end{align}

\begin{proposition} [{\cite[Proposition 4.11, Theorem 4.14]{WZ23}}]
\label{prop:TiBi}
For $i,j\in\I$, we have $\widetilde{\TT}_i(\K_j)=\K_{r_i(\alpha_j)}$ and
\[\widetilde{\TT}_i(B_i)=v^{\frac{1}{2}(\alpha_i-\alpha_{\tau i},\alpha_i)}\K_{\tau_i(i)}^{-1}B_{\tau_i(\tau i)},\quad \widetilde{\TT}_i(B_{\tau i})=v^{\frac{1}{2}(\alpha_i-\alpha_{\tau i},\alpha_i)}\K_{\tau_i(\tau i)}^{-1}B_{\tau_i(i)}.
\]
\end{proposition}

\begin{proof}
We give a direct proof in our greatly simplified quasi-split setting. In fact, we have
\begin{align*}
\widetilde{\mathscr{T}}_{r_i}^{-1}(B_i)&=\widetilde{\mathscr{T}}_{r_i}^{-1}(F_i+\varsigma_iE_{\tau i}K_i')\\
&=v_iE_{\tau_i(i)}K'^{-1}_{\tau_i(i)}+\varsigma_{i,\diamond}^2\varsigma_iv_i^{-1}K_{\tau_i(\tau i)}^{-1}F_{\tau_i(\tau i)}K'^{-1}_{\tau_i(i)}\\
&=v_i\varsigma_i\varsigma_{i,\diamond}^2(\varsigma_i^{-1}\varsigma_{i,\diamond}^{-2}v_i^{-c_{i,\tau i}}K_{\tau_i(\tau i)}E_{\tau_i(i)}+F_{\tau_i(\tau i)})(K_{\tau_i(\tau i)}K_{\tau_i(i)}')^{-1}\\
&=v_i\varsigma_i\varsigma_{i,\diamond}^2(\varsigma_iK_{\tau_i(\tau i)}E_{\tau_i(i)}+F_{\tau_i(\tau i)})(K_{\tau_i(\tau i)}K_{\tau_i(i)}')^{-1}\\
&=v^{\frac{1}{2}(\alpha_i,\alpha_i-\alpha_{\tau i})}B_{\tau_i(\tau i)}^\sigma\K_{\tau_i(\tau i)}^{-1},
\end{align*}
Since $\widetilde{\Upsilon}_iB_{\tau_i(\tau i)}^\sigma=B_{\tau_i(\tau i)}\widetilde{\Upsilon}_i$ (see Proposition \ref{prop:Kmatrix}), we conclude that the element
\[\widetilde{\TT}_i^{-1}(B_i)=v^{\frac{1}{2}(\alpha_i,\alpha_i-\alpha_{\tau i})}B_{\tau_i(\tau i)}\K_{\tau_i(\tau i)}^{-1}\]
satisfies \eqref{eq:relbraid1}. The formula of $\widetilde{\TT}_i(B_i)$ then follows by applying $\sigma^\imath$; cf. Theorem \ref{thm: relative T_i conjugate}.
\end{proof}

    \begin{theorem} [\cite{Z23}; \text{also cf. \cite{KP11, LW22b, WZ23}}]
		\label{thm:Ti}
	(1) For $i\in \I$ such that $\tau i=i$, there exists an automorphism $\widetilde{\TT}_i$ of the $\Q(v^{\frac12})$-algebra $\tUi$ such that
		$\widetilde{\TT}_i(\K_\mu) =\K_{\bs_i\mu}$, and
		\[
		\widetilde{\TT}_i(B_j)= \begin{cases}
			\K_i^{-1} B_i,  &\text{ if }j=i,\\
			B_j,  &\text{ if } c_{ij}=0, \\
			v_i^{-\frac{1}{2}}B_jB_i-v_i^{\frac{1}{2}}B_iB_j,  & \text{ if }c_{ij}=-1, \\
			[2]_i^{-1} \big(v_i^{-1}B_jB_i^{2}-[2]_i B_i B_jB_i +v_i B_i^{2} B_j \big) - \K_i B_j,  & \text{ if }c_{ij}=-2,\\
            [2]_i^{-1}[3]_i^{-1}(v_i^{-\frac{3}{2}}B_jB_i^3-v_i^{-\frac{1}{2}}[3]_iB_iB_jB_i^2+v_i^{\frac{1}{2}}[3]_iB_i^2B_jB_i\\
            -v_i^{\frac{3}{2}}B_i^3B_j-v_i^{-1}\K_i[B_j,B_i]_{v_i^3})+\K_i[B_j,B_i]_i,&\text{ if }c_{ij}=-3,
		\end{cases}
		\]
		for $\mu\in \Z^\I$ and $j\in \I$.
		
	(2) For $i\in \I$ such that $c_{i,\tau i}=0$, there exists an automorphism $\widetilde{\TT}_i$ of the $\Q(v^{\frac12})$-algebra $\tUi$ such that $\widetilde{\TT}_i(\K_\mu) =\K_{\bs_i\mu}$ and
		\[
		\widetilde{\TT}_i(B_j)= \begin{cases}
			v_i^{-\frac{1}{2}}B_jB_i -v_i^{\frac{1}{2}}B_iB_j,  & \text{ if }c_{ij}=-1 \text{ and } c_{\tau i,j}=0,
			\\
			v_i^{-\frac{1}{2}}B_jB_{\tau i}-v_i^{\frac{1}{2}} B_{\tau i}B_j,  & \text{ if } c_{ij}=0 \text{ and }c_{\tau i,j}=-1 ,
			\\
			v_i^{-1}[[B_j,B_i]_v,B_{\tau i}]_v-\K_i B_j,  & \text{ if } c_{ij}=-1 \text{ and }c_{\tau i,j}=-1 ,
			\\
			v_i\bK_{i}^{-1} B_{\tau i},  & \text{ if }j=i,
			\\
			v_i\bK_{\tau i}^{-1} B_i,  &\text{ if }j=\tau i,
			\\
			B_j, & \text{ otherwise.}
		\end{cases}
		\]
		
	(3) For $i\in \I$ such that $c_{i,\tau i}=-1$, there exists an automorphism $\widetilde{\TT}_i$ of the $\Q(v^{\frac12})$-algebra $\tUi$ such that $\widetilde{\TT}_i(\K_\mu) =\K_{\bs_i\mu}$ and
		\[
		\widetilde{\TT}_i(B_j)= \begin{cases}
			v_i^{-1}\big[ [B_j,B_i]_v,B_{\tau i} \big]_v-v_i^{-\frac{1}{2}}\K_i B_{j},  & \text{ if }c_{ij}=-1 \text{ and } c_{\tau i,j}=0,\\
			v_i^{-1}\big[ [B_j,B_{\tau i}]_v,B_{i} \big]_v-v_i^{-\frac{1}{2}}\bK_{\tau i} B_{j},  & \text{ if } c_{ij}=0 \text{ and }c_{\tau i,j}=-1 ,\\
			v_i^{-1}\Big[\big[ [B_j,B_i]_v,B_{\tau i}\big] ,[B_{\tau i},B_i]_v\Big]\\
			\qquad-v_i^{-\frac{1}{2}}\bK_{i}\big[B_j, [B_{\tau i},B_i]_{v^3} \big]-\bK_i\bK_{\tau i}B_j,  & \text{ if } c_{ij}=-1 \text{ and }c_{\tau i,j}=-1 ,\\
			v^{\frac{3}{2}}\bK_{\tau i}^{-1}B_{i},  & \text{ if }j=i,\\
			v^{\frac{3}{2}}\bK_{i}^{-1}B_{\tau i},  &\text{ if }j=\tau i,\\
			B_j, & \text{ otherwise.}
		\end{cases}
		\]
		
Moreover, there exists a homomorphism $\brW \rightarrow \Aut( \tUi)$, $r_i\mapsto \widetilde{\TT}_i$, for all $i\in \I_\tau$.
	\end{theorem}

 \subsection{Drinfeld presentation of affine iquantum groups}


The Drinfeld presentation of quasi-split (untwisted) affine $\tUi$ is established in a series of works \cite{LW21b,Z22,LWZ23,LPWZ25}, owing to the more numerous classes of Satake diagrams relative to Dynkin diagrams. Here, we review the reduced version following \cite{Z22}.

\begin{definition} 
		\label{def:iDRA1}
	Let ${^{\textup{\texttt{Dr}}}\tUi}$ be the $\Q(v^{\frac12})$-algebra generated by the elements $B_{i,l}$, $\Theta_{i,m}$, $\bK_i^{\pm1}$, $\ttC^{\pm1}$, where $i\in\II$, $l\in\Z$ and $m>0$, subject to the following relations for $i,j\in \II$, $m,n>0,$ and $l,k\in \Z$:
\begin{align}
	\ttC^{\pm1} \text{ is central,} \quad [\bK_i,\bK_j]&  =  [\bK_i,\bTH_{j}(z)]=0,\quad \bK_i\bB_{j}(z)=v_i^{c_{\tau i,j}-c_{ij}} \bB_{j}(z) \bK_i,
   \label{qsiDR0}
			\\
   \label{qsiDR1}
			 [\bTH_{i}(z),\bTH_{j}(w)] &=0,
			\\
   \label{qsiDR2}
		g_{\tau i,j}(zw\mathcal{C})\bTH_i(z)\bB_j(w)
		&= 
		g_{ij}(zw^{-1})\bB_j(w)\bTH_i(z) , 
			\\
            [\bB_i(z),\bB_{\tau i}(w)] &=\frac{\Delta(zw)}{v_i-v_i^{-1}} \big(\K_i\bTH_{\tau i}(w)-\K_{\tau i}\bTH_i(z)\big),
		\qquad \text{ if } c_{i,\tau i}=0,
            \\
        (w-v_i^{c_{ij}}z)\bB_i(z) \bB_j(w) &-(v_i^{c_{ij}}w-z) \bB_j(w) \bB_i(z)  \notag 
			\\
		=\delta_{i,\tau j}v_i^{-\frac{1}{2}c_{i,\tau i}}\frac{\Delta(zw)}{v_i-v_i^{-1}}\big((w&-v_i^{c_{ij}}z)\K_{\tau j}\bTH_j(w)-(v_i^{c_{ij}}w-z)\K_{\tau i}\bTH_i(z)\big), \label{qsiDR4}
\end{align}
and the Serre relations
\begin{align}
	&\sum_{s=0}^{r}(-1)^s\qbinom{r}{s}_iB_{i,k}^{r-s}B_{j,l}B_{i,k}^s=0, \qquad\quad\text{ for } \tau i\neq i\neq j\neq \tau i, r=1-c_{ij},\label{qsiDR5}\\
	&\sum_{s=0}^2(-1)^s\qbinom{2}{s}_iB_{i,k}^{2-s}B_{\tau i,k}B_{i,k}^s=v_i^{\frac{1}{2}}[2]_i(v_i\K_iB_{i,k}+v_i^{-2}\K_{\tau i}B_{i,k})\ttC^k,\quad \text{if }c_{i,\tau i}=-1,\label{qsiDR6}
    \end{align}
    and for $\tau i=i\neq j$, 
    \begin{align}
    &\sum_{s=0}^2(-1)^s\qbinom{2}{s}_iB_{i,k}^{2-s}B_{j,l}B_{i,k}^s=-B_{j,l}\K_i\ttC^k,\quad \text{if } c_{ij}=-1,
    \\
    &\sum_{s=0}^3(-1)^s\qbinom{3}{s}_iB_{i,k}^{3-s}B_{j,l}B_{i,k}^s=-[2]_i^2(B_{i,k}B_{j,l}-B_{j,l}B_{i,k})\K_i\ttC^k,\quad \text{if } c_{ij}=-2,
    \\
    \label{iSerre4}
    &\sum_{s=0}^4(-1)^s\qbinom{4}{s}_iB_{i,k}^{4-s}B_{j,l}B_{i,k}^s=-(1+[3]_i^2)(B_{j,l}B_{i,k}^2+B_{i,k}^2B_{j,l})\K_i\ttC^k
    \\\notag
    &\qquad\qquad+[4]_i(1+[2]_i^2)B_{i,k}B_{j,l}B_{i,k}\K_i\ttC^k-[3]_i^2B_{j,l}\K_i^2\ttC^{2k},\quad \text{if } c_{ij}=-3.
\end{align}
Here we introduce the generating functions
	\begin{align}\label{eq:Genfun}
		\begin{cases}
			\bB_{i}(z) =\sum_{k\in\Z} B_{i,k}z^{k},
			\\
			\bTH_{i}(z)  =1+ \sum_{m > 0}\Theta_{i,m}z^{m},
			\\
			\Delta(z)=\sum_{k\in\Z} \ttC^kz^k.
		\end{cases}
	\end{align}
\end{definition}
For convenience, we set 
$\Theta_{i,0} =1$.

We remark that the reduced form of the Serre relations is adopted herein to define 
$\tUi$ (cf. \cite[Definition 3.1]{Z22}, \cite[Proposition 6.20]{LWZ23}). This choice is motivated by the fact that the general Serre relations are derivable from \eqref{qsiDR5}–\eqref{iSerre4} when combined with the other relations of ${^{\textup{\texttt{Dr}}}\tUi}$. Indeed, the definition can be further simplified by setting $k=0=l$ in all relations \eqref{qsiDR5}–\eqref{iSerre4}. For example, by \cite[(5.7)]{LWZ23}, the general version of \eqref{qsiDR6}
is
\begin{align}
    \label{eq:qsiDR6-general}
&\Sym_{k_1,k_2}\Big(B_{i,k_1}B_{i,k_2}B_{\tau i,l}-[2]_iB_{i,k_1} B_{\tau i,l} B_{i,k_2} + B_{\tau i,l} B_{i,k_1}B_{i,k_2}\Big)
    \\\notag
    &=-\frac{v_i^{\frac{1}{2}}[2]_i}{v_i-v_i^{-1}}\Big(\Sym_{k_1,k_2}\sum_{p\geq 0}v_i^{2p}
            \big[\Theta_{\tau i,l-k_2-p}\K_i-v_i\Theta_{\tau i,l-k_2-p-2}\ttC\K_i, B_{i,k_1-p} \big]_{v_i^{-4p-1}}\ttC^{k_2+p}
            \notag 
            \\\notag
             &\quad-v_i\Sym_{k_1,k_2}\sum_{p\geq 0}v_i^{2p} \big[ B_{i,k_1+p+1},\Theta_{i,k_2-l-p+1}\K_{\tau i}-v_i\Theta_{i,k_2-l-p-1}\ttC \K_{\tau i}\big]_{v_i^{-4p-3}} \ttC^{l-1}\Big).
\end{align}
From the proof of \cite[Proposition 6.20]{LWZ23}, we know that \eqref{eq:qsiDR6-general} can be deduced from \eqref{qsiDR6} (even the one when $k=0$) under the assumption of \eqref{qsiDR0}--\eqref{qsiDR4}.

Similar to quantum affine algebras, there exists an isomorphism of $\Q(v^{\frac12})$-algebras \cite{LW21b,Z22,LWZ23,LWZ24,LPWZ25}
	\begin{align}
		\label{eq:phiUi1}
		\widetilde{\Omega}^\imath: {^{\textup{\texttt{Dr}}}\tUi} \longrightarrow \tUi,
	\end{align}
    with the help of relative braid group actions of $\tUi$.
    
\section{Drinfeld presentation via {\rm i}Hopf algebras}
\label{sec:iDrinfeld-affine}

In this section, we use the iHopf algebra machinery introduced in \cite{CLPRWa} to obtain the Drinfeld presentation of $\tUi$ from that of $\tU$ for affine quantum groups $\tU$. As an application, we obtain the relationship between the Drinfeld generators of $\tUi$ and those of $\tU$. 

\subsection{iHopf algebras}
\label{subsec:iHopf}
In this subsection, we recall the main construction of \cite{CLPRWa}, which produces a new associative algebra structure defined on a Hopf algebra $H$ with a Hopf pairing, called an iHopf algebra and denoted by $H^\imath$. 

Assume that $(H,m,\Delta,1,\varepsilon)$, or simply $H$ below, is a Hopf algebra with multiplication $m$, unit $1$, comultiplication $\Delta$ and counit $\varepsilon$. The antipode of $H$ is denoted by $S$, which will always be assumed to be invertible. We write $ab=m(a,b)$ and use the Sweedler notation $\Delta(a)=\sum a_{(1)}\otimes a_{(2)}$ for $a,b\in H$. Define  $\Delta^{(2)}:=(\Delta\otimes\Id)\circ\Delta $, and then $\Delta^{(n)}$ by iteration. Denote 
$\Delta^{(n)}(a)=\sum a_{(1)}\otimes a_{(2)}\otimes\cdots\otimes a_{(n+1)}$. 

Let $\varphi$ be a Hopf pairing of $H$ (see e.g. \cite[\S3.2.1]{Jo95}), that is, $\varphi$ is a bilinear form satisfying:
\begin{itemize}
    \item[(H1)] $\varphi(1,b)=\varepsilon(b)$,  $\varphi(a,1)=\varepsilon(a)$, $\forall a,b\in H$;
    \item[(H2)] $\varphi(a,bb')=\varphi(\Delta(a),b\otimes b')$, $\forall a,b,b'\in H$; 
    \item[(H3)] $\varphi(aa',b)=\varphi(a\otimes a',\Delta(b))$, $\forall a,a',b\in H$;
    \item[(H4)] $\varphi(a,S(b))=\varphi(S(a),b)$, $\forall a,b,\in H$.
\end{itemize}
Here the pairing on the tensor product space is given by $\varphi(a \otimes a', b \otimes b') = \varphi(a, b)\varphi(a', b')$.

\begin{definition}[\cite{CLPRWa}]
The {\em iHopf algebra}, denoted by $H^\imath =\mathrm{iHopf}\, (H,\varphi)$, is the same vector space as $H$ equipped with a new multiplication: 
\begin{align}\label{star product}
    a\ast b:=\sum \varphi(b_{(2)},a_{(1)})\cdot a_{(2)}b_{(1)},\quad \forall a,b\in H,
\end{align}
where $\Delta(a)=\sum a_{(1)}\otimes a_{(2)}$, $\Delta(b)=\sum b_{(1)}\otimes b_{(2)}$.
\end{definition}
Then $H^\imath$ is an associative algebra with the unit $1$; see \cite[Proposition 2.2]{CLPRWa}.

Let $H$ be a Hopf algebra with a Hopf pairing $\varphi$. Then $H\otimes H$ is a Hopf algebra with a Hopf pairing $\varphi_{\sharp}$ given by  
$\varphi_{\sharp}(c\otimes d,a\otimes b)=\varphi(a,d)\, \varphi(c,b)$. 
The iHopf algebra $(H\otimes H)^\imath$ associated with $(H\otimes H,\varphi_{\sharp})$ is called the iHopf algebra of diagonal type. Its multiplication is explicitly given by
\[(a\otimes b)\ast (c\otimes d)=\sum\varphi (a_{(1)},d_{(2)})\cdot \varphi(c_{(2)},b_{(1)})\cdot a_{(2)}c_{(1)}\otimes b_{(2)}d_{(1)},\qquad\forall a,b,c,d\in H.
\]
Moreover, $(H\otimes H)^\imath$ 
 can be made into a Hopf algebra with coproduct, counit and antipode given by 
\begin{align*}
       \Delta^\imath(a\otimes b)&=\sum\varphi(a_{(2)},b_{(2)})\cdot (a_{(1)}\otimes b_{(3)})\otimes (a_{(3)}\otimes b_{(1)}),
       \\
       \varepsilon^\imath(a\otimes b)&=\varphi(a,S^{-1}(b)),\\
       S^\imath(a\otimes b)&=S(a)\otimes S^{-1}(b).
   \end{align*} 
   See \cite[Proposition 2.7]{CLPRWa}.

There is a construction of Drinfeld doubles $D(H)$; see \cite{Dr87}, \cite[\S3.2]{Jo95}. From \cite{CLPRWa}, we see that there is an isomorphism of Hopf algebras
    \begin{align*}
\Phi^{\textup{\texttt{D}}}:D(H)\longrightarrow (H\otimes H)^\imath,\quad (a\otimes b)\mapsto (a\otimes 1)\ast(1\otimes b),\quad\forall a,b\in H.
    \end{align*}

\begin{definition}[{\cite[Definition 2.10]{CLPRWa}}]
\label{def:chi}
     Let $H$ be a Hopf algebra with a Hopf pairing $\varphi$, and $\tau$ a Hopf algebra endomorphism of $H$. A linear map $\chi:H\rightarrow \F$ is {\em $\tau$-twisted compatible} if the identity $\chi(ab)=\sum\chi(a_{(1)})\chi(b_{(2)})\varphi(\tau (a_{(2)}),b_{(1)})$ holds for all $a,b\in H$.
\end{definition}
We just call  $\chi$  a compatible map if it is  $\tau$-twisted compatible with $\tau=\Id$.
We refer to Lemma \ref{lem:fvartheta} for a natural example of $\chi$ when $H$ is a Borel quantum group. 


For a Hopf algebra endomorphism $\tau$ of $H$, we denote by $H^\imath_\tau$  the iHopf algebra associated with $H$ and the Hopf pairing $\varphi \circ(\tau \otimes 1)$. 

\begin{proposition}[{\cite[Proposition 2.12]{CLPRWa}}]
\label{prop:embedding-iHopf-diagonal}
    Let $H$ be a Hopf algebra with a Hopf pairing $\varphi$, and $\chi:H\rightarrow \F$ be a $\tau$-twisted compatible map for a Hopf algebra endomorphism $\tau$ of $H$. Then there is an algebra homomorphism 
    \begin{align} \label{xi}
    \xi_{\tau,\chi}: H^\imath_\tau\longrightarrow (H\otimes H)^\imath,
    \qquad  a \mapsto \sum \chi(a_{(2)})\cdot \tau(a_{(3)})\otimes a_{(1)}.
    \end{align}
\end{proposition}

\subsection{iQuantum groups via iHopf algebras}
\label{subsec:iQG-iHopf}
The results given
in this subsection hold for any quantum groups and iquantum groups (of Kac-Moody type).

Recall the Cartan matrix $C=(c_{ij})$ and $D=\diag(d_i\mid i\in\I)$. Let $'\ff$ be the free associative $\Q(v^{\frac12})$-algebra with generators $\theta_i$ ($i\in\I$); see \cite[Chap. 1]{Lus93}. Let $'\tB$  be the $\Q(v^{\frac12})$-algebra generated by
$\theta_i,\h_i^{\pm1}$ $(i\in\I)$ subject to 
\begin{align*}
  [\h_i,\h_j]=0, \qquad h_i\theta_j=v_i^{c_{ij}} \theta_j h_i.  
\end{align*}

Let $\ff$ (resp. $\tB$) be the quotient algebra of $'\ff$ (resp. $'\tB$) by the ideal generated by
\begin{align}
& \sum_{r=0}^{1-c_{ij}} (-1)^r \left[ \begin{array}{c} 1-c_{ij} \\r \end{array} \right]_i  \theta_i^r \theta_j  \theta_i^{1-c_{ij}-r},\qquad \forall i\neq j.
\end{align}
We endow $\ff$ with an $\N^\I$-grading by setting $\wt(\theta_i)=\alpha_i$. Let $\ff_\mu$ be the homogeneous subspace of degree $\mu$. Then $\ff=\bigoplus_{\mu\in\N^\I}\ff_\mu$. 

We identify
\begin{align} \label{ffUU}
\ff \stackrel{\cong}{\longrightarrow}\U^+, \,  \theta_i \mapsto \theta_i^+ :=E_i,  \qquad
\ff \stackrel{\cong}{\longrightarrow}\U^-, \,\theta_i \mapsto \theta_i^- := F_i. 
\end{align}
Then we have  $\tB\cong \U^+\otimes \Q(v^{\frac12})[h_i^{\pm1}\mid i\in\I]$.

Define the coproduct, counit and antipode by
\begin{gather}
\Delta(\theta_i)= \theta_i\otimes 1 +  h_i\otimes \theta_i, \qquad
 \Delta(\h_{i}) = \h_{i} \otimes \h_{i},\quad \forall i\in\I;
 \\
 \varepsilon(\theta_i)= 0, \qquad
 \varepsilon(\h_{i}) = 1=\varepsilon(\h_i^{-1}),\quad \forall i\in\I,\\
 S(\theta_i)=-h_i^{-1}\theta_i,\quad S(h_i)=h_i^{-1},\quad \forall i\in\I.
\end{gather}
In this way, ${'\tB},\tB$ are Hopf algebras. 

Define 
\begin{align}
\label{eq:hopf-pairing}
\varphi(\theta_i,\theta_j)=\frac{\delta_{ij}}{v_i-v_i^{-1}},\quad \varphi(\h_i,\h_j)=v_i^{c_{ij}},\quad\varphi(\theta_i,\h_j)=0,\quad \forall i,j\in\I.
\end{align}
Then it gives (symmetric) Hopf pairings on the Hopf algebras $'\tB,\tB$. Moreover, this Hopf pairings is non-degenerate on $\tB$.

Consider the iHopf algebras of diagonal type,  $(\tB\otimes\tB )^\imath$, associated to $(\tB,\varphi)$, respectively.   

\begin{lemma}\label{lem:U=iHopfBB}
We have a Hopf algebra isomorphism 
\begin{align*}
\widetilde{\Phi}_{\sharp}:&\tU\longrightarrow (\tB\otimes\tB )^\imath,\\
E_i\mapsto\theta_i\otimes 1,&\quad F_i\mapsto 1\otimes \theta_i,\quad K_i \mapsto  h_i\otimes1,\quad K_i'\mapsto 1\otimes h_i,\quad \forall i\in\I.
\end{align*}
\end{lemma}

Let $\tau$ be an involution of the Cartan matrix $C$. Clearly $\tau$ preserves the Hopf pairing $\varphi$. Denote by
\[
\tB^\imath_\tau =\text{iHopf}\, \big(\tB, \varphi\circ (\tau \otimes 1)\big)
\]
the iHopf algebras defined on  $\big(\tB, \varphi\circ (\tau \otimes 1)\big)$.

\begin{lemma}[{\cite[Lemma 4.2]{CLPRWa}}]
\label{lem:fvartheta}
    Let $\tau$ be an involution of $C$. Then there exists a unique $\tau$-twisted compatible map $\chi:\tB\rightarrow \Q(v^{\frac12})$ such that $\chi(1)=1$, 
    $\chi(\theta_i)=0$, $\chi(h_i)=\varsigma_i\varphi(h_i,h_{\tau i})$, for $i\in\I$.
\end{lemma}
From Proposition \ref{prop:embedding-iHopf-diagonal}, we have the following. 
\begin{proposition} \label{lem:iHembed}
There is an embedding
\begin{equation} \label{eq:xi_theta embedding of iHopf}
\widetilde{\xi}_\tau: \tB^\imath_\tau\longrightarrow(\tB\otimes\tB)^\imath
\end{equation}
which sends
\begin{align}
\label{eq:Xitheta}
\theta_{i}\mapsto 1\otimes\theta_i+\varsigma_i(\theta_{\tau i}\otimes 1)\ast(1\otimes h_i), 
\quad
h_{i}\mapsto \varsigma_i(h_{\tau i}\otimes 1)\ast(1\otimes h_i),
\qquad\forall i\in \I.
\end{align}
\end{proposition}

\begin{theorem}
[{\cite[Theorem 4.4]{CLPRWa}}]
\label{thm:iH=iQG}
    The embedding of Proposition~\ref{lem:iHembed} induces an algebra isomorphism
    \begin{align*} &\widetilde{\Phi}^\imath:\tB^\imath_\tau\longrightarrow \tUi, \\
        \theta_i \mapsto& B_i,\qquad h_i\mapsto \varsigma_i\tk_{\tau i},\quad \forall i\in\I.
    \end{align*}
\end{theorem}

For $\alpha\in\Z^\I$, let us set
\begin{align}
\H_\alpha:=v^{\frac{1}{2}(\alpha,\tau\alpha)}h_\alpha.
\end{align}
Then it is clear from definition that $\H_\alpha\ast \H_\beta=\H_{\alpha+\beta}$. Moreover, since
\[\widetilde{\Phi}^\imath(\H_i)=v_i^{\frac{1}{2}c_{i,\tau i}}\varsigma_i\tilde{k}_{\tau i}=\varsigma_i\varsigma_{i,\diamond}\K_{\tau i}=\K_{\tau i},\]
we conclude that 
\[\widetilde{\Phi}^\imath(\H_\alpha)=\K_{\tau\alpha},\quad\forall\alpha\in\Z^\I.\]
In particular, we obtain $\widetilde{\Phi}^\imath(h_\delta)=\widetilde{\Phi}^\imath(\H_\delta)=\K_\delta$.

Let us recall the braid group actions $\widetilde{\TT}_i$ on $\tUi$. Using the isomorphism of Theorem~\ref{thm:iH=iQG}, $\widetilde{\TT}_i$ may be regarded as an automorphism of $\tB^\imath_\tau$. The following is one of the main results of \cite{CLPRWb}.

\begin{theorem}[{\cite[Theorem 3.7]{CLPRWb}}]\label{thm: braid on ihopf comp}
For any $x\in\ff$ such that $\widetilde{T}_{r_i}(x^+)\in\tU^+$, we have $\widetilde{\TT}_i(x)=\widetilde{T}_{r_i}(x)$.
\end{theorem}

\subsection{Drinfeld presentation}

Let us restrict our attention to affine quantum groups and iquantum groups in this subsection. 

Recall the braid group actions $\widetilde{T}_i$ on $\tU$. Let $\varrho=\frac{1}{2}\sum_{\alpha \in \mathcal{R}_0^+} \alpha$.  Let us fix a reduced expression $2\varrho=s_{i_1}\cdots s_{i_N}$ in $\widetilde{W}$ where $N=\ell(2\varrho)$. Define an infinite sequence $(\cdots,i_{-1},i_0,i_1,\dots)$ by $i_k:=i_{\bar{k}}$ for $k\in\Z$, where $\bar{k}$ denote the image of $k$ in $\Z/N\Z$. Then the PBW basis of $\tU^+$ is defined as
\begin{align*}
E_{\beta_k}:=\begin{cases}
\widetilde{T}_{i_0}^{-1}\widetilde{T}_{i_{-1}}^{-1}\cdots \widetilde{T}_{i_{k+1}}^{-1}(E_{i_k})&\text{if $k\leq 0$},\\
\widetilde{T}_{i_1}\widetilde{T}_{i_2}\cdots \widetilde{T}_{i_{k-1}}(E_{i_k})&\text{if $k\geq 1$}.
\end{cases}
\end{align*}
It follows from \cite[Proposition 40.1.3]{Lus93} that the elements $E_{\beta_k}\in\tU^+$. We also define (our convention is different from the literature because of the normalization on $\widetilde{T}_i$)
\[\widetilde{E}_{m\delta,i}:=v_iE_{m\delta-\alpha_i}E_i-v_i^{-1}E_iE_{m\delta-\alpha_i},\quad \forall m\geq1.\]
Moreover, we can choose the reduced expression of $2\varrho$ such that $E_{r\delta+\alpha_i}=\widetilde{T}_{\omega_i}^{-r}(E_i)$ and $E_{(r+1)\delta-\alpha_i}=\widetilde{T}_{\omega_i}^{r+1}(T_i^{-1}(E_i))$ for $r\geq 0$. 
We set
\[1+\sum_{m\geq 1}(v_i-v_i^{-1})\widetilde{E}_{m\delta,i}u^m=\exp\Big((v_i-v_i^{-1})\sum_{m\geq 1}E_{m\delta,i}u^m\Big).\]
For convenience, we set $\widetilde{E}_{m\delta,i}=\frac{1}{v_i-v_i^{-1}}$ for $m=0$.

Let us choose a map $o:\I\to\{\pm 1\}$ such that $o(i)o(j)=-1$ whenever $c_{ij}<0$. Then the elements $E_{r\delta\pm\alpha_i}$ and $E_{m\delta,i}$ satisfy the following relations \cite{Da24}:
\begin{align}
[E_{k\delta,i},E_{l\delta,j}]&=0\quad \forall k,l>0,\label{PBW rel-1}\\
[E_{m\delta,i},E_{r\delta\pm\alpha_j}]&=\pm o(i)^mo(j)^m\frac{[mc_{ij}]_i}{m}E_{(r+m)\delta\pm\alpha_j},\label{PBW rel-2}\\
E_{(k+1)\delta\pm\alpha_i}E_{l\delta\pm\alpha_j}-v_i^{\pm c_{ij}}E_{l\delta\pm\alpha_j}E_{(k+1)\delta\pm\alpha_i}&=v_i^{\pm c_{ij}}E_{(l+1)\delta\pm\alpha_j}E_{k\delta\pm\alpha_i}-E_{k\delta\pm\alpha_i}E_{(l+1)\delta\pm\alpha_j}\label{PBW rel-3}.
\end{align}

Let us define the following sets:
\begin{gather*}
\mathcal{R}^+(+\infty):=\{\beta_k\mid k\geq 1\},\quad \mathcal{R}^+(-\infty):=\{\beta_k\mid k\leq 0\},\\
\mathcal{R}^+(\mathrm{im}):=\{(m\delta,i)\mid m>0,i\in\I_0\},\\
\mathcal{R}^+(+\tilde{\infty}):=\mathcal{R}^+(+\infty)\cup \mathcal{R}^+(\mathrm{im}),\quad \mathcal{R}^+(-\tilde{\infty}):=\mathcal{R}^+(-\infty)\cup \mathcal{R}^+(\mathrm{im})
\end{gather*}

For $\ast\in\{\mathrm{im},\pm\infty,\pm\tilde{\infty}\}$,
let $\widetilde{\mathbf U}^{+}(\ast)$ be the subalgebra of
$\widetilde{\mathbf U}^{+}$ generated by the root vectors associated
with $R^{+}(\ast)$; namely,
\begin{align*}
\widetilde{\mathbf U}^{+}(+\infty)
&=
\left\langle E_{\beta_k}\mid k\geq 1\right\rangle,\qquad
&\widetilde{\mathbf U}^{+}(-\infty)
=
\left\langle E_{\beta_k}\mid k\leq 0\right\rangle,\\
\widetilde{\mathbf U}^{+}(\mathrm{im})
&=
\left\langle E_{m\delta,i}\mid m>0,\ i\in I_0\right\rangle,\qquad
&\widetilde{\mathbf U}^{+}(\pm\tilde{\infty})
=
\langle
\widetilde{\mathbf U}^{+}(\pm\infty),
\widetilde{\mathbf U}^{+}(\mathrm{im})
\rangle.
\end{align*}
We further set
\[
\widetilde{\mathbf U}^{\geq0}(\ast)
:=
\widetilde{\mathbf U}^{+}(\ast)\widetilde{\mathbf U}^{0}.
\]

According to \cite{Da24}, we have the following formulas for the coproduct of $E_{r\delta+\alpha_i}$ and $E_{m\delta,i}$:
\begin{align}
\Delta(E_{m\delta,i})&\in E_{m\delta,i}\otimes 1+K_{m\delta}\otimes E_{m\delta,i}+\tU^{\geq 0}(+\infty)\otimes\tU^+(-\infty),\label{coprod of root vec-1}\\
\Delta(E_{r\delta+\alpha_i})&\in E_{r\delta+\alpha_i}\otimes 1+(v_i-v_i^{-1})\sum_{0\leq s\leq r}\widetilde{E}_{s\delta,i}K_{(r-s)\delta+\alpha_i}\otimes E_{(r-s)\delta+\alpha_i}\label{coprod of root vec-2}\\
&\qquad+\sum_{\substack{\eta\in Q_0\setminus\{0\}\\0\leq s<r}}\tU^{\geq 0}_{(r-s)\delta-\eta}\otimes \tU^+_{s\delta+\alpha_i+\eta},\notag
\end{align}
where $m\geq 1$ and $r\geq 0$. More generally, for $k\in\Z$, we have (see \cite{Da12})
\begin{align*}
\Delta(E_{\beta_k})-(E_{\beta_k}\otimes 1+K_{\beta_k}\otimes E_{\beta_k})\in\begin{cases}
\tU^{\geq 0}(+\infty)\otimes\tU^+&\text{if $k\geq 1$},\\
\tU^{\geq 0}\otimes\tU^+(-\infty)&\text{if $k\leq 0$},
\end{cases}
\end{align*}
and in particular
\begin{align}
\Delta(\tU^+(-\infty))\subseteq \tU^{\geq 0}\otimes\tU^+(-\infty),\quad \Delta(\tU^+(+\infty))\subseteq \tU^{\geq 0}(+\infty)\otimes\tU^+.
\end{align}

We also recall the explicit action of the isomorphism $\widetilde{\Phi}: {^{\textup{\texttt{Dr}}}\tU} \longrightarrow \tU$ on root vectors:
\begin{lemma}[\cite{Be94}]
The isomorphism $\widetilde{\Phi}: {^{\textup{\texttt{Dr}}}\tU} \longrightarrow \tU$ in \eqref{eq:phi1} sends
\begin{align*}
    X_{i,r}^+\mapsto o(i)^rE_{r\delta+\alpha_i},
    \quad X_{i,r}^-\mapsto -o(i)^rK_{\alpha_i-r\delta}E_{r\delta-\alpha_i},
   \quad \psi_{i,m}^+\mapsto o(i)^m(v_i-v_i^{-1})\widetilde{E}_{m\delta,i},
\end{align*}
for all $r,m\geq1$.
\end{lemma}

For $k\in\Z$, let us denote by $\theta_{\beta_k}$ the unique elements in $\tB$ such that $\theta_{\beta_k}^+=E_{\beta_k}$. Similarly we define elements $\theta_{m\delta,i}$ and $\widetilde{\theta}_{m\delta,i}$ for $m\geq1$ so that $\theta_{m\delta,i}^+=E_{m\delta,i}$, $\widetilde{\theta}_{m\delta,i}=\widetilde{E}_{m\delta,i}$. From \cite[Corollary 9.9]{Da98} and \cite[Proposition 38.2.1~(a)]{Lus93}, we see that
\begin{align}
\label{eq:HPim-1}
\varphi(\theta_{m\delta,i},\theta_{n\delta,j})&=o(i)^mo(j)^m\delta_{m,n}\frac{[mc_{ij}]_i}{m(v_j-v_j^{-1})},\\
\label{eq:HPre-1}
\varphi(\theta_{k\delta\pm\alpha_i},\theta_{l\delta\pm\alpha_j})&=\frac{\delta_{k,l}\delta_{i,j}}{v_i-v_i^{-1}}.
\end{align}

For $m\geq 0$, let us define 
\begin{align}
  \psi_{i,m}=\begin{cases}
  o(i)^m(v_i-v_i^{-1})\widetilde{\theta}_{m\delta,i}&\text{ if }m\geq1,
  \\
  1&\text{ if }m=0,
  \\
  0&\text{ if }m<0.
  \end{cases}
\end{align}
Then \eqref{coprod of root vec-1},
\begin{align}
    \Delta(\psi_{i,m})-\sum_{k=0}^m \psi_{i,k}h_{(m-k)\delta}\otimes \psi_{i,m-k}\in\tU^{\geq 0}(+\tilde{\infty})\otimes\tU^+(-\tilde{\infty}).
\end{align}

Also recall that $g_{ij}(z)=\frac{1-v_i^{-c_{ij}}z}{1-v_i^{c_{ij}}z}$. Then we have the following lemma.
\begin{lemma}\label{lem: Hopf pair of psi_i}
 For $m\geq0$, we have
\[\varphi(\psi_{i,m},\psi_{j,m})=g_{ij}^{(m)},\]
where $g_{ij}(z)=\sum_{n\geq 0}g_{ij}^{(n)}z^n$. More precisely,
\begin{align}
\label{eq:gijn}
g_{ij}^{(n)}=\begin{cases}
1&\text{if $n=0$},\\
v_i^{nc_{ij}}-v_i^{(n-2)c_{ij}}&\text{if $n>0$}.
\end{cases}
\end{align}
\end{lemma}
\begin{proof}
Let us define \begin{align*}
H_i(u):=\sum_{m\geq 1}\theta_{m\delta,i}u^m,\quad \psi_i(u):=1+\sum_{m\geq 1}\psi_{i,m}u^m.
\end{align*}
Then using \eqref{coprod of root vec-1} and the Hopf pairing property, we find that
\[\varphi(H_i(z)^m,H_j(w)^n)=\delta_{n,m}n!\cdot\varphi(H_i(z),H_j(w))^n.\]
Plugging into the definition of $\psi_i(u)$ then gives
\begin{align*}
\varphi(\psi_i(z),\psi_j(w))&=\varphi(\exp((v_i-v_i^{-1})H_i(z)),\exp((v_j-v_j^{-1})H_j(w)))\\
&=\sum_{m,n\geq 0}\frac{1}{n!m!}\varphi((v_i-v_i^{-1})^nH_i(z)^n,(v_j-v_j^{-1})^mH_j(w)^m)\\
&=\sum_{n\geq 0}\frac{1}{n!}\varphi((v_i-v_i^{-1})H_i(z),(v_j-v_j^{-1})H_j(w))^n\\
&=\exp(\log\frac{1-v_i^{-c_{ij}}zw}{1-v_i^{c_{ij}}zw})=g_{ij}(zw)
\end{align*}
where we have used the fact (see \eqref{eq:HPim-1}) that
\begin{align*}
\varphi\big((v_i-v_i^{-1})H_i(z),(v_j-v_j^{-1})H_j(w)\big)&=\sum_{m\geq 1}\frac{v_i^{mc_{ij}}-v_i^{-mc_{ij}}}{m}z^mw^m=\log\frac{1-v_i^{-c_{ij}}zw}{1-v_i^{c_{ij}}zw}.
\end{align*}
This completes the proof of the lemma.
\end{proof}

Recall that in \cite{LWZ24} the authors define root vectors of $\tUi$ by
\[B_{i,r}:=o(i)^r\widetilde{\TT}_{\bm{\varpi}_i}^{-r}(B_i),\quad r\in\Z,\]
where the element $\bm{\varpi}_i$ is defined by
\[\bm{\varpi}_i:=\begin{cases}
\varpi_i\varpi_{\tau i}&\text{if $i\neq\tau i$},\\
\varpi_i&\text{if $i=\tau i$}.
\end{cases}\]
We also remark that any reduced expression of $\bm{\varpi}_i$ ends at $r_i$ \cite[Lemma 2.5]{LWZ24}, so we set $\bm{\varpi}_i'=\bm{\varpi}_ir_i$.

Let $\widetilde{\xi}_\tau:\tB^\imath_\tau\to(\tB\otimes\tB)^\imath\cong\tU$ be the embedding in Proposition~\ref{lem:iHembed}.
We set
\[\ttC:=o(i)o(\tau i)h_{\delta}\]
for some $i\in\I_0$, so that $\widetilde{\Phi}^\imath(\ttC)=o(i)o(\tau i)\K_\delta$ (note that the value of $o(i)o(\tau i)$ is independent of $i\in\I_0$; see \cite{LWZ24}). 
The following lemma encodes the relation of $B_{i,r}$ with the root vectors $\theta_{r\delta\pm\alpha_i}$. 

\begin{lemma}\label{lem: B_ir as PBW}
For $r\geq 0$, we have
\begin{align}
B_{i,r}&=o(i)^r\widetilde{\xi}_\tau(\theta_{r\delta+\alpha_i}),\\
B_{i,-(r+1)}&=o(i)^{r+1}v^{\frac{1}{2}(\alpha_i,\alpha_i-\alpha_{\tau i})}\K_{-(r+1)\delta+\alpha_i}\widetilde{\xi}_\tau(\theta_{(r+1)\delta-\alpha_{\tau i}}).
\end{align}
\end{lemma}
\begin{proof}
The first equation follows directly from Theorem~\ref{thm: braid on ihopf comp}:
\begin{equation*}
\widetilde{\TT}_{\bm{\varpi}_i}^{-r}(B_i)=o(i)^r\widetilde{\xi}_\tau(\widetilde{T}_{\varpi_i}^{-r}(\theta_i))=o(i)^r\widetilde{\xi}_\tau(\theta_{r\delta+\alpha_i})
\end{equation*}
where we have used the fact that $\widetilde{T}_{\varpi_i}(E_j)=E_j$ for $j\neq i$. Let us prove the second one. For this we recall from Theorem~\ref{prop:TiBi} that
\[\widetilde{\TT}_i(B_i)=v^{\frac{1}{2}(\alpha_i-\alpha_{\tau i},\alpha_i)}\K_{\tau_i(i)}^{-1}B_{\tau_i(\tau i)},\]
so
\begin{align*}
B_{i,-(r+1)}&=o(i)^{r+1}\widetilde{\TT}_{\bm{\varpi}_i}^{r+1}(B_i)\\
&=o(i)^{r+1}v^{\frac{1}{2}(\alpha_i-\alpha_{\tau i},\alpha_i)}\widetilde{\TT}_{\bm{\varpi}_i}^r\widetilde{\TT}_{\bm{\varpi}_i'}(\K_{\tau_i(i)}^{-1}B_{\tau_i(\tau i)})\\
&=o(i)^{r+1}v^{\frac{1}{2}(\alpha_i-\alpha_{\tau i},\alpha_i)}\K_{-(r+1)\delta+\alpha_i}\widetilde{\TT}_{\bm{\varpi}_i}^r\widetilde{\TT}_{\bm{\varpi}_i'}(B_{\tau_i(\tau i)})\\
&=o(i)^{r+1}v^{\frac{1}{2}(\alpha_i-\alpha_{\tau i},\alpha_i)}\K_{-(r+1)\delta+\alpha_i}\widetilde{\xi}_\tau(\widetilde{T}_{\bm{\varpi}_i}^r\widetilde{T}_{\bm{\varpi}_i'}(\theta_{\tau_i(\tau i)}))\\
&=o(i)^{r+1}v^{\frac{1}{2}(\alpha_i-\alpha_{\tau i},\alpha_i)}\K_{-(r+1)\delta+\alpha_i}\widetilde{\xi}_\tau(\widetilde{T}_{\varpi_{\tau i}}^{r+1}T_{\tau i}^{-1}(\theta_{\tau i}))\\
&=o(i)^{r+1}v^{\frac{1}{2}(\alpha_i-\alpha_{\tau i},\alpha_i)}\K_{-(r+1)\delta+\alpha_i}\widetilde{\xi}_\tau(\theta_{(r+1)\delta-\alpha_{\tau i}}).
\end{align*}
This completes the proof of the lemma.
\end{proof}

Let us define $\theta_{i,r}=o(i)^r\widetilde{\TT}_{\bm{\varpi}_i}^{-r}(\theta_i)$ in $\tB^\imath_\tau$, for $r\in\Z$. Using Lemma~\ref{lem: B_ir as PBW}, it can also be written as
\begin{align}
\label{eq:real-root}
\theta_{i,r}=\begin{cases}
o(i)^r\theta_{r\delta+\alpha_i}&\text{if $r\geq 0$},\\ o(i)^rv^{\frac{1}{2}(\alpha_i,\alpha_i-\alpha_{\tau i})}\H_{r\delta+\alpha_{\tau i}}\ast\theta_{-r\delta-\alpha_{\tau i}}&\text{if $r<0$}.
\end{cases}
\end{align}
We then have the following multiplication formulas in $\tB^\imath_\tau$.

\begin{proposition}\label{prop: psi theta rel comp}
For any $m,n\geq 0$ and $r,l\in\Z$, the following relations are valid in $\tB^\imath_\tau$: 
\begin{align}
\psi_{i,m}\ast\psi_{j,n}&=\sum_{s\geq 0}g_{i,\tau j}^{(s)}\psi_{i,m-s}\psi_{j,n-s}\ttC^s,\label{eq:psi theta rel comp-1}\\
\sum_{0\leq s\leq m}g_{i,\tau j}^{(s)}\ttC^s\ast\psi_{i,m-s}\ast\theta_{j,r-s}&=\sum_{0\leq s\leq m}g_{ij}^{(s)}\theta_{j,r+s}\ast\psi_{i,m-s},\label{eq:psi theta rel comp-2}
\\
 \theta_{i,r}*\theta_{\tau i,l}-\theta_{\tau i,l}*\theta_{i,r} &=\frac{1}{v_i-v_i^{-1} }\big(\mathcal{C}^rh_{\tau i}\psi_{\tau i,l-r}-\mathcal{C}^lh_i\psi_{i,r-l}\big),
		\quad \text{if } c_{i,\tau i}=0,
        \\
\theta_{i,r+1}\ast\theta_{j,l}-v_i^{c_{ij}}\theta_{j,l}\ast\theta_{i,r+1}&-v_i^{c_{ij}}\theta_{i,r}\ast\theta_{j,l+1}+\theta_{j,l+1}\ast\theta_{i,r} \notag\\
=\frac{\delta_{i,\tau j}}{v_i-v_i^{-1}}\Big(\ttC^{r+1}h_j\psi_{j,l-r-1}-&v_i^{c_{ij}}\ttC^{l}h_i\psi_{i,r-l+1}-v_i^{c_{ij}}\ttC^{r}h_j\psi_{j,l-r+1}+\ttC^{l+1}h_i\psi_{i,r-l-1}\Big).\label{eq:psi theta rel comp-3}
\end{align}
\end{proposition}
\begin{proof}
Let us first assume that $r,l\geq 0$ and $r\geq m$. In this case $\theta_{j,r}$ cannot appear in $\Delta(\psi_{i,m})$ by degree consideration, and the following identities are immediate:
\begin{align}
\psi_{i,m}\ast\psi_{j,n}&=\sum_{s\geq 0}g_{i,\tau j}^{(s)}\psi_{i,m-s}\psi_{j,n-s}\ttC^s,\\
\psi_{i,m}\ast\theta_{j,r}&=\psi_{i,m}\theta_{j,r},\\
\theta_{j,r}\ast \psi_{i,m}&=\sum_{0\leq s\leq m}g_{i,\tau j}^{(s)}\theta_{j,r-s}\psi_{i,m-s}\ttC^s,\\
\theta_{r\delta+\alpha_i}\ast\theta_{l\delta+\alpha_j}&=\theta_{r\delta+\alpha_i}\theta_{l\delta+\alpha_j}+\frac{\delta_{i,\tau j}}{v_i-v_i^{-1}}\ttC^rh_j\psi_{j,l-r},
\end{align}
where we have used \eqref{coprod of root vec-1}, \eqref{coprod of root vec-2} and Lemma~\ref{lem: Hopf pair of psi_i}. The desired relations then follow from these. For the general case, we can apply the braid group symmetry $\widetilde{\TT}_{\bm{\varpi_i}}$ on \eqref{eq:psi theta rel comp-2} and \eqref{eq:psi theta rel comp-3}, noting that $\widetilde{\TT}_{\bm{\varpi_j}}(\psi_{i,m})=\widetilde{T}_{\bm{\varpi_j}}(\psi_{i,m})=\psi_{i,m}$ for any $j\in\I_0$.
\end{proof}

We define generating functions $\psi_i(z)$ and $\theta_j(w)$ by
\[\psi_i(z)=1+\sum_{m\geq 1}\psi_{i,m}z^m,\quad \theta_j(w):=\sum_{r\in\Z}\theta_{j,r}w^r.\]
Then Proposition~\ref{prop: psi theta rel comp} has the following equivalent form:
\begin{align}
\psi_i(z)\ast\psi_j(w)&=g_{i,\tau j}(zw\mathcal{C})\psi_i(z)\psi_j(w),\label{iHopf Dr1}\\
g_{i,\tau j}(zw\mathcal{C})\psi_i(z)\ast\theta_j(w)&=g_{ij}(zw^{-1})\theta_j(w)\ast\psi_i(z),\label{iHopf Dr2}\\
\theta_{i}(z)*\theta_{\tau i}(w)-\theta_{\tau i}(w)*\theta_{i}(z)&=\frac{\Delta(zw)}{v_i-v_i^{-1}} \big(\H_{\tau i}*\psi_{\tau i}(w)-\H_{i}*\psi_i(z)\big),
		 \text{ if } c_{i,\tau i}=0,
\\
(w-zv_i^{c_{ij}})\theta_i(z)\ast\theta_j(w)&-(v_i^{c_{ij}}w-z)\theta_j(w)\ast\theta_i(z)\notag\\
=\delta_{i,\tau j}v_i^{-\frac{1}{2}c_{i,\tau i}}\frac{\Delta(zw)}{v_i-v_i^{-1}}\Big((w-zv_i^{c_{ij}})&\H_j\ast\psi_j(w)-(v_i^{c_{ij}}w-z)\H_i\ast\psi_i(z)\Big).\label{iHopf Dr3}
\end{align}

At this point we can state the main result of this subsection, which provides another proof of Drinfeld presentation for iquantum groups.

\begin{theorem}
\label{thm:iDrinfeld-iHopf}
The isomorphism $\widetilde{\Phi}^\imath:\tB^\imath_\tau\to \tUi$ of Theorem~\ref{thm:iH=iQG} satisfies
\[\widetilde{\Phi}^\imath(\H_i)=\K_{\tau i},\quad \widetilde{\Phi}^\imath(\psi_{i,m})=\Theta_{i,m},\quad \widetilde{\Phi}^\imath(\theta_{i,r})=B_{i,r},\quad \widetilde{\Phi}^\imath(\mathcal{C})=\mathcal{C}.\]
\end{theorem}
\begin{proof}
We have already seen that $\widetilde{\Phi}^\imath(\H_i)=\K_{\tau i}$ and $\widetilde{\Phi}^\imath(\mathcal{C})=\mathcal{C}$. Also Lemma~\ref{lem: B_ir as PBW} implies $\widetilde{\Phi}^\imath(\theta_{i,r})=B_{i,r}$, so it remains to identify $\widetilde{\Phi}^\imath(\psi_{i,m})$ with $\Theta_{i,m}$. But this follows directly by comparing the Drinfeld presentation of $\tUi$ with the relations in Proposition~\ref{prop: psi theta rel comp}.
\end{proof}

\begin{remark}
  The isomorphism established in Theorem \ref{thm:iDrinfeld-iHopf} is also valid for the twisted affine case, which thereby allows us to derive the Drinfeld type presentation of $\tUi$ for the twisted affine type.
\end{remark}

\subsection{Current generators of $\tUi$}
In this subsection we analyze the image of $\theta_{i,r}$ and $\psi_{i,m}$ under the embedding $\widetilde{\xi}_\tau: \tB^\imath_\tau\to(\tB\otimes\tB)^\imath\cong\tU$. To this end we need the following identities:
\begin{equation}\label{eq:Delta^2 of theta_ir}
\begin{aligned}
\Delta^{(2)}(\theta_{r\delta+\alpha_i})&\in\theta_{r\delta+\alpha_i}\otimes 1\otimes 1+\sum_{\substack{s,t\geq 0\\ s+t\leq r}}\psi_{i,s}h_{(r-s)\delta+\alpha_i}\otimes\psi_{i,t}h_{(r-s-t)\delta+\alpha_i}\otimes\theta_{(r-s-t)\delta+\alpha_i}\\
&+\sum_{0\leq s\leq r}h_{s\delta+\alpha_i}\psi_{i,s}\otimes \theta_{(r-s)\delta+\alpha_i}\otimes 1+\tU^{\geq 0}(+\infty)_{[<0]}\otimes\tU^+(-\infty)\otimes 1\\
&+\tU^{\geq 0}(+\infty)\otimes\tU^{\geq 0}\otimes\tU^+(-\infty)_{[>\alpha_i]},
\end{aligned}
\end{equation}
\begin{equation}\label{eq:Delta^2 of psi_im}
\begin{aligned}
\Delta^{(2)}(\psi_{i,m})&\in\sum_{s,t\geq 0}\psi_{i,m-s-t}h_{(s+t)\delta}\otimes\psi_{i,s}h_{t\delta}\otimes\psi_{i,t}+\tU^+(\mathrm{im})\otimes\tU^{\geq 0}(+\infty)\otimes\tU^+(-\infty)_{[>0]}\\
&+\tU^{\geq 0}(+\infty)_{[<0]}\otimes\tU^{\geq 0}\otimes\tU^+(-\infty).
\end{aligned}
\end{equation}
where for any $\alpha\in\mathcal{R}_0$, we set
\[\tU^+_{[>\alpha]}=\sum_{\beta>\alpha,r\in\Z}\tU^+_{r\delta+\beta},\quad \tU^+_{[<\alpha]}=\sum_{\beta<\alpha,r\in\Z}\tU^+_{r\delta+\beta}\]
and similar for $\tU^+(\pm\infty)$ and $\tU^{\geq 0}(\pm\infty)$. 

Let us define 
\begin{align}
\vs_\delta:=\prod_{i\in\I} \vs_i^{a_i} \text{ if } \delta=\sum_{i\in\I} a_i\alpha_i.
\end{align}

\begin{lemma}
\label{lem:chi-C}
\begin{enumerate}
    \item If $\tau=\Id$, then we have \begin{align}
    \label{eq:chi-central-split}
    \chi(\ttC)=\prod_{i\in\I} v_i^{a_i}\vs_\delta=\prod_{i\in\I}(v_i\vs_i)^{a_i}.
    \end{align}
    \item If $\tau\neq\Id$, we have
    \begin{align}
    \label{eq:chi-central-qsplit}
        \chi(\ttC)=\begin{cases}
            -\vs_\delta& \text{ for }{\rm AIII}^{(\tau)}_{2n} \quad (n\geq1),
            \\
            v^2\vs_\delta & \text{ for }{\rm AIII}^{(\tau)}_{2n+1} \quad (n\geq1),
            \\
            v^{2n-4}\vs_\delta &\text{ for } {\rm DI}^{(\tau)}_n \quad (n\geq4),
            \\
            v^6\vs_\delta& \text{ for }{\rm EI}^{(\tau)}_6.
        \end{cases}
    \end{align}
\end{enumerate}
\end{lemma}

\begin{proof}
By definition, $\ttC=o(i)o(\tau i)h_\delta$ for some $i\in\II$, so it is enough to compute $\chi(h_\delta)$. The proof uses only elementary computations, which we will omit here.
\end{proof}

\begin{lemma}\label{lem: chi on theta_ir and psi_im}
For any $i\in\I_0$ and $r,m\geq 0$, we have 
\[\chi(\theta_{i,r})=0,\quad \chi(\psi_{i,m})=\begin{cases}
    1&\text{if $m=0$},\\
    0 &\text{ if }2\nmid m,
    \\
    \chi(\mathcal{C})^{\frac{m}{2}}(v_i^{\frac{m}{2}c_{i,\tau i}}-v_i^{(\frac{m}{2}-1)c_{i,\tau i}})& \text{ if $m>0$ and $2\mid m$}.
\end{cases}  \]
\end{lemma}
\begin{proof}
Consider the map $\varepsilon^\imath\circ\widetilde{\xi}_\tau:\tB\to\Q(v^{\frac12})$, where $\varepsilon^\imath$ is the counit on $(\tB\otimes\tB)^\imath$ defined by $\varepsilon^\imath(a\otimes b)=\varphi(a,S^{-1}(b))$. Since $S^{-1}(\tU^+_\eta)\subseteq\tU^{\geq 0}_{\eta}$, from \eqref{eq:Delta^2 of theta_ir} we find that
\begin{align}
\varepsilon^\imath\widetilde{\xi}_\tau(\theta_{r\delta+\alpha_i})=\chi(\theta_{r\delta+\alpha_i}).
\end{align}
It therefore suffices to show that $\varepsilon^\imath\widetilde{\xi}_\tau(\theta_{r\delta+\alpha_i})=\varepsilon^\imath(\widetilde{\TT}_{\bm{\varpi}_i}^{-r}(B_i))=0$ for any $r\geq 0$, which is immediate from the formulas of $\widetilde{\TT}_i$ in Theorem~\ref{thm:Ti}.

Next we consider $\chi(\psi_{i,m})$. For this we shall use the following equation for $r,l\geq 0$:
\[\chi(\theta_{i,r}\theta_{\tau i,l})=\frac{\varsigma_iv_i^{c_{i,\tau i}}}{v_i-v_i^{-1}}\chi(\mathcal{C})^l\chi(\psi_{\tau i,r-l}).\]
Combined with \eqref{PBW rel-3}, we then find
\begin{align*}
&\chi(\mathcal{C})^{l}\chi(\psi_{\tau i,r-l+1})-v_i^{c_{i,\tau i}}\chi(\mathcal{C})^{r+1}\chi(\psi_{i,l-r-1})\\
&=v_i^{c_{i,\tau i}}\chi(\mathcal{C})^{l+1}\chi(\psi_{i,r-l-1})-\chi(\mathcal{C})^{r}\chi(\psi_{\tau i,l-r+1}).
\end{align*}
Recall that $\chi$ is completely determined by its values at $\theta_i$ and $h_i$, and in particular we have $\chi(\tau x)=\chi(x)$ for any $x\in\tB$. Therefore, we obtain
\[\chi(\psi_{i,m+1})=v_i^{c_{i,\tau i}}\chi(\mathcal{C})\chi(\psi_{i,m-1})-\chi(\mathcal{C})^{m}\chi(\psi_{i,-m+1})\]
for any $m\geq 0$. From this we can then solve for the value of $\chi(\psi_{i,m})$.
\end{proof}

We can now state the following proposition about the embedding of $\psi_{i,m}$ and $\theta_{i,r}$ in $\tU$.

\begin{proposition}\label{emb of psi, theta estimate}
For $r,m\geq 0$, we have
\begin{align}
\widetilde{\xi}_\tau(\psi_{i,m})&\in\sum_{t,s\geq 0}\chi(\mathcal{C})^t\chi(\psi_{i,s})\psi_{\tau i,t}\otimes \psi_{i,m-t-s}h_{(t+s)\delta}+\tU^+(-\infty)_{[>0]}\otimes\tU^{\geq 0}(\mathrm{im})\\
&+\tU^+(-\infty)\otimes \tU^{\geq 0}(+\infty)_{[<0]},\\
\widetilde{\xi}_\tau(\theta_{i,r})&\in1\otimes\theta_{i,r}+\sum_{\substack{s,t\geq 0\\ s+t\leq r}}v_i^{c_{i,\tau i}}\chi(\mathcal{C})^{r-s-t}\chi(\psi_{i,s})\theta_{\tau i,r-s-t}\otimes h_{(r-t)\delta+\alpha_i}\psi_{i,t}\notag\\
&\hspace{5mm}+1\otimes\tU^{\geq 0}(+\infty)_{[<0]}+\tU^+(-\infty)_{[>\alpha_i]}\otimes\tU^{\geq 0}(+\infty).
\end{align}
\end{proposition}
\begin{proof}
This follows immediately from Lemma~\ref{lem: chi on theta_ir and psi_im} and \eqref{eq:Delta^2 of psi_im}, \eqref{eq:Delta^2 of theta_ir}.
\end{proof}

Let ${^{\textup{\texttt{Dr}}}\tU^{\geq 0}_+}$ denote the subspace of ${^{\textup{\texttt{Dr}}}\tU^{\geq 0}}$ consisting of elements with nonzero Drinfeld grading. Then we have the following result by viewing $\tUi\subseteq\tU$.

\begin{theorem}
\label{thm:cmpatible-imagine}
We have
\begin{align}
\label{eq:compatible-imagine}
\Theta_{i}(z)&-\frac{1-\chi(\mathcal{C})z^2C'^2}{1-v_i^{-c_{i,\tau i}}\chi(\mathcal{C})z^2C'^2}\psi_{\tau i}^+(\chi(\mathcal{C})zC')\psi_i^-(z^{-1})\in{^{\textup{\texttt{Dr}}}\tU^{\geq 0}_+}\llbracket z\rrbracket,
\\
\Delta(\Theta_{i}(z))&-\Theta_i(zC'_2)\otimes\psi_{\tau i}^+(z\chi(\mathcal{C})\mathcal{C}_1C'_2)\psi_i^-(z^{-1})\in\big(\tUD\otimes{^{\textup{\texttt{Dr}}}\tU^{\geq 0}_+}\big)\llbracket z\rrbracket,
\end{align}
where $\chi(\ttC)$ is given in \eqref{eq:chi-central-split}--\eqref{eq:chi-central-qsplit}. 
\end{theorem}

\begin{proof}
It is equivalent to considering $\widetilde{\xi}_\tau(\psi_i(z))$. 
From the definition of product on $(\tB\otimes\tB)^\imath$, it is easy to see that ${^{\textup{\texttt{Dr}}}\tU^{\geq 0}_+}$ contains both $\tU^+(-\infty)\otimes \tU^{\geq 0}(+\infty)_{[<0]}$ and $\tU^+(-\infty)_{[>0]}\otimes\tU^{\geq 0}(\mathrm{im})$, and we have $\Delta({^{\textup{\texttt{Dr}}}\tU^{\geq 0}_+})\subseteq \tUD\otimes{^{\textup{\texttt{Dr}}}\tU^{\geq 0}_+}$. The theorem then follows from Proposition~\ref{emb of psi, theta estimate}.
\end{proof}

\begin{remark}
The formula in \eqref{eq:compatible-imagine} was conjectured in an unpublished manuscript by Weiqiang Wang and Weinan Zhang, which is established by extensive computational arguments for split type A in \cite{P25}, split types B, C and D in \cite{LiP25}, and quasi-split type AIII in \cite{LiP26}. Strictly speaking,  \eqref{eq:compatible-imagine} is slightly different from the one obtained in \cite{P25,LiP25,LiP26}, which stems from the different interpretations of universal iquantum groups. 
\end{remark}

\section{{\rm i}Quantum affinization}
\label{sec:iQLA}

In this section, we assume that $\fg$ is any Kac-Moody algebra (not necessarily of affine type), with the generalized Cartan matrix (GCM) $C=(c_{ij})_{i,j\in\II}$. Let $\Z^{\II}:=\bigoplus_{i\in\II}\Z\alpha_i$ be the root lattice. Define $\Z^\I:=\Z\delta\oplus \Z^{\II}$ by a slight abuse of notation, where $\delta$ is just a formal symbol. 

Let $\tUD:=\tUD_v(\widehat{\fg})$ be the (untwisted) quantum affinization of $\tU_v(\fg)$ (with central charges), which is a natural generalization of affine quantum groups; see \cite{H05}.  More explicitly, $\tUD$ is  the $\Q(v^{\frac12})$-algebra with generators $X_{i,r}^{\pm}$, $K_i^{\pm1}$, $K_i'^{\pm1}$, $C^{\pm 1}$, $C'^{\pm 1}$, $\psi_{i,\pm m}^{\pm}$, $i\in\I_0$, $r\in\Z$, $m\geq 1$, subject to \eqref{Drge1a}--\eqref{Drge6}. We shall use the machinery of iHopf algebras to construct a coideal subalgebra $\tUiD:=\tUiD(\widehat{\fg})$ of the topological Hopf algebra $\tUD$, as a generalization of the affine iquantum groups ${^\textup{\texttt{Dr}}}\tUi(\widehat{\fg})$.

\subsection{Hopf algebra structure}
For the quantum affinization $\tUD=\tUD_v(\widehat{\fg})$, it has a natural $\Z^{\I}$-grading defined by
\begin{align}
\label{eq:degree-loop}
\deg X^\pm_{i,r}=r\delta\pm\alpha_i,\quad \deg \psi^\pm_{\pm m}=\pm m\delta,
\end{align}
and the degrees of other generators are $0$. Then $\tUD=\oplus_{\gamma\in\Z^\I}\tUD_\gamma$, where $\tUD_\gamma$ is the homogeneous subspace of degree $\gamma\in\Z^\I$. 
According to this gradation, we denote by $\tUDT$ the topological completion of $\tUD$, and $\tUD\widehat{\otimes}\tUD$ the completed tensor algebra of $\tUD\otimes\tUD$ constructed in \cite[\S1.2]{G07}; see also \cite{Dr88}. 

Recall the generating functions defined in \eqref{eq:GF-QG1}--\eqref{eq:GF-QG2}. We have the following proposition.

\begin{proposition}[\cite{Dr88,G07,Da24}]
The $\Q(v^{\frac12})$-algebra $\tUD$ is a topological Hopf algebra with the coproduct $\Delta$, the counit $\varepsilon$, and the antipode $S$ given by
\begin{align}
    &\Delta(K_i)=K_i\otimes K_i,\qquad \Delta(K_i')=K_i'\otimes K_i',
    \\
    &\Delta(C)=C\otimes C,\qquad \Delta(C')=C'\otimes C',
    \\
    &\Delta(X_{i}^{+}(z))=X_{i}^+(z)\otimes 1+K_i\psi_i^+(z)\otimes X_i^+(zC_1),
    \\
    &\Delta(X_{i}^{-}(z))=1\otimes X_{i}^-(z)+ X_{i}^-(zC'^{-1}_2)\otimes K_i'\psi_i^-(z),
    \\
    &\Delta(\psi_i^+(z))=\psi_i^+(z)\otimes\psi_i^+(zC_1),\\
    &\Delta(\psi_i^-(z))=\psi_i^-(zC'^{-1}_2)\otimes\psi_i^-(z),
    \\
    &S(K_i)=K_i^{-1},\quad S(C)=C^{-1},\quad S(K_i')=K_i'^{-1},\quad S(C')=C'^{-1},
    \\
    &S(X_{i}^+(z))=-K_i^{-1}\psi_i^+(zC^{-1})^{-1}X_i^+(zC^{-1}),
    \\
    &S(X_{i}^-(z))=-X_i^-(zC')K_i'^{-1}\psi_i^-(zC')^{-1},
    \\
    &S(\psi_i^+(z))=\psi_i^+(z)^{-1},\quad S(\psi_{i}^-(z))=\psi_i^-(z)^{-1};
    \\
    &\varepsilon(C)=\varepsilon(C')=\varepsilon(K_i)=\varepsilon(K_i')=1,\quad \varepsilon(X_{i}^{\pm}(z))=0,\quad \varepsilon(\psi^\pm(z))=1.
\end{align}
Here $C_1=C\otimes 1$, $C_2=1\otimes C$, $C_1'=C'\otimes 1$, $C_2'=1\otimes C'$.
\end{proposition}

Let ${^{\textup{\texttt{Dr}}}\tB}:={^{\textup{\texttt{Dr}}}\tB}_v(\widehat{\fg})$ be the algebra generated by $\theta_{i,r}$, $\psi_{i,m},h_i^{\pm 1},C^{\pm 1}$ for $i\in\II$, $r\in\Z$, $m\geq1$, subject to the defining relations:
\begin{align}
\label{eq:B1}
 	&	C\text{ is  central},\quad [h_i,h_j]  =0,
		\quad [h_i,\psi_j(z)] =0,
        \quad [\psi_i(z),\psi_j(w)] = 0,
		\\
        \label{eq:B2}
		&h_i\theta_j(w)=v_i^{c_{ij}}\theta_j(w)h_i,
		\quad 
		\psi_i(z)\theta_j(w)=g_{ij}(zw^{-1})\theta_j(w)\psi_i(z),
		\\
        \label{eq:B3}
		&(w-zv_i^{c_{ij}})\theta_i(z)\theta_j(w)=(v_i^{ c_{ij}}w-z)\theta_j(w)\theta_i(z),
		\\
	\label{eq:B4}
		&\sym_{z_1,\dots,z_r}\sum_{s=0}^{r}(-1)^s \qbinom{r}{s}_i
		\theta_i(z_1)\cdots
		\theta_i(z_s) \theta_j(w) \theta_i(z_{s+1}) \cdots \theta_i(z_r) =0, \\\notag
        &\hspace{7cm}\text{ for } i\neq j, r=1-c_{ij}.
	\end{align} 
    Here 
\[\psi_i(z)=1+\sum_{m\geq 1}\psi_{i,m}z^m,\quad \theta_j(w):=\sum_{r\in\Z}\theta_{j,r}w^r.\]
Due to the triangular decomposition of $\tUD$ given in \cite{H05}, the algebra ${^{\textup{\texttt{Dr}}}\tB}$ is isomorphic to the (Hopf) subalgebra of ${^{\textup{\texttt{Dr}}}\tU}$ generated by $X_{i,r}^+,\psi_{i,m}^+,K_i^{\pm 1},C^{\pm 1}$, $i\in\II$ $r\in\Z$. 

We also need some variants of ${^{\textup{\texttt{Dr}}}\tB}$. Let ${^\textup{\texttt{Dr}}}\widetilde{\mathbf{B}}''$ be the algebra freely generated by $\theta_{i,r}$, $\psi_{i,m},h_i^{\pm1},C^{\pm 1}$ for $i\in\II$, $r\in\Z$, $m\geq1$. Both  $^{\textup{\texttt{Dr}}}\tB$ and ${^\textup{\texttt{Dr}}}\widetilde{\mathbf{B}}''$ are $\Z^{\I}$-graded algebras; see \eqref{eq:degree-loop}. We denote by ${^\textup{\texttt{Dr}}}\widetilde{\mathbf{B}}^{\texttt{T}}$ and ${^\textup{\texttt{Dr}}}\widetilde{\mathbf{B}}^{''\texttt{T}}$ their topological completion, respectively.
\begin{lemma}
  The $\Q(v^{\frac12})$-algebra ${^\textup{\texttt{Dr}}}\widetilde{\mathbf{B}}^{''\textup{\texttt{T}}}$ is a Hopf algebra with the coproduct $\Delta$, the counit $\varepsilon$ and the antipode $S$ given by
  \begin{align}
  \label{eq:H1}
      \Delta(h_i)&=h_i\otimes h_i,\quad \Delta(C)=C\otimes C,\quad  \Delta(\psi_i(z))=\psi_i(z)\otimes\psi_{i}(zC_1),
      \\
        \label{eq:H2}
      \Delta(\theta_i(z))&=\theta_i(z)\otimes 1+h_i\psi_i(z)\otimes \theta_i(zC_1),
      \\
        \label{eq:H3}
\varepsilon(C)&=1=\varepsilon(h_i),\quad \varepsilon(\theta_i(z))=0,\quad\varepsilon(\psi_i(z))=1.
\\ 
S(h_i)&=h_i^{-1},\quad S(C)=C^{-1},\quad S(\psi_{i}(z))=\psi_i(z)^{-1},\\\label{eq:HS}
S(\theta_i(z))&=-h_i^{-1}\psi_i(zC^{-1})^{-1}\theta_i(zC^{-1}).
  \end{align}
\end{lemma}

\begin{proof}
    Since ${^\textup{\texttt{Dr}}}\widetilde{\mathbf{B}}^{''\texttt{T}}$ is a free algebra, the maps $\Delta$, $\varepsilon$, $S$ are well-defined. Verifying that ${^\textup{\texttt{Dr}}}\widetilde{\mathbf{B}}^{''\texttt{T}}$ forms a Hopf algebra is straightforward, and we omit the details of the proof herein. 
\end{proof}

Let ${^{\textup{\texttt{Dr}}}\tB'}$ be the algebra generated by $\theta_{i,r}$, $\psi_{i,m},h_i^{\pm1},C^{\pm 1}$ for $i\in\II$, $r\in\Z$, $m\geq1$, subject to \eqref{eq:B1}--\eqref{eq:B3} (except the Serre relations \eqref{eq:B4}). Its topological completion is denoted by ${^{\textup{\texttt{Dr}}}\tB^{'\texttt{T}}}$.

\begin{lemma}
  The $\Q(v^{\frac12})$-algebra ${^\textup{\texttt{Dr}}}\widetilde{\mathbf{B}}^{'\texttt{T}}$ is a topological Hopf algebra with the coproduct $\Delta$, the counit $\varepsilon$ and the antipode $S$ given as in \eqref{eq:H1}--\eqref{eq:HS}.
\end{lemma}

\begin{proof}
It is direct to prove that the maps $\Delta$, $\varepsilon$ and $S$ on ${^\textup{\texttt{Dr}}}\widetilde{\mathbf{B}}^{''\texttt{T}}$ preserve the relations 
    \eqref{eq:B1}--\eqref{eq:B3}.
\end{proof}    

Let us denote by $\SS_{ij}(z_1,\dots,z_r|w)$ the LHS of \eqref{eq:B4}, for $r=1-c_{ij}$, which is viewed in ${^\textup{\texttt{Dr}}}\widetilde{\mathbf{B}}^{'\texttt{T}}$.

\begin{lemma}\label{coproduct of Serre}
   In ${^\textup{\texttt{Dr}}}\widetilde{\mathbf{B}}^{'\texttt{T}}$, we have
   \begin{align}
       \Delta(\SS_{ij}(z_1,\dots,z_r|w))=&\SS_{ij}(z_1,\dots,z_r|w)\otimes 1\\\notag
       &+ h_i^rh_j\psi_i(z_1)\cdots\psi_i(z_r) \psi_{j}(w)\otimes\SS_{ij}(z_1C_1,\dots,z_rC_1|wC_1),
   \end{align}
   where $r=1-c_{ij}$
\end{lemma}

\begin{proof}
We know that  ${^\textup{\texttt{Dr}}}\widetilde{\mathbf{B}}'$ is $\Z^{\I}$-graded:
\[
{^\textup{\texttt{Dr}}}\widetilde{\mathbf{B}}'=\bigoplus_{\mu\in\Z^\I} {^{\textup{\texttt{Dr}}}\tB'_\mu}.\]
Following \cite{Da24}, we consider another gradation. 
In this case ${^\textup{\texttt{Dr}}}\widetilde{\mathbf{B}}'$ is  $\Z^{\II}$-graded:
\[{^\textup{\texttt{Dr}}}\widetilde{\mathbf{B}}'=\bigoplus_{\mu\in\Z^{\II}} {^{\textup{\texttt{Dr}}}\tB'_{[\mu]}}, \text{ where }  {^{\textup{\texttt{Dr}}}\tB'_{[\mu]}}:=\bigoplus_{r\in\Z}{^{\textup{\texttt{Dr}}}\tB'_{r\delta+\mu}}.\]
Denote by  ${^{\textup{\texttt{Dr}}}\tB^{'\texttt{T}}_{[\mu]}}$ its topological completion.

Note that $
{^{\textup{\texttt{Dr}}}\tB'_{[\mu]}}\neq0$ only if $\mu\in\N^{\II}$. There is a natural partial order on $\Z^{\II}$, i.e., $\gamma\leq \mu$ if $\mu-\gamma\in\N^{\II}$. 
We set
$${^{\textup{\texttt{Dr}}}\tB'_{[<\mu]}}:=\bigoplus_{\gamma<\mu}  {^{\textup{\texttt{Dr}}}\tB'_{[\gamma]}},\quad \forall \mu\in\N^{\II}.$$

From \cite{Dr87,G07,Da24}, we know that $\Delta:{^{\textup{\texttt{Dr}}}\tB}\rightarrow {^{\textup{\texttt{Dr}}}\tB}\widehat{\otimes} {^{\textup{\texttt{Dr}}}\tB}$ is an algebra homomorphism induced by the algebra homomorphism 
$\Delta:{^{\textup{\texttt{Dr}}}\tB'}\rightarrow {^{\textup{\texttt{Dr}}}\tB'}\widehat{\otimes} {^{\textup{\texttt{Dr}}}\tB'}$. 
Let $\mathcal{I}$ be the two-sided (graded) ideal of ${^{\textup{\texttt{Dr}}}\tB'}$ generated by all $\SS_{ij}(z_1,\dots,z_r|w)$ ($i,j\in\II$). Then we know
$${^{\textup{\texttt{Dr}}}\tB}\widehat{\otimes} {^{\textup{\texttt{Dr}}}\tB}={^{\textup{\texttt{Dr}}}\tB'}\widehat{\otimes} {^{\textup{\texttt{Dr}}}\tB'}/\big(\mathcal{I}\widehat{\otimes} {^{\textup{\texttt{Dr}}}\tB'}+ {^{\textup{\texttt{Dr}}}\tB'}\widehat{\otimes}\mathcal{I} \big),$$
which induces 
$\Delta(\SS_{ij}(z_1,\dots,z_r|w))\in \mathcal{I}\widehat{\otimes} {^{\textup{\texttt{Dr}}}\tB'}+ {^{\textup{\texttt{Dr}}}\tB'}\widehat{\otimes}\mathcal{I}$. 

On the other hand, note that $\SS_{ij}(z_1,\dots,z_r|w)\in {^{\textup{\texttt{Dr}}}\tB'}_{[r\alpha_i+\alpha_j]}$. By definition, we have, in ${^{\textup{\texttt{Dr}}}\tB'}\widehat{\otimes} {^{\textup{\texttt{Dr}}}\tB'}$, 
    \begin{align*}
         &\Delta(\SS_{ij}(z_1,\dots,z_r|w))
         \\&=\sym_{z_1,\dots,z_r}\sum_{a=0}^{r} (-1)^a \qbinom{r}{a}_i
		\Delta(\theta_i(z_1))\cdots
		\Delta(\theta_i(z_a)) \Delta(\theta_j(w)) 
        \Delta( \theta_i(z_{a+1})) \cdots \Delta(\theta_i(z_r))
        \\
        &=\SS_{ij}(z_1,\dots,z_r|w)\otimes 1+ h_i^rh_j\psi_{i}(z_1)\cdots \psi_{i}(z_r)\psi_{j}(w)\otimes \SS_{ij}(z_1C_1,\dots,z_rC_1|wC_1)
        \\
        &\quad+{^{\textup{\texttt{Dr}}}\tB'_{[<r\alpha_i+\alpha_j]}}\widehat{\otimes} {^{\textup{\texttt{Dr}}}\tB'_{[<r\alpha_i+\alpha_j]}}.
    \end{align*}
However, $\mathcal{I}\cap {^{\textup{\texttt{Dr}}}\tB'_{[<r\alpha_i+\alpha_j]}}=0$. So we have 
$$\big({^{\textup{\texttt{Dr}}}\tB'_{[<r\alpha_i+\alpha_j]}}\widehat{\otimes} {^{\textup{\texttt{Dr}}}\tB'_{[<r\alpha_i+\alpha_j]}}\big)\cap \big(\mathcal{I}\widehat{\otimes} {^{\textup{\texttt{Dr}}}\tB'}+ {^{\textup{\texttt{Dr}}}\tB'}\widehat{\otimes}\mathcal{I} \big)=0,$$
and the desired formula follows.
\end{proof}

 
 \subsection{Weak Hopf pairings}
 
In this subsection, we recall the weak Hopf pairing on the topological Hopf algebra ${^{\textup{\texttt{Dr}}}\tB}$.  
A weak Hopf pairing $\langle -,-\rangle$ for  ${^{\textup{\texttt{Dr}}}\tB}$ 
 is a bilinear form that satisfies the Hopf pairing conditions (H1)–(H4), with only finitely many nonzero summands appearing in the sums over $\Delta(a)$ and $\Delta(b)$ in conditions (H2)–(H4).
 
\begin{proposition}[{\cite[Proposition 3]{G07}}]
There exists a weak Hopf pairing $\langle-,-\rangle$ on ${^{\textup{\texttt{Dr}}}\tB}$ such that 
\begin{align}
   &\langle h_i,h_j\rangle=v_i^{c_{ij}},\quad
   \langle C,C\rangle=1,
    \quad \langle C,h_i\rangle=1=\langle h_i,C\rangle,
   \\
   &\langle C,\theta_{i,r}\rangle=\langle C,\psi_{i,m}\rangle=0,\quad \langle \theta_{i,r},C\rangle=\langle \psi_{i,m},C\rangle=0,
   \\
   &\langle \psi_{i,m},\theta_{j,r}\rangle=\langle \theta_{j,r},\psi_{i,m}\rangle=0,
   \quad\langle \psi_{i,m},\psi_{j,n}\rangle=\delta_{m,n}g_{ij}^{(n)},
   \\
   &\langle \theta_{i,r},\theta_{j,t}\rangle=\frac{\delta_{ij}\delta_{r,t}}{v_i-v_i^{-1}}.
\end{align}
See \eqref{eq:gijn} for the definition of $g_{ij}^{(n)}$. 
\end{proposition}

Consider the $\Z^{\II}$-gradation of ${}^{\textup{\texttt{Dr}}}\widetilde{\mathbf{B}}$: 
\begin{align}
\label{eq:wt}
\wt(\theta_{i,r})=\alpha_i,\quad \wt(\psi_{i,m})=0=\wt(h_i)=\wt(C).
\end{align}
Then the weak Hopf pairing $\langle-,-\rangle$ is $\Z^{\II}$-graded, that is 
$\langle x,y\rangle=0$ if $\wt(x)\neq \wt(y)$; see \cite[Proposition 3]{G07}.

\subsection{iQuantum affinization}

Given an involution $\tau$ of the GCM $C=(c_{ij})$, let us define the (untwisted) iquantum affinization $\tUiD:={}^{\textup{\texttt{Dr}}}\tUi_v(\widehat{\fg})$ as a generalization of affine iquantum groups in Drinfeld type presentation, which is the $\Q(v^{\frac12})$-algebra generated by the elements $B_{i,l}$, $\Theta_{i,m}$, $\bK_i^{\pm1}$, $\ttC^{\pm1}$, where $i\in\II$, $l\in\Z$ and $m>0$. In order to formulate the defining relations, we first define the idivided powers of the root vectors; cf. \cite{BW18b,CLW18}.

For $i\in\I_0$ with $\tau i\neq i$, imitating Lusztig's divided powers, we define the idivided power of $B_{i,l}$ to be 
$$B_{i,l}^{(m)}:=\frac{B_{i,l}^m}{[m]_i^!},\quad \forall m\geq0.$$
For $i\in\I_0$ with $\tau i=i$, we define the idivided powers of $B_{i,l}$ to be 
\begin{align}
    B_{i,l;\odd}^{(m)}=\frac{1}{[m]_i^!}\begin{cases} B_{i,l}\prod_{j=1}^k (B_{i,l}^2+\bK_i\ttC^l[2j-1]_i^2 ) & \text{if }m=2k+1,\\
\prod_{j=1}^k (B_{i,l}^2+\bK_i\ttC^l[2j-1]_i^2) &\text{if }m=2k; 
\end{cases}
  \label{eq:iDPodd}\\
B_{i,l;\ev}^{(m)}= \frac{1}{[m]_i^!}\begin{cases}  B_{i,l}\prod_{j=1}^k (B_{i,l}^2+\bK_i\ttC^l[2j]_i^2 ) & \text{if }m=2k+1,\\
\prod_{j=1}^{k} (B_{i,l}^2+\bK_i\ttC^l[2j-2]_i^2) &\text{if }m=2k.
\end{cases}
 \label{eq:iDPev}.
\end{align}
Here $\ev,\odd\in\Z/2\Z$. 

Denote
\[
(a;x)_0=1, \qquad (a;x)_n =(1-a)(1-ax)  \cdots (1-ax^{n-1}), \quad \forall n\ge 1.
\]
Then the defining relations of $\tUiD$ are \eqref{qsiDR0}--\eqref{qsiDR4}, and 
\begin{align}
\label{eq:relation1}
\sum_{s=0}^{r}(-1)^sB_{i,k}^{(r-s)}B_{j,l}B_{i,k}^{(s)}&=0, \qquad\quad\text{ for } \tau i\neq i\neq j\neq \tau i, r=1-c_{ij},
\\
\sum_{n=0}^{1-c_{i,\tau i}} (-1)^nB_{i,k}^{(n)}B_{\tau i,k}B_{i,k}^{(1-c_{i,\tau i}-n)}& =\frac{v_i^{-\frac{c_{i,\tau i}}{2}}}{v_i-v_i^{-1}}
\label{eq:relation5}     \\
 \qquad\times \Big(v_i^{c_{i,\tau i}} (v_i^{-2};v_i^{-2})_{-c_{i,\tau i}} B_{i,k}^{(-c_{i,\tau i})}& \bK_{i} 
   -(v_i^{2};v_i^{2})_{-c_{i,\tau i}}B_{i,k}^{(-c_{i,\tau i})} \bK_{\tau i} \Big)\ttC^k,
\notag\text{ if } \tau i \neq i,
 \notag
\\
    \sum_{n=0}^{1-c_{ij}} (-1)^n  B_{i,k;\overline{c_{ij}}+\overline{p_i}}^{(n)}B_{j,l} B_{i,k;\overline{p}_i}^{(1-c_{ij}-n)} &=0,\quad  \text{ for any }\ov{p}_i\in\Z/2\Z, \text{ if }\tau i=i\neq j.
\label{eq:relation2}
\end{align}
When $C$ is an affine Cartan matrix, we see that $\tUiD$ is the algebra $\tUiD$ defined in Definition \ref{def:iDRA1}.
Moreover, the relations \eqref{eq:relation1}--\eqref{eq:relation2} can be reduced to the special cases for $k=0=l$ under the assumption of \eqref{qsiDR0}--\eqref{qsiDR4}. 

Consider the quantum group $\tU_v(\fg)$, which is generated by $E_i,F_i,K_i^{\pm1}, K_i'^{\pm1}$ ($i\in\II$). Let $\tUi_v(\fg)$ be the iquantum group associated to the Satake diagram $(\II,\tau)$, with generators \begin{equation}
		B_i= F_i+\varsigma_iE_{\tau i} \tK_i',
		\qquad \tk_i= \tK_i \tK_{\tau i}', \quad \forall i \in \I_0.
	\end{equation}
    where $\varsigma_i\in\Q(v^{\frac{1}{2}})^\times$ is a system of parameters satisfying $\varsigma_i=\varsigma_{\tau i}$; see \eqref{eq:Bi}. 
Using the Borel subalgebra $\widetilde{\mathbf{B}}(\fg)$ of the quantum group $\tU_v(\fg)$, we can construct the iHopf algebra $\widetilde{\mathbf{B}}(\fg)^\imath_\tau$; see \S\ref{subsec:iQG-iHopf}. By Theorem \ref{thm:iH=iQG}, we have an algebra isomorphism
    \begin{align} &\widetilde{\Phi}^\imath:\tB(\fg)^\imath_\tau\longrightarrow \tUi_v(\fg), \\
        \theta_i \mapsto& B_i,\qquad h_i\mapsto \varsigma_i\tk_{\tau i},\quad \forall i\in\I.
    \end{align}

Return to the setting of $\tU=\tUD_v(\widehat{\fg})$. 
From the involution $\tau$ of the GCM $C=(c_{ij})$, we can define an automorphism of ${^{\textup{\texttt{Dr}}}\tB}$ via
\[h_i^{\pm 1}\mapsto h_{\tau i}^{\pm 1},\quad C^{\pm 1}\mapsto C^{\pm 1},\quad \theta_{i,r}\mapsto\theta_{\tau i,r}.\]
Following \cite{CLPRWa}, we define ${^{\textup{\texttt{Dr}}}\tB^\imath_\tau}$ to be the $\Q(v^{\frac12})$-algebra defined on the same vector space as ${^{\textup{\texttt{Dr}}}\tB}$ equipped with a new multiplication: 
\begin{align}\label{star product}
    a\ast b:=\sum \langle \tau b_{(2)},a_{(1)}\rangle\cdot a_{(2)}b_{(1)},\quad \forall a,b\in {^{\textup{\texttt{Dr}}}\tB^\imath},
\end{align}
where $\Delta(a)=\sum a_{(1)}\otimes a_{(2)}$, $\Delta(b)=\sum b_{(1)}\otimes b_{(2)}$.
Then ${^{\textup{\texttt{Dr}}}\tB^\imath_\tau}$ is called the iHopf algebra associated with ${^{\textup{\texttt{Dr}}}\tB}$ and $\tau$. 

\begin{lemma}
\label{lem:iHopf-iso}
    There exists a homomorphism of algebras
  \begin{gather}
  \widetilde{\phi}:\widetilde{\mathbf{B}}(\fg)^\imath_\tau\longrightarrow {^{\textup{\texttt{Dr}}}\tB^\imath_\tau}
  \\
  \theta_i\mapsto\theta_{i,0},\quad h_i\mapsto h_i,\quad \forall i\in\II.
  \end{gather}
\end{lemma}

\begin{proof}
There exists a homomorphism of algebras $\tU_v(\fg)\rightarrow \tUD_v(\widehat{\fg})$, $E_i\mapsto X^+_{i,0}$, $F_i\mapsto X^-_{i,0}$, $K_i\mapsto K_i$, $K_i'\mapsto K_i'$, for $i\in\II$; see \cite[Theorem 2.1]{Jin98}. Using this, we can induce a homomorphism of algebras: $\widetilde{\phi}:\widetilde{\mathbf{B}}(\fg)\rightarrow {^{\textup{\texttt{Dr}}}\tB}$, which maps $\theta_i\mapsto\theta_{i,0}$,$ h_i\mapsto h_i$, for all $i\in\II$.

Below we prove that the map $\widetilde{\phi}:\widetilde{\mathbf{B}}(\fg)\rightarrow {^{\textup{\texttt{Dr}}}\tB}$
induces a homomorphism of iHopf algebras 
$\widetilde{\phi}:\widetilde{\mathbf{B}}(\fg)^\imath_\tau\longrightarrow {^{\textup{\texttt{Dr}}}\tB^\imath_\tau}$. For this it is enough to show that $\widetilde{\phi}(x*y)=\widetilde{\phi}(x)*\widetilde{\phi}(y)$ for any $x,y\in\widetilde{\mathbf{B}}(\fg)$.

By induction it suffices to consider the case where $x=\theta_i$ or $h_i$. We also assume $y$ to be a monomial, i.e., $y=\theta_{j_1}\cdots \theta_{j_r}h_\beta$. Set $\gamma=\deg(y)=\alpha_{j_1}+\cdots +\alpha_{j_r}$. Then in $\widetilde{\mathbf{B}}(\fg)$, we have
\begin{align*}
    \Delta(\theta_i)&=\theta_i\otimes 1+h_i\otimes \theta_{i},
    \\
    \Delta(y)&=y\otimes h_\beta+ h_{\beta+\gamma}\otimes y+\sum_{i\in\II}y_ih_{\alpha_i+\beta}\otimes \theta_ih_\beta +({\rm rest}),
\end{align*}
where the ``rest" involves terms of the form $w_1\otimes w_2h_\beta$ with $w_2\notin\{\theta_{i}\mid i\in\II\}$; see \cite[\S6.14]{Jan96}. 
Then we have in $\widetilde{\mathbf{B}}(\fg)^\imath_\tau$
\begin{align}
    \theta_i*y=\theta_i\cdot y+\frac{1}{v_i-v_i^{-1}}y_{\tau i}h_{\tau i+\beta}.
\end{align}

On the other hand, in ${^{\textup{\texttt{Dr}}}\tB}$,
\begin{align*}
\Delta(\theta_{i,0})&=\theta_{i,0}\otimes 1+h_i\otimes\theta_{i,0}+ \sum_{l\geq1} h_iC^{-l}\psi_{i,l}\otimes \theta_{i,-l},
    \\
    \Delta(\widetilde{\phi}(y))&=\widetilde{\phi}(y)\otimes h_\beta+ h_{\beta+\gamma}\otimes \widetilde{\phi}(y)+\sum_{i\in\II}\widetilde{\phi}(y_i)h_{\alpha_i+\beta}\otimes \theta_{i,0}h_\beta +({\rm rest}),
\end{align*}
where the ``rest" involves terms of the form $w_1\otimes w_2h_\beta$ with $w_2\notin\{\theta_{i,0}\mid i\in\II\}$. 
Since the weak Hopf pairing is $\Z^{\II}$-graded (see \eqref{eq:wt}) and 
$$\langle \theta_{i,r},\theta_{j,t}\rangle=\frac{\delta_{ij}\delta_{r,t}}{v_i-v_i^{-1}},$$ 
we have in ${^{\textup{\texttt{Dr}}}\tB^\imath_\tau}$
\begin{align*}
    \theta_i*\widetilde{\phi}(y)&=\theta_i\cdot \widetilde{\phi}(y)+\frac{1}{v_i-v_i^{-1}}\widetilde{\phi}(y_{\tau i})h_{\tau i+\beta}\\
    &=\widetilde{\phi}(\theta_i\cdot y)+\frac{1}{v_i-v_i^{-1}}\widetilde{\phi}\big(y_{\tau i}h_{\tau i+\beta}\big).
\end{align*}
This proves the claim for $x=\theta_i$. The proof for $x=h_i$ is immediate. The proof is completed.
\end{proof}

In ${}^{\textup{\texttt{Dr}}}\widetilde{\mathbf{B}}$ (and also $\widetilde{\mathbf{B}}(\fg)$), we set 
\[\H_\alpha:=v^{\frac{1}{2}(\alpha,\tau\alpha)}h_\alpha.\]
The following theorem is the main result of this subsection.

\begin{theorem}\label{thm:iH=iQG-Dr}
    We have an algebra isomorphism 
    \begin{align*}
            &\qquad\widetilde{\Xi}: {^{\textup{\texttt{Dr}}}\tB^\imath_\tau}\longrightarrow \tUiD,\\
        \H_i &\mapsto\K_{\tau i},\quad C\mapsto \mathcal{C},
        \quad
        \theta_{i,r}\mapsto B_{i,r},\quad
        \psi_{i,m}\mapsto \Theta_{i,m},\quad \forall i\in\II, r\in\Z,m>0.
    \end{align*}
\end{theorem}

\begin{proof}
Using the definition of coproduct and bilinear pair, one can see that
\begin{gather*}
h_i\ast \theta_j(w)=h_i\theta_j(w),\quad \theta_j(w)\ast h_i=v^{c_{i,\tau j}}\theta_j(w)h_i,\\
\psi_i(z)\ast\psi_j(w)=g_{\tau i, j}(zwC)\psi_i(z)\psi_j(w),\\
\psi_i(z)\ast \theta_j(w)=\psi_i(z)\theta_j(w),\quad \theta_j(w)\ast \psi_i(z)=g_{i,\tau j}(zwC)\theta_j(w)\psi_i(z),\\
\theta_i(z)\ast \theta_j(w)=\theta_i(z)\theta_j(w)+\delta_{\tau i,j}\frac{\Delta(zw)}{v_i-v_i^{-1}}h_j\psi_j(w).
\end{gather*}
These then imply the following relations for the algebra ${^{\textup{\texttt{Dr}}}\tB_\tau^\imath}$:
\begin{align}
 	    C\text{ is  central},\quad &[h_i,h_j]=0,\quad [h_i,\psi_j(w)]=0,\label{iBeq-1}\\
        h_i\ast \theta_j(w)&=v^{c_{ij}-c_{\tau i,j}}\theta_j(w)\ast h_i,\label{iBeq-2}\\
        \psi_i(z)\ast\psi_j(w)&=\psi_j(w)\ast\psi_i(z),\label{iBeq-3}\\
        \theta_j(w)\ast\psi_i(z)&=\frac{g_{\tau i,j}(zwC)}{g_{ij}(zw^{-1})}\psi_i(z)\ast \theta_j(w),\label{iBeq-4}\\
        (w-v_i^{c_{ij}}z)\theta_i(z)\ast \theta_{\tau i}(w)&-(v_i^{c_{ij}}w-z)\theta_{\tau i}(w)\ast\theta_i(z)\label{iBeq-5}\\
        =\frac{\Delta(zw)}{v_i-v_i^{-1}}\Big((w-v_i^{c_{ij}}z)&h_{\tau i}\ast\psi_{\tau i}(w)-(v_i^{c_{ij}}w-z)h_i\ast\psi_i(z)\Big),\label{iBeq-5}
\end{align}
where we define $\Delta(z)=\sum_{k\in\Z}C^kz^k$. 

It then remains to verify the reduced version of Serre relations for ${^{\textup{\texttt{Dr}}}\tB_\tau^\imath}$. Using Lemma \ref{lem:iHopf-iso} and the isomorphism $\widetilde{\Phi}^\imath:\tB(\fg)^\imath_\tau\rightarrow \tUi_v(\fg)$, we can transfer the iSerre relations of $\tUi_v(\fg)$ (see \cite{CLW18,LW22a}) to ${^{\textup{\texttt{Dr}}}\tB_\tau^\imath}$. In fact, iSerre relations of $\tB(\fg)^\imath_\tau\equiv\tUi_v(\fg)$  are
\begin{align}
\sum_{s=0}^{r}(-1)^s\theta_{i}^{(r-s)}\theta_j\theta_{i}^{(s)}&=0, \qquad\quad\text{ for } \tau i\neq i\neq j\neq \tau i, r=1-c_{ij},
\\
\sum_{n=0}^{1-c_{i,\tau i}} (-1)^n\theta_{i}^{(n)}\theta_{\tau i}\theta_{i}^{(1-c_{i,\tau i}-n)}& =\frac{v_i^{-\frac{c_{i,\tau i}}{2}}}{v_i-v_i^{-1}}
\\
 \qquad\times \Big(v_i^{c_{i,\tau i}} (v_i^{-2};v_i^{-2})_{-c_{i,\tau i}} \theta_{i}^{(-c_{i,\tau i})}& \H_{i} 
   -(v_i^{2};v_i^{2})_{-c_{i,\tau i}}\theta_{i}^{(-c_{i,\tau i})} \H_{\tau i} \Big),
\notag\text{ if } \tau i \neq i,
 \notag
\\
    \sum_{n=0}^{1-c_{ij}} (-1)^n  \theta_{i,\overline{c_{ij}}+\overline{p_i}}^{(n)}\theta_{j} \theta_{i,\overline{p}_i}^{(1-c_{ij}-n)} &=0,\quad  \text{ for any }\ov{p}_i\in\Z/2\Z, \text{ if }\tau i=i\neq j.
\end{align}
Here 
the idivided powers of $\theta_i$ are defined as in \eqref{eq:iDPodd}--\eqref{eq:iDPev}; see \cite{BW18b,CLW18,LW22a}. This proves that the corresponding iSerre relations of \eqref{eq:relation1}--\eqref{eq:relation2} for $k=0=l$ hold in ${^{\textup{\texttt{Dr}}}\tB_\tau^\imath}$, which are enough. 

It is clear that $\widetilde{\Xi}$ is surjective. For the injectivity, one may apply the same argument as in \cite[Proof of Theorem 3.13]{LW21b} by passing to the associated graded algebras, where $\gr(\tUiD)$ and $\gr({^{\textup{\texttt{Dr}}}\tB_\tau^\imath})$ both reduce to ${^{\textup{\texttt{Dr}}}\tB}$ and $(\widetilde{\Xi})^{\gr}$ becomes the identity map. We skip the details.
\end{proof}

As an analog of quantum groups, the general iSerre relations also hold in ${^{\textup{\texttt{Dr}}}\tB^\imath_\tau}$ (equivalently in ${^{\textup{\texttt{Dr}}}\tUi_v(\widehat{\fg})}$), with the leading term
\begin{align}
   \SS_{i,j}^\imath(w_1,\dots,w_r|z)&:=\sym_{w_1,\dots,w_r}\sum_{s=0}^{r}(-1)^s \qbinom{r}{s}_i \\
   &\notag\times
		\theta_i(w_1)*\cdots
		*\theta_i(w_s) *\theta_j(z) *\theta_i(w_{s+1})* \cdots *\theta_i(w_r),\quad r=1-c_{ij}.
\end{align}
These relations can be explicitly computed using
\[\theta_{i,r}\ast \theta_{j,l}=\theta_{i,r}\theta_{j,l}+\delta_{\tau i,j}\frac{1}{v_i-v_i^{-1}}h_jC^r\psi_{j,l-r}.\]
However, it is very difficult to write down explicit formulas. To conclude this subsection, we give them under the assumption of $c_{ij}c_{ji}\leq3$ for $i\neq j\in\II$, which includes all the affine Cartan matrices.

The following iSerre relations hold in ${^{\textup{\texttt{Dr}}}\tB_\tau^\imath}$:
\begin{align}
        &\mathbb{S}^\imath_{i,j}(w_1,w_2|z)=0, \text{ if $i\neq\tau i\neq j$},\label{iBSerre-1}\\
        &\mathbb{S}^\imath_{i,\tau i}(w_1,w_2|z)=v_i[2]_i\sym_{w_1,w_2}\frac{\Delta(zw_1)(1-v_izw_1^{-1})}{(1-v_izw_2^{-1})(1-v_i^2w_2w_1^{-1})}h_{\tau i}\ast\psi_{\tau i}(z)\ast \theta_i(w_2)\notag\\
        &+v_i^{-2}[2]_i\sym_{w_1,w_2}\frac{\Delta(zw_2)(1-v_iz^{-1}w_2)}{(1-v_i^{-1}z^{-1}w_1)(1-v_i^{-2}w_2w_1^{-1})}h_i\ast\psi_i(w_2)\ast \theta_i(w_1)\label{iBSerre-2}\\
        &\quad\quad\quad \text{ if $c_{i,\tau i}=-1$}.\notag
        \end{align}
        
     For the case $i=\tau i\neq j$, we have the following.
     If $c_{ij}=-1$, then 
        \begin{align}
        &\mathbb{S}^\imath_{i,j}(w_1,w_2|z)=-v_i\sym_{w_1,w_2}\frac{\Delta(w_1w_2)(1-v_i^{-2}w_2w_1^{-1})}{(1-v_iw_2z^{-1})(1-v_i^{-1}zw_1^{-1})}h_i\ast\psi_i(w_2)\ast \theta_j(z),\label{iBSerre-3}
        \end{align}
        If $c_{ij}=-2$, then 
        \begin{align}
        &\mathbb{S}^\imath_{i,j}(w_1,w_2,w_3|z)\notag\\
        &=\sym_{w_1,w_2,w_3}\frac{\Delta(w_1w_2)}{v_i-v_i^{-1}}h_i\ast\Big(A_1\cdot \psi_i(w_1)\ast \theta_i(w_3)\ast \theta_j(z)+[3]_i\theta_i(w_3)\ast \theta_j(z)\ast\psi_i(w_1)\Big)\notag \\
        &+\sym_{w_1,w_2,w_3}\frac{\Delta(w_1w_2)}{v_i-v_i^{-1}}h_i\ast\Big(A_2\cdot\psi_i(w_2)\ast \theta_j(z)\ast \theta_i(w_3)-\theta_j(z)\ast \theta_i(w_3)\ast\psi_i(w_2)\Big),\label{iBSerre-4}
        \end{align}
        where 
        \begin{align*}
A_1&=1-v_i^2[2]_i^2g_{ii}(w_1w_3^{-1})^{-1}\frac{(1-w_3w_2^{-1})(1-v_i^{-4}w_1z^{-1})}{(1-v_i^2w_3w_2^{-1})(1-v_i^2w_1z^{-1})},\\
A_2&=v_i^2[2]_i^2g_{ij}(w_1z^{-1})^{-1}\frac{(1-w_2w_3^{-1})(1-v_i^{-4}zw_1^{-1})}{(1-v_i^{-2}w_2w_3^{-1})(1-v_i^{-2}w_2w_3^{-1})}-[3]_i.
\end{align*}
These formulas first appeared in \cite{LW21b,Z22}, while the formulations presented here are more concise.

If $c_{ij}=-3$, then 
        \begin{align}
        &\mathbb{S}^\imath_{i,j}(w_1,w_2,w_3,w_4|z)\hspace{4cm}\notag\\
        &=\sym_{w_1,w_2,w_3,w_4}\frac{\Delta(w_1w_2)}{v_i-v_i^{-1}}h_i\ast\Big(D_1\cdot\psi_i(w_2)\ast \theta_i(w_3)\ast \theta_i(w_4)\ast \theta_j(z)\notag\\
        &+\qbinom{4}{2}_i\theta_i(w_3)\ast \theta_i(w_4)\ast \theta_j(z)\ast\psi_i(w_2)\Big)\notag\\
        &+\sym_{w_1,w_2,w_3,w_4}\frac{\Delta(w_1w_2)}{v_i-v_i^{-1}}h_i\ast\Big(D_2\cdot\psi_i(w_2)\ast \theta_i(w_3)\ast \theta_j(z)\ast \theta_i(w_4)\notag\\
        &\hspace{8cm}-[4]_i\theta_i(w_3)\ast \theta_j(z)\ast \theta_i(w_4)\ast\psi_i(w_2)\Big)\notag\\
        &+\sym_{w_1,w_2,w_3,w_4}\frac{\Delta(w_1w_2)}{v_i-v_i^{-1}}h_i\ast\Big(D_3\cdot\psi_i(w_2)\ast \theta_j(z)\ast \theta_i(w_3) \ast \theta_i(w_4)\notag\\
        &\hspace{8cm}+\theta_i(w_3)\ast \theta_j(z)\ast \theta_i(w_4)\ast\psi_i(w_2)\Big)\notag\\
        &+\sym_{w_1,w_2,w_3,w_4}\frac{\Delta(w_1w_2)}{v_i-v_i^{-1}}h_i^{\ast 2}\ast\Big(D_4\cdot \psi_i(w_2)\ast\psi_i(w_4)\ast \theta_j(z)+D_5\cdot \theta_j(z)\ast\psi_i(w_4)\Big),\label{iBSerre-5}
        \end{align}
        where 
\begin{align*}
D_1&=-[3]_ig_{ii}(w_2w_3^{-1})^{-1}g_{ii}(w_2w_4^{-1})^{-1}\frac{1-v_i^{-5}w_2z^{-1}}{1-v_i^3w_2z^{-1}}\Big(v_i^3[2]_ig_{ii}(w_3w_1^{-1})\frac{1-w_4w_1^{-1}}{1-v_i^2w_4w_1^{-1}}+1\Big)\\
&\qquad+(1+v_i^{-2})g_{ii}(w_2w_3^{-1})^{-1}\frac{1-w_3w_1^{-1}}{1-v_i^2w_3w_1^{-1}}+1,\\
D_2&=-[4]_i[2]_ig_{ii}(w_2w_3^{-1})^{-1}\Big(g_{ij}(w_2z^{-1})^{-1}\frac{1-w_2w_4^{-1}}{1-v_i^{-2}w_2w_4^{-1}}\Big(v_i^{-1}g_{ii}(w_3w_1^{-1})^{-1}g_{ij}(zw_1^{-1})\\
&\qquad-v_i[3]_i\frac{1-w_3w_1^{-1}}{1-v_i^2w_3w_1^{-1}}\Big)+v_i^{-1}\frac{1-w_3w_1^{-1}}{1-v_i^2 w_3w_1^{-1}}\Big)-[4]_i,\\
D_3&=-[3]_ig_{ij}(w_2z^{-1})^{-1}\frac{1-v_izw_1^{-1}}{1-v_i^{-3}zw_1^{-1}}\Big(v_i^3[2]_i\frac{1-w_2w_3^{-1}}{1-v_i^{-2}w_2w_3^{-1}}+g_{ii}(w_2w_3^{-1})^{-1}g_{ii}(w_2w_4^{-1})^{-1}\Big)\\
&\qquad+v_i^{-1}[2]_ig_{ii}(w_2w_3^{-1})^{-1}g_{ii}(w_3w_1^{-1})^{-1}g_{ij}(w_2z^{-1})^{-1}g_{ij}(zw_1^{-1})\frac{1-w_2w_3^{-1}}{1-v_i^{-2}w_2w_3^{-1}}+\qbinom{4}{2}_i,\\
D_4&=-[4]_iv_i^3g_{ij}(w_4z^{-1})^{-1}g_{ij}(zw_3^{-1})g_{ij}(w_2z^{-1})^{-1}\Big(1+v_i^{-3}[2]_ig_{ii}(w_2w_3^{-1})^{-1}\frac{1-w_2w_4^{-1}}{1-v_i^{-2}w_2w_4^{-1}}\Big)\\
&-[3]_i\Big(v_i^4+v_i[2]_ig_{ii}(w_2w_3^{-1})^{-1}g_{ii}(w_2w_4^{-1})^{-1}\frac{1-w_3w_1^{-1}}{1-v_i^2w_3w_1^{-1}}\Big)\frac{1-v_i^{-5}w_2z^{-1}}{1-v_i^3w_2z^{-1}}\\
&+v_i^3[2]_i\qbinom{4}{2}_ig_{ii}(w_2w_3^{-1})^{-1}g_{ii}(w_2w_4^{-1})^{-1}g_{ij}(w_2z^{-1})^{-1}g_{ij}(w_4z^{-1})^{-1}\frac{1-w_3w_2^{-1}}{1-v_i^2w_3w_2^{-1}}\\
&+\qbinom{4}{2}_ig_{ij}(w_4z^{-1})^{-1}g_{ij}(zw_3^{-1}),\\
D_5&=-v_i^{-3}[2]_ig_{ii}(w_2w_3^{-1})^{-1}g_{ii}(w_2w_4^{-1})^{-1}\frac{1-w_3w_1^{-1}}{1-v_i^2w_3w_1^{-1}}-1.
\end{align*}
These relations can be checked directly, but \eqref{iBSerre-4}--\eqref{iBSerre-5} require long computations, so we omit the details.

\subsection{iQuantum affinization of diagonal type}

The iHopf algebras of diagonal type $({^{\textup{\texttt{Dr}}}\tB}\otimes{^{\textup{\texttt{Dr}}}\tB} )^\imath$ is defined on the same vector space as ${^{\textup{\texttt{Dr}}}\tB}\otimes{^{\textup{\texttt{Dr}}}\tB}$, equipped with a new multiplication:
\[(a\otimes b)\ast (c\otimes d)=\sum\langle a_{(1)},d_{(2)}\rangle\cdot \langle c_{(2)},b_{(1)}\rangle\cdot a_{(2)}c_{(1)}\otimes b_{(2)}d_{(1)},\qquad\forall a,b,c,d\in {^{\textup{\texttt{Dr}}}\tB}.
\]

By \cite[Proposition 4]{G07}, we can define the Drinfeld double $D({^{\textup{\texttt{Dr}}}\tB})$ of ${^{\textup{\texttt{Dr}}}\tB}$ with respect to the weak Hopf pairing $\langle-,-\rangle$. More explicitly, it is the associative algebra defined on ${^{\textup{\texttt{Dr}}}\tB}\otimes {^{\textup{\texttt{Dr}}}\tB}$ subject to the following relations for any $a,b,a',b'\in {^{\textup{\texttt{Dr}}}\tB}$:
\begin{itemize}
    \item[(D1)]  $(a\otimes 1)(a'\otimes1)=aa'\otimes1$;
    \item[(D2)] $(1\otimes b)(1\otimes b')=1\otimes bb'$;
    \item[(D3)] $(a\otimes 1)(1\otimes b)=a\otimes b$;
    \item[(D4)] $\sum \langle a_{(1)},b_{(2)}\rangle\cdot(1\otimes b_{(1)})(a_{(2)}\otimes 1)=\sum \langle a_{(2)},b_{(1)}\rangle \cdot (a_{(1)}\otimes b_{(2)})$.
\end{itemize}
See \cite[Remark 3.2.2]{Jo95}.

\begin{lemma}[{\cite[Proposition 2.8]{CLPRWa}}]
    There exists an algebra isomorphism 
    \begin{align*}
        \Phi^{\textup{\texttt{D}}}:D({^{\textup{\texttt{Dr}}}\tB})\longrightarrow ({^{\textup{\texttt{Dr}}}\tB}\otimes {^{\textup{\texttt{Dr}}}\tB})^\imath,\quad (a\otimes b)\mapsto (a\otimes 1)\ast(1\otimes b),\quad\forall a,b\in {^{\textup{\texttt{Dr}}}\tB}.
    \end{align*}
\end{lemma}

\begin{remark}
    For a Hopf algebra $H$ with a Hopf pairing $\langle-,-\rangle$, its Drinfeld double $D(H)$ can be made into a Hopf algebra; see \cite[Proposition 2.7]{CLPRWa}. In this case, the morphism $\Phi^{\textup{\texttt{D}}}$ is a Hopf algebra homomorphism. However, for ${^{\textup{\texttt{Dr}}}\tB}$ with the weak Hopf pairing, we can also consider the completion of $D({^{\textup{\texttt{Dr}}}\tB})$, which is a topological Hopf algebra. Then the morphism $\Phi^{\textup{\texttt{D}}}$ is a Hopf algebra homomorphism.
\end{remark}

\begin{proposition}\label{prop: DrB double iso}
    There is an algebra isomorphism 
    \begin{align}
    &\widetilde{\Phi}_{\sharp}:{^{\textup{\texttt{Dr}}}\tU}\longrightarrow ({^{\textup{\texttt{Dr}}}\tB}\otimes{^{\textup{\texttt{Dr}}}\tB} )^\imath
        \\
        K_i\mapsto &h_i\otimes 1,\quad K_i'\mapsto 1\otimes h_i,\quad C\mapsto C\otimes1,
        \\
     C'\mapsto &1\otimes C,\quad X_{i,l}^+\mapsto \theta_{i,l}\otimes 1,\quad  X_{i,l}^-\mapsto 1\otimes \theta_{i,-l},
     \\
     &\psi_{i,m}^+\mapsto \psi_{i,m}\otimes1,\quad \psi_{i,-m}^-\mapsto 1\otimes\psi_{i,m}.
    \end{align}
\end{proposition}

\begin{proof}
    From \cite[Proposition 8]{G07}, we know that there is an algebra isomorphism $\tUD\rightarrow D({^{\textup{\texttt{Dr}}}\tB})$. By combining it with $\Phi^{\texttt{D}}$, the desired map $\widetilde{\Phi}_{\sharp}$ follows.   
\end{proof}

\subsection{Coideal subalgebra structure of {\rm i}quantum affinization}
The $\tau$-twisted compatible linear maps given in Definition \ref{def:chi} can also be extended to topological Hopf algebras with weak Hopf pairings.
Then we have the following.

\begin{lemma}
There exists a unique $\tau$-twisted compatible linear map ${^{\textup{\texttt{Dr}}}\chi}:{^{\textup{\texttt{Dr}}}\tB}\to\Q(v^{\frac12})$ such that
  \begin{align}
\label{eq:chi1}
    &{^{\textup{\texttt{Dr}}}\chi}(h_i)=v^{(\alpha_i,\alpha_{\tau i})}, \quad {^{\textup{\texttt{Dr}}}\chi}(C)=1,\quad {^{\textup{\texttt{Dr}}}\chi}(\theta_{i,r})=0,\quad
    \\
    \label{eq:chi2}
    &{^{\textup{\texttt{Dr}}}\chi}(\psi_{i,m})=\begin{cases}
    1&\text{if $m=0$},\\
    0 &\text{ if }2\nmid m,
    \\
    (v^{\frac{m}{2}c_{i,\tau i}}-v_i^{(\frac{m}{2}-1)c_{i,\tau i}})& \text{ if $m>0$ and $2\mid m$}.
\end{cases}    
\end{align}
\end{lemma}

\begin{proof}
Based on \cite[Lemma 2.11]{CLPRWa}, we can define a $\tau$-twisted compatible linear map on ${^{\textup{\texttt{Dr}}}\tB''}$, which is the algebra freely generated by $\theta_{i,r}$, $\psi_{i,m}, h_i^{\pm1},C^{\pm 1}$ for $i\in\II$, $r\in\Z$, $m\geq1$.

First, we prove that ${^{\textup{\texttt{Dr}}}\chi}$ descends to ${^{\textup{\texttt{Dr}}}\tB'}$. To see this, we shall check that ${^{\textup{\texttt{Dr}}}\chi}$ vanishes on the ideal $\mathscr{R}$ generated by \eqref{eq:B1}--\eqref{eq:B2}. Note that for $C$, from the definition of $\langle-,-\rangle$ we get ${^{\textup{\texttt{Dr}}}\chi}(aC)={^{\textup{\texttt{Dr}}}\chi}(a){^{\textup{\texttt{Dr}}}\chi}(C)={^{\textup{\texttt{Dr}}}\chi}(Ca)$ for any $a\in{^{\textup{\texttt{Dr}}}\tB''}$, so ${^{\textup{\texttt{Dr}}}\chi}$ vanishes on $aC-Ca$. Also, since the following identities hold in ${^{\textup{\texttt{Dr}}}\tB''}$:
\begin{align*}
\Delta([h_i,h_j])=[h_i,h_j]\otimes h_ih_j+h_jh_i\otimes [h_i,h_j],\\
\Delta([h_i,\psi_j(z)])=[h_i,\psi_j(z)]\otimes h_i\psi_{j}(zC_1)+\psi_j(z)h_i\otimes [h_i,\psi_j(zC_1)],\\
\Delta([\psi_i(z),\psi_j(w)])= 
[\psi_i(z),\psi_j(w)]\otimes  \psi_i(zC_1)\psi_j(wC_1)+\psi_j(w)\psi_i(z)\otimes [\psi_i(zC_1),\psi_j(wC_1)].
\end{align*}
Thus ${^{\textup{\texttt{Dr}}}\chi}$ descends to the quotient of ${}^{\texttt{Dr}}\widetilde{\mathbf B}'$
by \eqref{eq:B1}.

Furthermore, in this quotient algebra, we have 
\begin{align*}
\Delta([h_i,\theta_j(z)]_{v_i^{c_{ij}}})=[h_i,\theta_j(z)]_{v_i^{c_{ij}}}\otimes h_i+h_ih_j\psi_j(z)\otimes [h_i,\theta_j(zC_1)]_{v_i^{c_{ij}}},\\
\Delta(\psi_i(z)\theta_j(w)-g_{ij}(zw^{-1}) \theta_j(w) \psi_i(z))=(\psi_i(z)\theta_j(w)-g_{ij}(zw^{-1}) \theta_j(w) \psi_i(z))\otimes \psi_i(zC_1)
\\
+h_j\psi_i(z)\psi_j(w)\otimes(\psi_i(zC_1)\theta_j(wC_1)-g_{ij}(zw^{-1}) \theta_j(wC_1) \psi_i(zC_1)),
\end{align*}
it is immediate to see that we only need to verify for the generators of $\mathscr{R}$, which is straightforward from the definition of ${^{\textup{\texttt{Dr}}}\chi}$. 

Let us check \eqref{eq:B3}.  For convenience, set
\[
R_{ij}(z,w)
:=
(w-v_i^{c_{ij}}z)\theta_i(z)\theta_j(w)
-(v_i^{c_{ij}}w-z)\theta_j(w)\theta_i(z).
\]
If $j\neq \tau i$, then the $\mathbb Z^{\I_0}$-gradedness of the weak Hopf
pairing, together with ${^{\textup{\texttt{Dr}}}\chi}(\theta_{k,r})=0$, implies
immediately that
\[
{^{\textup{\texttt{Dr}}}\chi}\bigl(\theta_i(z)\theta_j(w)\bigr)
=
{^{\textup{\texttt{Dr}}}\chi}\bigl(\theta_j(w)\theta_i(z)\bigr)
=0.
\]
Hence ${^{\textup{\texttt{Dr}}}\chi}(R_{ij}(z,w))=0$.

It remains to prove the case $j=\tau i$.
By \ref{eq:chi2}, we have
\[
{^{\textup{\texttt{Dr}}}\chi}(\psi_i(u))
=
1+\sum_{m\ge1}{^{\textup{\texttt{Dr}}}\chi}(\psi_{i,m})u^m
=
\frac{1-u^2}{1-v_i^{c_{i,\tau i}} u^2}.
\]
Using the $\tau$-twisted compatibility of ${^{\textup{\texttt{Dr}}}\chi}$, the
coproduct formula
\[
\Delta(\theta_i(z))
=
\theta_i(z)\otimes 1
+h_i\psi_i(z)\otimes \theta_i(zC_1),
\]
and the weak Hopf pairing
\[
\langle \theta_{k,r},\theta_{\ell,s}\rangle
=
\frac{\delta_{k\ell}\delta_{rs}}{v_k-v_k^{-1}},
\]
we obtain
\[
{^{\textup{\texttt{Dr}}}\chi}\bigl(\theta_i(z)\theta_{\tau i}(w)\bigr)
=
\frac{v_i^{c_{i,\tau i}}}{v_i-v_i^{-1}}\,
\frac{1-z^2}{1-v_i^{c_{i,\tau i}} z^2}\,
\delta(zw),
\]
and, similarly,
\[
{^{\textup{\texttt{Dr}}}\chi}\bigl(\theta_{\tau i}(w)\theta_i(z)\bigr)
=
\frac{v_i^{c_{i,\tau i}}}{v_i-v_i^{-1}}\,
\frac{1-w^2}{1-v_i^{c_{i,\tau i}} w^2}\,
\delta(zw).
\]
Therefore
\[
\begin{aligned}
{^{\textup{\texttt{Dr}}}\chi}\bigl(R_{i,\tau i}(z,w)\bigr)
&=
\frac{v_i^{c_{i,\tau i}}}{v_i-v_i^{-1}}
\left[
(w-v_i^{c_{i,\tau i}} z)\frac{1-z^2}{1-v_i^{c_{i,\tau i}} z^2}
-(v_i^{c_{i,\tau i}} w-z)\frac{1-w^2}{1-v_i^{c_{i,\tau i}} w^2}
\right]\delta(zw).
\end{aligned}
\]
The rational function in brackets satisfies
\[
\begin{aligned}
&(w-v_i^{c_{i,\tau i}} z)\frac{1-z^2}{1-v_i^{c_{i,\tau i}} z^2}
-(v_i^{c_{i,\tau i}} w-z)\frac{1-w^2}{1-v_i^{c_{i,\tau i}} w^2}\\
&\qquad
=
-\frac{(v_i^{c_{i,\tau i}}-1)(w+z)(zw-1)(v_i^{c_{i,\tau i}} zw-1)}
{(v_i^{c_{i,\tau i}} w^2-1)(v_i^{c_{i,\tau i}}z^2-1)}.
\end{aligned}
\]
Since
\[
(zw-1)\delta(zw)=0,
\]
we conclude that
\[
{^{\textup{\texttt{Dr}}}\chi}\bigl(R_{i,\tau i}(z,w)\bigr)=0.
\]
Thus, ${^{\textup{\texttt{Dr}}}\chi}$ vanishes on the defining relation \eqref{eq:B3}.

It remains to verify that ${^{\textup{\texttt{Dr}}}\chi}$ vanishes on the ideal generated by the Serre relation \eqref{eq:B4}. Let $u_{ij}$ be the LHS of \eqref{eq:B4}. Then by making use of Lemma~\ref{coproduct of Serre}, it suffices to check on the generator $u_{ij}$. For integers $0\leq r,s\leq 1-c_{ij}$ such that $r+s=1-c_{ij}$, let us write
\[a:=\theta_i(w_1)\cdots \theta_i(w_r),\quad b:=\theta_j(z),\quad c:=\theta_i(u_1)\cdots \theta_i(u_s).\]
Then we have
\begin{align*}
{^{\textup{\texttt{Dr}}}\chi}(abc)&=\sum{^{\textup{\texttt{Dr}}}\chi}(a_{(1)}){^{\textup{\texttt{Dr}}}\chi}(c_{(3)})\langle a_{(2)},\theta_{\tau j}(z)\rangle\langle a_{(3)},\tau c_{(1)}\rangle\langle 1,\tau c_{(2)}\rangle\\
&+{^{\textup{\texttt{Dr}}}\chi}(h_i){^{\textup{\texttt{Dr}}}\chi}(\psi_i(z))\sum{^{\textup{\texttt{Dr}}}\chi}(a_{(1)}){^{\textup{\texttt{Dr}}}\chi}(c_{(3)})\langle\tau a_{(2)},h_i\psi_i(z)\rangle\langle\tau a_{(3)},c_{(1)}\rangle\langle \theta_{\tau j}(z),c_{(2)}\rangle\\
&=\sum{^{\textup{\texttt{Dr}}}\chi}(a_{(1)}){^{\textup{\texttt{Dr}}}\chi}(c_{(2)})\langle a_{(2)},\theta_{\tau j}(z)\rangle\langle a_{(3)},\tau c_{(1)}\rangle\\
&+{^{\textup{\texttt{Dr}}}\chi}(h_i){^{\textup{\texttt{Dr}}}\chi}(\psi_i(z))\sum{^{\textup{\texttt{Dr}}}\chi}(a_{(1)}){^{\textup{\texttt{Dr}}}\chi}(c_{(3)})\langle h_{a_{(2)}},h_i\psi_i(z)\rangle\langle \tau a_{(2)},c_{(1)}\rangle\langle \theta_{\tau j}(z),c_{(2)}\rangle
\end{align*}
where $h_{a_{(2)}}$ is some product of $h_i$'s and $\psi_i$'s depending on $a_{(2)}$. Note that for the RHS to be nonzero, we must have $\tau j=i$, so let us assume this. Note that $\tau i\neq i$ under this assumption, and $a_{(i)}$ (resp. $c_{(j)})$) has weights (see \eqref{eq:wt}) lying in $\N\alpha_i$ (resp. $\N\alpha_{\tau i}$), so we can have $\langle a_{(3)},\tau c_{(1)}\rangle\neq0$ only if their weights are both zero. The same argument applies to $\langle\tau a_{(2)},c_{(1)}\rangle$ in the second term, and we finally get
\begin{align*}
{^{\textup{\texttt{Dr}}}\chi}(abc)=&\sum{^{\textup{\texttt{Dr}}}\chi}(a_{(1)}){^{\textup{\texttt{Dr}}}\chi}(c)\langle\tau a_{(2)},\theta_j(z)\rangle\\
&+{^{\textup{\texttt{Dr}}}\chi}(h_i){^{\textup{\texttt{Dr}}}\chi}(\psi_i(z))\sum{^{\textup{\texttt{Dr}}}\chi}(a){^{\textup{\texttt{Dr}}}\chi}(c_{(2)})\langle \theta_{\tau j}(z),c_{(1)}\rangle.
\end{align*}
Since $i\neq\tau i$, we have ${^{\textup{\texttt{Dr}}}\chi}(a)=0$ if $r>0$ from the definition of ${^{\textup{\texttt{Dr}}}\chi}$ (and similar for ${^{\textup{\texttt{Dr}}}\chi}(c)$), so we conclude that ${^{\textup{\texttt{Dr}}}\chi}(abc)=0$ unless $c_{ij}=0$. However, if $c_{ij}=0$ then from the equations (recall that $j=\tau i$)
\begin{align*}
{^{\textup{\texttt{Dr}}}\chi}(\theta_i(w)\theta_{\tau i}(z))&=\frac{{^{\textup{\texttt{Dr}}}\chi}(h_i)\delta(wz)}{v_i-v_i^{-1}}{^{\textup{\texttt{Dr}}}\chi}(\psi_i(w))=\frac{\delta(wz)}{v_i-v_i^{-1}},\\
{^{\textup{\texttt{Dr}}}\chi}(\theta_{\tau i}(z)\theta_i(w))&=\frac{{^{\textup{\texttt{Dr}}}\chi}(h_{\tau i})\delta(wz)}{v_i-v_i^{-1}}{^{\textup{\texttt{Dr}}}\chi}(\psi_{\tau i}(z))=\frac{\delta(wz)}{v_i-v_i^{-1}},
\end{align*}
we see that ${^{\textup{\texttt{Dr}}}\chi}(u_{ij})=0$ in any case, which completes the proof.
\end{proof}

\begin{proposition}\label{prop:DrB iHopf embedding}
Let ${^{\textup{\texttt{Dr}}}\widetilde{\xi}_\tau}:{^{\textup{\texttt{Dr}}}\tB_\tau^\imath}\to ({^{\textup{\texttt{Dr}}}\tB}\otimes {^{\textup{\texttt{Dr}}}\tB})^{\imath,\textup{\texttt{T}}}$ be the embedding defined by
\[{^{\textup{\texttt{Dr}}}\widetilde{\xi}_\tau}(x)=\sum{^{\textup{\texttt{Dr}}}\chi}(x_{(2)})\tau x_{(3)}\otimes x_{(1)}.\]
Then we have 
\begin{align}
&{^{\textup{\texttt{Dr}}}\widetilde{\xi}_\tau}(h_i)=(h_{\tau i}\otimes 1)\ast(1\otimes h_i),\quad {^{\textup{\texttt{Dr}}}\widetilde{\xi}_\tau}(C)=C\otimes C,\\
&{^{\textup{\texttt{Dr}}}\widetilde{\xi}_\tau}(\psi_i(z))=\frac{1-z^2C_2^2}{1-v_i^{-c_{i,\tau i}}z^2C_2^2}(\psi_{\tau i}(zC_2)\otimes 1)\ast(1\otimes\psi_i(z)),\\
&{^{\textup{\texttt{Dr}}}\widetilde{\xi}_\tau}(\theta_i(z))=1\otimes \theta_i(z)+\frac{1-z^2C_2^2}{1-v_i^{-c_{i,\tau i}}z^2C_2^2}(\theta_{\tau i}(zC_2)\otimes 1)\ast(1\otimes h_i\psi_i(z)).
\end{align}
\end{proposition}
\begin{proof}
We note that 
\[\Delta^{(2)}(\psi_{i,m})=\sum_{s,t\geq 0}\psi_{i,m-s-t}C^{s+t}\otimes\psi_{i,s}C^t\otimes\psi_{i,t},\]
so
\begin{align*}
{^{\textup{\texttt{Dr}}}\widetilde{\xi}_\tau}(\psi_i(z))&=\sum_{m\geq 0}{^{\textup{\texttt{Dr}}}\widetilde{\xi}_\tau}(\psi_{i,m})z^m\\
&=\sum_{m\geq 0,s,t\geq 0}{^{\textup{\texttt{Dr}}}\chi}(\psi_{i,s})\psi_{\tau i,t}\otimes\psi_{i,m-s-t}C^{s+t}z^m\\
&=\frac{1-z^2C_2^2}{1-v_i^{c_{i,\tau i}}z^2C_2^2}\psi_{\tau i}(zC_2)\otimes\psi_i(z)\\
&=\frac{1-z^2C_2^2}{1-v_i^{-c_{i,\tau i}}z^2C_2^2}(\psi_{\tau i}(zC_2)\otimes 1)\ast(1\otimes\psi_i(z)).
\end{align*}

Similarly, from the identity
\[\Delta^{(2)}(\theta_{i,n})=\theta_{i,n}\otimes 1\otimes 1+\sum_{s\geq 0}h_iC^{n-s}\psi_{i,s}\otimes\theta_{i,n-s}\otimes 1+\sum_{s,t\geq 0}h_iC^{n-s}\psi_{i,s}\otimes h_iC^{n-s-t}\psi_{i,t}\otimes\theta_{i,n-s-t},\]
we obtain
\begin{align*}
{^{\textup{\texttt{Dr}}}\widetilde{\xi}_\tau}(\theta_i(z))&={^{\textup{\texttt{Dr}}}\widetilde{\xi}_\tau}(\sum_{n\in\Z}\theta_{i,n}z^n)\\
&=1\otimes \theta_i(z)+\sum_{s,t\geq 0,n\in\Z}{^{\textup{\texttt{Dr}}}\chi}(h_iC^{n-s-t}\psi_{i,t})\theta_{\tau i,n-s-t}\otimes h_iC^{n-s}\psi_{i,s}z^n\\
&=1\otimes \theta_i(z)+v_i^{c_{\tau i,i}}\frac{1-z^2C_2^2}{1-v_i^{c_{\tau i,i}}z^2C_2^2}\theta_{\tau i}(zC_2)\otimes h_i\psi_i(z)\\
&=1\otimes \theta_i(z)+\frac{1-z^2C_2^2}{1-v_i^{-c_{\tau i,i}}z^2C_2^2}(\theta_{\tau i}(zC_2)\otimes 1)\ast(1\otimes h_i\psi_i(z)).\qedhere
\end{align*}
\end{proof}

\begin{theorem}
\label{thm:subalgebra-Dr}
    There exists an injective algebra homomorphism
    \begin{align}
        &\mathcal{L}^{\textup{\texttt{Dr}}}:{}^\textup{\texttt{Dr}}\tUi\longrightarrow {}^{\textup{\texttt{Dr}}}\tU^{\textup{\texttt{T}}},
        \\
        &\K_i\mapsto v^{\frac{1}{2}(\alpha_i,\alpha_{\tau i})}K_iK_{\tau i}', \quad \ttC\mapsto CC'
        \\
        \label{eq:em-Thetai}
        &\Theta_i(z)\mapsto \frac{1-z^2C'^2}{1-v_i^{-c_{i,\tau i}}z^2C'^2}\psi_{\tau i}^+(zC')\psi_i^-(z^{-1}),\\
        &B_{i}(z)\mapsto X_i^-(z^{-1})+\frac{1-z^2C'^2}{1-v_i^{-c_{i,\tau i}}z^2C'^2}X_{\tau i}^+(zC')K_i'\psi_i^-(z^{-1}).
    \end{align}
\end{theorem}

\begin{proof}
This follows directly from Theorem~\ref{thm:iH=iQG} and Proposition~\ref{prop:DrB iHopf embedding} by using the isomorphism given in Proposition~\ref{prop: DrB double iso}.
\end{proof}

Theorem \ref{thm:subalgebra-Dr} enables us to define the iquantum loop algebra $\tUiD$ as a subalgebra of $\tUDT$.

\begin{remark}
For the iquantum affine algebra $\tUi$, that is, $\fg$ is a simple Lie algebra, by choosing suitable special parameters $\vs_i$ ($i\in\I$), we can make $\chi(\ttC)=1$; see Lemma \ref{lem:chi-C}. 
Notably, in this special case, the formula \eqref{eq:compatible-imagine} is compatible with \eqref{eq:em-Thetai}.
\end{remark}

We know $\tUi$ is a (right) coideal subalgebra of $\tU$ for Letzter's $\imath$quantum groups $\tUi$. By Theorem \ref{thm:subalgebra-Dr}, we can view the iquantum loop algebra ${}^\textup{\texttt{Dr}}\tUi$ as a subalgebra of the quantum loop algebra ${}^{\textup{\texttt{Dr}}}\tU^{\texttt{T}}$ via the embedding $\mathcal{L}^{\textup{\texttt{Dr}}}$. We now prove that ${}^\textup{\texttt{Dr}}\tUi$ is a right coideal subalgebra of ${}^{\textup{\texttt{Dr}}}\tU^{\texttt{T}}$.

\begin{proposition}
\label{prop:coideal}
    We have $\Delta(\tUiD)\subseteq {}^\textup{\texttt{Dr}}\tUi\widehat{\otimes} {}^\textup{\texttt{Dr}}\tU$. Moreover, 
    \begin{align}
        \Delta(\K_i)&=\K_i\otimes K_iK_{\tau i}',\quad \Delta(\ttC)=\ttC\otimes CC',
        \\
        \Delta(\Theta_i(z))&=\Theta_i(zC'_2)\otimes\psi_{\tau i}^+(z\mathcal{C}_1C'_2)\psi_i^-(z^{-1}),
        \\
        \Delta(B_i(z))&=1\otimes X_i^-(z^{-1})+B_i(zC'_2)\otimes K_i'\psi_i^-(z^{-1})\\\notag
        &\quad+v^{-\frac{1}{2}(\alpha_i,\alpha_{\tau i})}\K_{\tau i}\Theta_i(zC'_2)\otimes X_{\tau i}^+(z\mathcal{C}_1C'_2)K_i'\psi_i^-(z^{-1}),
    \end{align}
    where $\mathcal{C}_1=\mathcal{C}\otimes 1$ and $C'_2=1\otimes C'$.
\end{proposition}

\begin{proof}
This follows from a direct computation using Theorem~\ref{thm:subalgebra-Dr}. For example, if we write
\[\frac{1-z^2}{1-v_i^{-c_{i,\tau i}}z^2}=\sum_{n\geq 0}\gamma_{i,n}z^n\]
where $\gamma_{i,n}\in\Q(v^{\frac{1}{2}})$, then
\[\Theta_{i,m}=\sum_{s,t}\gamma_{i,s}\psi_{\tau i,t}^+\psi_{i,s+t-m}^-C'^{s+t},\]
so
\begin{align*}
\Delta(\Theta_{i,m})&=\sum_{s,t}\gamma_{i,s}\Delta(\psi_{\tau i,t}^+)\Delta(\psi_{i,s+t-m}^-)\Delta(C'^{s+t})\\
&=\sum_{s,t,a,b}\gamma_{i,s}\psi_{\tau i,a}\psi_{i,-b}^-C^{t-a}C'^{s+t}\otimes\psi_{\tau i,t-a}^+\psi_{i,s+t+b-m}^-C'^{s+t+b}\\
&=\sum_{s,t,a,b}\Theta_{i,s+a+b}\mathcal{C}^{t-a}\otimes\psi_{\tau i,t-a}^+\psi_{i,s+t+b-m}^-C'^{s+t+b},
\end{align*}
which gives the formula for $\Delta(\Theta_i(z))$.
\end{proof}



\begin{thebibliography}{CLW11a}
    \bibitem[AP23]{AP23} A. Appel, T. Prze\'zdziecki, {\em
		Generalized Schur-Weyl dualities for quantum affine symmetric pairs and orientifold KLR algebras}, Adv. Math. {\bf435} (2023), Paper No. 109383. 
        
         \bibitem[AV22]{AV22}	A. Appel and B. Vlaar, {\em Universal K-matrices for quantum Kac-Moody algebras}, Represent. Theory {\bf26} (2022), 764--824.


        \bibitem[BK19]{BK19}
M. Balagovic and S. Kolb, {\em Universal K-matrix for quantum symmetric pairs}, J. Reine
 Angew. Math. {\bf747} (2019), 299--353.		
		
		\bibitem[BW18a]{BW18a} H. Bao and W. Wang,
		{\em  A new approach to Kazhdan-Lusztig theory  of type $B$ via quantum symmetric pairs}, Ast\'erisque {\bf 402}, 2018, vii+134pp.
		
		\bibitem[BW18b]{BW18b} H. Bao and W. Wang,
		{\em Canonical bases arising from quantum symmetric pairs}, Inventiones Math. {\bf 213} (2018), 1099--1177.

     \bibitem[BB10]{BB10} P. Baseilhac, S. Belliard, {\em Generalized q-Onsager algebras and boundary affine Toda field theories},
Lett. Math. Phys. {\bf93} (2010), 213--228.

	\bibitem[Be94]{Be94} J. Beck,
		{\em Braid group actions and quantum affine algebras}, Commun. Math. Phys. {\bf 165} (1994), 555--568.
        \bibitem[BG17]{BG17a} A. Berenstein and J. Greenstein, {\em Double canonical bases}, Adv. Math. {\bf 316} (2017), 381--468.
        
        
			

      \bibitem[CH10]{CH10}  V. Chari and D. Hernandez, Beyond Kirillov-Reshetikhin modules, Contemp. Math. {\bf506}, 49-81,
Amer. Math. Soc., Providence, RI (2010).
		 \bibitem[CLPRW25]{CLPRWa} J. Chen, M. Lu, X. Pan, S. Ruan and W. Wang, 
{\em $\mathrm i$Quantum groups and $\mathrm i$Hopf algebras I: foundations}, \href{https://arxiv.org/abs/2511.11291}{arXiv:2511.11291}

      \bibitem[CLPRW26]{CLPRWb} J. Chen, M. Lu, X. Pan, S. Ruan and W. Wang, 
{\em $\mathrm i$Quantum groups and $\mathrm i$Hopf algebras II: dual canonical bases}, \href{https://arxiv.org/abs/2601.00524}{arXiv:2601.00524}

		\bibitem[CLW21]{CLW18} X. Chen, M. Lu and W. Wang,
		{\em A Serre presentation for the $\imath$quantum groups}, Transform. Groups {\bf26} (2021), 827--857.
				
		

\bibitem[Da98]{Da98}I. Damiani, {\em La R-matrice pour les alg\`ebres quantiques de type
affine non tordu}, Ann. Sci. \'Ec. Norm. Sup\'er. (4) {\bf31} (1998), 493–523.

\bibitem[Da12]{Da12}I. Damiani, {\em Drinfeld realization of affine quantum algebras: the relations}, Publ. Res. Inst. Math. Sci.
 {\bf48} (2012), 661--733.
 
\bibitem[Da15]{Da15} I. Damiani,
		{\em From the Drinfeld realization to the Drinfeld-Jimbo presentation of affine quantum algebras: injectivity}, Publ. Res. Inst. Math. Sci. {\bf 51} (2015), 131--171.
\bibitem[Da24]{Da24} I. Damiani, {\em On the Drinfeld coproduct}, Pure Appl. Math. Q. {\bf20} (2024), no. 1, 171--232.

\bibitem[DK19]{DK19} L. Dobson and S. Kolb,
{\em Factorisation of quasi K-matrices for quantum symmetric pairs}, Selecta Math. (N.S.) {\bf 25} (2019), Paper No. 63, 53 pp.		
\bibitem[Dr87]{Dr87} V.G. Drinfeld, {\em Quantum groups}, in: Proceedings of the International Congress of Mathematicians,
vols. 1, 2, Berkeley, Calif., 1986, Amer. Math. Soc., Providence, RI, 1987, pp. 798–820.
\bibitem[Dr88]{Dr88} V.G. Drinfeld, {\em A new realization of Yangians and quantum affine algebras}, Soviet. Math. Dokl. {\bf 36}  (2)(1988), 212--216. 

\bibitem[EKP07]{EKP07} B. Enriquez, S. Khoroshkin, S. Pakuliak, {\em  Weight functions and Drinfeld currents}, Comm. Math.
Phys. {\bf276} (2007), 691--725.

\bibitem[FL$^3$W20]{FLLLW20} Z. Fan, C. Lai, Y. Li, L. Luo, W. Wang, 
{\em Affine flag varieties and quantum symmetric pairs}, Mem. Amer. Math. Soc. {\bf 265} (2020), No. 1285.

\bibitem[FL$^3$WW20]{FLLLWW20} Z. Fan, C. Lai, Y. Li, L. Luo, W. Wang, H. Watanabe, {\em Quantum Schur duality of affine type C with three parameters}, Math. Res. Lett. {\bf 27} (2020), 79--114.

\bibitem[FR99]{FR99}  E. Frenkel, N. Reshetikhin, {\em The q-characters of representations of quantum affine algebras and deformations of W-algebras}, in: Recent Developments in Quantum Affine
Algebras and Related Topics, Contemp. Math. {\bf248} (1999), 163--205.

\bibitem[GKV95]{GKV95} V. Ginzburg, M. Kapranov and E. Vasserot, {\em Langlands reciprocity for algebraic
surfaces}, Math. Res. Lett. {\bf2} (1995), 147--160.

\bibitem[G07]{G07} P. Gross\'{e}, \emph{
On quantum shuﬄe and
quantum aﬃne algebras}, J. Algebra {\bf318} (2007), 495--519.

\bibitem[H05]{H05} D. Hernandez, {\em Representations of quantum affinizations and fusion product}, Transform.
Groups {\bf10} (2005), 163--200. 

\bibitem[H07]{H07} D. Hernandez, {\em  Drinfeld coproduct, quantum fusion tensor category and applications},
Proc. Lond. Math. Soc. {\bf95} (2007), 567--608.




\bibitem[H09]{H09} D. Hernandez,
{\em Quantum toroidal algebras and their representations}, Sel. Math. New Ser. {\bf14} (3-4) (2009), 701--725.


\bibitem[Jan96]{Jan96} J.C. Jantzen, Lectures on quantum groups, Graduate studies in mathematics, Vol. {\bf6}, AMS, 1995.

\bibitem[Jim85]{Jim85} M. Jimbo,
{\em A $q$-Difference analogue of $U(\mathfrak g)$ and the Yang-Baxter equation},
 Lett. Math. Phys. {\bf 10} (1985), 63--69.


\bibitem[Jin98]{Jin98} N. Jing, {\em Quantum Kac-Moody algebras and Vertex
representations}, Lett. Math. Phys. {\bf44} (1998), 261--271.


      
\bibitem[Jo95]{Jo95}  A. Joseph, Quantum Groups and Their Primitive Ideals, Ergebnisse der Math. (3), vol. {\bf29}, Springer-Verlag, Berlin, 1995.
		

		
		\bibitem[Ko14]{Ko14} S. Kolb,
		{\em Quantum symmetric Kac-Moody pairs}, Adv. Math. {\bf 267} (2014), 395--469.
		
\bibitem[KP11]{KP11} S. Kolb and J. Pellegrini,
{\em Braid group actions on coideal subalgebras of quantized enveloping algebras}, J. Algebra {\bf 336} (2011), 395--416.

        
        	
		\bibitem[Let99]{Let99} G. Letzter,
		{\em Symmetric pairs for quantized enveloping algebras}, J. Algebra {\bf 220} (1999), 729--767.
		
		
\bibitem[LiP25]{LiP25} J.-R. Li, T. Prze\'zdziecki, {\em Compatibility of Drinfeld presentations for split affine Kac-Moody quantum
 symmetric pairs}. Lett. Math. Phys. {\bf115} (2025), Paper No. 75.	


\bibitem[LiP26]{LiP26} J.-R. Li, T. Prze\'zdziecki,
{\em Compatibility of Drinfeld presentations and q-characters for affine Kac-Moody quantum symmetric pairs: quasi-split case}, \href{https://arxiv.org/abs/2601.02165}{arXiv:2601.02165}

\bibitem[KLWZ25a]{KLWZ25a} K. Lu, W. Wang, W. Zhang, {\em A Drinfeld type presentation of twisted Yangians}, Representation Theory {\bf29} (2025), 838--870.

\bibitem[KLWZ25b]{KLWZ25b} K. Lu, W. Wang, W. Zhang, {\em Affine $\imath$quantum groups and twisted Yangian in Drinfeld presentations}, Comm. Math. Phys. {\bf406} (2025), Paper No. 98.
\bibitem[LZ26]{LZ26} K. Lu,  W. Zhang, {\em
A Drinfeld type presentation of twisted Yangians of quasi-split type}, Commun. Contemp. Math. (to appear), 2026.





\bibitem[LPWZ25]{LPWZ25} M. Lu, X. Pan, W. Wang and W. Zhang,  {\em Braid group action and quasi-split affine iquantum groups III}, \href{https://arxiv.org/abs/2511.00882}{arXiv:2511.00882}.

  

        
 \bibitem[LR24]{LR24} M. Lu and S. Ruan,  {\em $\imath$Hall algebras of weighted projective lines and quantum symmetric pairs III: quasi-split type}, \href{https://arxiv.org/abs/2411.13078}{arXiv:2411.13078}
		

        \bibitem[LW21]{LW21b} M. Lu and W. Wang, {\em A Drinfeld type presentation of affine ıquantum groups I: split ADE type}, Adv. in Math. {\bf393} (2021), 108111, 46 pp.
		
	
		
		\bibitem[LW22a]{LW22a} M. Lu and W. Wang,
		{\em Hall algebras and quantum symmetric pairs I: foundations},  Proc. London Math. Soc. (3) {\bf124} (2022), 1--82. 
		
		\bibitem[LW22b]{LW22b} M. Lu and W. Wang,
		{\em Braid group symmetries on quasi-split $\imath$quantum groups via $\imath$Hall algebras},   Selecta Math.(N.S.) {\bf 28} (2022), Paper No. 84. 
		


		\bibitem[LWZ23]{LWZ23} M. Lu, W. Wang and W. Zhang,
		{\em Braid group action and quasi-split affine $\imath$quantum groups I},  Represent. Theory {\bf27} (2023), 1000--1040.

\bibitem[LWZ24]{LWZ24} M. Lu, W. Wang and W. Zhang, Braid group actions and quasi-split affine ıquantum groups II: higher ranks, Comm. Math. Phys. {\bf405} (2024): 142.



		
\bibitem[Lus93]{Lus93} G. Lusztig, Introduction to Quantum Groups, Birkh\"{a}user, Boston, 1993.

\bibitem[MRS03]{MRS03} A. Molev, E. Ragoucy, P. Sorba, {\em Coideal subalgebras in quantum affine algebras},
Rev. Math. Phys. {\bf15} (8) (2003), 789--822.
\bibitem[N01]{N01} H. Nakajima, {\em Quiver varieties and finite-dimensional representations of quantum affine algebras}, J. Amer. Math. Soc. {\bf14} (2001), 145--238.

\bibitem[P25]{P25} T. Prze\'zdziecki, {\em Drinfeld rational fractions for affine Kac-Moody quantum symmetric pairs}, Selecta Math.(N.S.) {\bf 31} (2025), Paper No. 87.

      

        
		
		
	
		

   \bibitem[Sk88]{Sk88} E.  Sklyanin, {\em Boundary conditions for integrable quantum systems}, J. Phys. A {\bf21}(2375)(1988), 2375--2389.
	
		\bibitem[SW24]{SW24} C. Su and W. Wang, {\em Affine iquantum groups and Steinberg varieties of type C}, Selecta Math.(N.S.) {\bf32} (2026), no. 3, Paper No. 50.
\bibitem[VV96]{VV96}
M. Varagnolo and E. Vasserot, {\em Schur duality in the toroidal setting}, Comm.
Math. Phys. {\bf182} (1996), 469–483.

        \bibitem[WZ23]{WZ23} W. Wang and W. Zhang,
	{\em An intrinsic approach to relative braid group symmetries on $\imath$quantum groups}, Proc. London Math. Soc. \textbf{127} (2023), 1338--1423. 
    
        \bibitem[Z22]{Z22} W. Zhang, {\em A Drinfeld type presentation of affine$\imath$quantum groups II: split BCFG type}, Lett. Math. Phys. \textbf{112} (2022), 89.
        
        \bibitem[Z23]{Z23} W. Zhang, {\em Relative braid group symmetries on $\imath$quantum groups of Kac-Moody type}, Selecta Math.(N.S.) \textbf{29} (2023), Paper No. 59, 54 pp.

	\end{thebibliography}
\end{document}